\documentclass[notitlepage]{amsart}

\newtheorem{pro}{Proposition}[section]
\newtheorem{teo}[pro]{Theorem}
\newtheorem{defi}[pro]{Definition}
\newtheorem{lem}[pro]{Lemma}
\newtheorem{cor}[pro]{Corollary}
\newtheorem{rk}[pro]{Remark}
\newtheorem{ex}[pro]{Example}

\newcommand{\Ext}{\mathrm{Ext}}
\newcommand{\Hom}{\mathrm{Hom}}

\newcommand{\A}{\mathcal{A}}
\newcommand{\B}{\mathcal{B}}
\newcommand{\I}{\mathcal{I}}

\newcommand{\F}{\mathcal{F}}
\newcommand{\G}{\mathcal{G}}
\newcommand{\Q}{\mathcal{P}}
\newcommand{\X}{\mathcal{X}}
\newcommand{\Y}{\mathcal{Y}}

\newcommand{\W}{\mathcal{W}}
\newcommand{\pd}{\mathrm{pd}}
\newcommand{\Gpd}{\mathrm{Gpd}}
\newcommand{\WGpd}{\mathrm{WGpd}}
\newcommand{\WGid}{\mathrm{WGid}}
\newcommand{\Dpd}{\mathrm{Dpd}}
\newcommand{\Proj}{\mathrm{Proj}}
\newcommand{\proj}{\mathrm{proj}}
\newcommand{\Inj}{\mathrm{Inj}}
\newcommand{\inj}{\mathrm{inj}}
\newcommand{\id}{\mathrm{id}}
\newcommand{\Gid}{\mathrm{Gid}}
\newcommand{\resdim}{\mathrm{resdim}}

\newcommand{\coresdim}{\mathrm{coresdim}}
\newcommand{\add}{\mathrm{add}}
\newcommand{\Add}{\mathrm{Add}}
\newcommand{\Prod}{\mathrm{Prod}}
\newcommand{\Flat}{\mathrm{Flat}}
\newcommand{\FP}{\mathrm{FP}}
\newcommand{\modu}{\mathrm{mod}}
\newcommand{\Modu}{\mathrm{Mod}}
\newcommand{\Ker}{\mathrm{Ker}}
\newcommand{\Ima}{\mathrm{Im}}
\newcommand{\GP}{\mathcal{GP}}
\newcommand{\FGPD}{\mathrm{FGPD}}
\newcommand{\FGID}{\mathrm{FGID}}

\newcommand{\glPD}{\mathrm{gl.PD}}

\newcommand{\glGPD}{\mathrm{gl.GPD}}
\newcommand{\glGID}{\mathrm{gl.GID}}
\newcommand{\glDPD}{\mathrm{gl.DPD}}
\newcommand{\FPD}{\mathrm{FPD}}
\newcommand{\FID}{\mathrm{FID}}
\newcommand{\WFGPD}{\mathrm{WFGPD}}
\newcommand{\glWGID}{\mathrm{gl.WGID}}
\newcommand{\glWGPD}{\mathrm{gl.WGPD}}
\newcommand{\WFGID}{\mathrm{WFGID}}
\newcommand{\DP}{\mathcal{DP}}
\newcommand{\FDPD}{\mathrm{FDPD}}

\newcommand{\GI}{\mathcal{GI}}
\newcommand{\Coker}{\mathrm{CoKer}}
\newenvironment{dem}{\noindent\bf Proof. \rm }{$\ \Box$}
\usepackage{latexsym,amssymb,amscd} 
\usepackage{amsmath}
\usepackage[all]{xypic}
\usepackage{amsthm}

\usepackage[Lenny]{fncychap}
\usepackage{amscd}

\begin{document}

\title[relative G-objects]{Relative Gorenstein  objects in abelian categories}
%\thanks{}
\author{Victor Becerril,  Octavio Mendoza and Valente Santiago}
%\date{}
\thanks{
The authors thanks the Project PAPIIT-Universidad Nacional Aut\'onoma de M\'exico IN103317.}

%%%%%%%%%%%%%%%%%%%%%%%%%%%%%%%%%%%%%%%%%%%%%%%%%%%%%%%%%%%%%%%%%%%%%%%%%%%%%%
\begin{abstract} 
Let $\A$ be an abelian category. For a  pair  $(\X,\Y)$ of classes of objects in $\A,$ we define the weak and the $(\X,\Y)$-Gorenstein relative 
projective objects in $\A.$ We point out that such objects generalize the usual  Gorenstein projective objects and others generalizations appearing in the literature as Ding-projective, Ding-injective, $\X$-Gorenstein projective, Gorenstein AC-projective and $G_C$-projective modules  and Cohen-Macaulay objects in abelian categories. We show that the principal results on Gorenstein projective modules remains true for the weak and the  $(\X,\Y)$-Gorenstein relative objects. Furthermore, by using Auslander-Buchweitz approximation theory, a relative   
version of Gorenstein homological dimension is developed.  Finally, we introduce the notion of $\W$-cotilting pair in the abelian category 
$\A,$ which is very strong connected with the cotorsion pairs related with relative Gorenstein objects in $\A.$ It  is  worth mentioning that the 
$\W$-cotilting pairs generalize the notion of cotilting objects in the sense of  L. Angeleri H\"ugel and F. Coelho \cite{AC}.
\end{abstract}  
\maketitle
%\centerline{}

\setcounter{tocdepth}{1}
\tableofcontents

\section{Introduction.}

In homological algebra, the injective and projective objects play an impotant role. In 1969 M. Auslander and M. Bridger intoduced the $G$-dimension \cite{AuBri} for the category of finitely generated modules over a commutative noetherian ring $R$. They proved the inequality $\mathrm{G}$-$\mathrm{dim}\,M\leq \pd\, M$ for every finitely generated $R$-module $M$; and moreover it was also shown that the equality holds when $\pd \,M$ is finite. The previous inequality was used in order to characterize  Gorenstein local rings  \cite[Section 3.2]{Aus0} and to give a proof of a generalization of the Auslander-Buchsbaum formula for the case of the $G$-dimension. In the early 1990's the notion of G-dimension was extended beyond the world of finitely generated modules over a noetherian ring. For any ring $R$ (associative with unit), Enochs and Jenda defined in  \cite{EJ0} the \textit{Gorenstein projective dimension} $\Gpd \,M$ for an arbitrary module $M$ and not just for the finitely generated ones.
Later on, L. L. Avramov, R. O. Buchweitz, A. Martsinkovsky and I. Reiten proved, in the unpublished paper \cite{ABMR},  that a finitely generated module $M,$ over a noetherian ring, is Gorenstein projective if and only if  $\mathrm{G}$-$\mathrm{dim}\,M=0.$ A proof of this fact can be found in  \cite[Theorem 4.2.6]{Chris}.

Recently, H. Holm  showed in \cite{Holm} that the class of Gorenstein modules, studied by Enochs and Jenda in \cite{EJ0}, is a resolving class; and that fact allow us to use relative homological algebra in the category $\Modu\,(R)$ of left $R$-modules. Inspired by that,  D. Bravo, J. Guillespie and M. Hovey defined the AC-Gorenstein projectives (injectives) in \cite{BGH}, and N. Ding, Y. Li and L. Mao defined the now called  Ding-projective modules \cite{DingLi}. Various authors have generalized these kinds of Gorenstein projective objects: for example,  D. Bennis \cite{BO}, followed by  M. Tamekkante \cite{Ta},  and also by F. Meng and  Q. Pan \cite{MP}. Other variations were given by  Q. Pan and F. Cai \cite{PC} and independently by  Y. Geng and N. Ding \cite{GD} and D. Bennis, J. R. Garc\'ia Rozas and L. Oyonarte \cite{BGRO}.

The aim of this work is to unify all the notions of Gorenstein objects, that there exists in the previous literature, in a given one which replace  all of them. For this, we consider a pair of classes of objects $(\X,\Y),$ with certain conditions in an abelian category 
$\mathcal{A},$  and  define the \textit{$(\X,\Y)$-Gorenstein projective objects} in $\A.$  Throughout the paper,  we develop the properties of this objects by using the Auslander-Buchweitz approximation theory \cite{AuB, BMPS}. We show that our results have as a corollary the results that were obtained previously in the papers mentioned above. Furthermore, we obtain certain relative cotorsion pairs, in the sense of \cite{BMPS}, and prove that they are related with the notion of $\W$-cotilting, which is a generalization of the tilting objects in the sense of Angeleri-Coelho \cite{AC}. We also develop the properties of relative homological dimensions associated to the relative $(\X,\Y)$-Gorenstein projective objects  and its relationship with other dimensions. Many of  the results we get in this paper are a generalization of classical well-known results from \cite{Holm}. 

This paper is organized as follows. In Section 2, we recall fundamental results of the Auslander-Buchweitz theory  developed in the seminal paper \cite{AuB}. We also introduce the notation given in \cite{BMPS} that will be used throughout this paper.

In Section 3, we define the principal object of study of this paper, namely, the { $(\X,\Y)$-Gorenstein projective} objects in an abelian category $\A.$ The class of all these objects is denote by 
$\GP_{(\X,\Y)}$  (see Definition \ref{defiGP}). It is shown in Theorem \ref{GP6} that, under mild conditions on the pair $(\X,\Y),$ the class $\GP_{(\X,\Y)}$ is left thick (i.e. it is closed under direct summands, extensions and kernels of epimorphisms  between its objects). We also define the class of all the \textit{$(\X,\Y)$-weak Gorenstein projectives}, which is denoted by $W\GP_{(\X, \Y)}$ (see Definition \ref{debildefi2}). Let $\omega\subseteq \Y\subseteq \A$ be classes of objects in $\A.$ We prove, in Theorem \ref{thickcat}, that  if 
$\omega$ closed under direct summands, then $W\GP_{(\omega,\Y)}$ is left thick. It is also  proven that, under certain conditions on the pair $(\X,\Y),$  the following equalities hold  true (see Theorem \ref{iguales}) 
$$W\GP _{(\X,\Y)} = W\GP _{(\X \cap \Y, \Y)} = \GP  _{(\X,\Y)}  = \GP^{2}_{(\X,\Y)} = W\GP ^2 _{(\X,\Y)},$$ 
where $\GP^{2}_{(\X,\Y)}:=\GP_{( \GP  _{(\X,\Y)},\Y)}$ and $W\GP^{2}_{(\X,\Y)}:=W\GP_{( W\GP  _{(\X,\Y)},\Y)}.$ As particular cases of the above equalities,  we get  \cite[Theorem 2.8]{Xu} and \cite[Proposition 2.16 (2)]{BGRO}. Finally, we close this section by giving sufficient conditions in order to construct  strong left Frobenius pairs from GP-admissible pairs in abelian categories. The importance of the strong left Frobenius pairs relies on the fact that they give us exact model structures on exact categories \cite[Section 4]{BMPS}.

In Section 4, we develop, in an unified way, the theory of the relative Gorenstein homological dimensions. We stablish relationships 
between different kinds of relative homological dimension, namely: (weak) relative Gorenstein projective, relative projective, finitistic and 
resolution dimensions. By taking different pairs $(\X,\Y)$ of classes of objects in an abelian category $\A,$ as an application of 
the obtained results, we get as a corollary the well known results. For example, it is proved in Theorem \ref{ThickGP5.6}, that under 
certain conditions on the pair $(\X,\Y),$ the finitistic  $(\X,\Y)$-Gorenstein projective dimension of $\A$ and the finitistic projective dimension of $\A$ coincides, which is a generalization of \cite[Proposition 2.17]{Holm}.

In Section 5, we introduce the notion of  $\W$-tilting and $\W$-cotilting pairs in an abelian category $\A.$ We show that there is a 
strong relationship between the weak-Gorenstein projective objects $W\GP _{(\omega, \Y)},$ obtained from  a WGP-admisible pair 
$(\omega, \Y)$  (see Definition \ref{WGPadmi}), and relative cotorsion pairs in the sense of \cite{BMPS}. In more detail, given a WGP-admisible  pair $(\omega,\Y),$ with $\omega$ closed under direct summands, the pair $(W\GP _{(\omega, \Y)}, \omega^\wedge)$ turns out to be a $W\GP ^{\wedge} _{(\omega, \Y)}$-relative cotorsion pair in the abelian category $\A$ (see Proposition \ref{cotorpair} ). Furthermore, we show that the equality $W\GP _{(\omega,\Y)} = {^{\perp}}\Y$ holds if and only if the   WGP-admissible pair $(\omega, \Y)$ is  $\W$-cotilting, for some $\W\subseteq \A$ (see Corollary \ref{GdebilCotil}). If the abelian category $\A$ has enough projectives and for the $\W$-cotilting pair $(\omega, \Y)$ we have that $\omega$ is closed under direct summands and $\id\,(\Y)<\infty,$ we get from Theorem \ref{tiltcotor} that $(W\GP _{(\omega,\Y)},\omega^\wedge)$ is a hereditary complete cotorsion pair in $\A.$ Moreover, the weak global $(\omega,\Y)$-Gorenstein proyective dimension of $\A$ coincide with different kinds of relative finitistict projective and resolution dimensions (see the details in Theorem \ref{tiltcotor} (c) and (e)). The corresponding results for $(\X,\Y)$-Gorenstein projective also hold true as can be seen in this section. 

In Section 6, we introduce the notion of tilting and cotilting objects in abelian categories, which is an extension of the definition given by 
Angeleri-Coehlo \cite{AC} for the setting of the abelian category $\Modu\,(R),$ for any ring $R.$ This definition of tilting (cotilting) 
object will be used throughout this section to be compared with the notion of $\W$-tilting ($\W$-cotilting). For example, 
from Corollary \ref{Cotil3}, we get that the notion of $\W$-cotilting pair is a strict generalization of cotilting object. \\
In this section, we were able to apply the results obtained  in the preceding section and thus we get several nice results for tilting and cotilting objects. Namely, in Theorem \ref{Cotil2}, we prove that for an AB4*-abelian category $\A,$ with injective cogenerators and enough projectives, if $(\X,\Y)$ is 
a hereditary complete cotorsion pair in $\A$ such that $\id(\Y)<\infty$ and $\omega:=\X\cap\Y$ is closed under products, then there is some cotilting object $M\in\A$ such that $\omega=\Prod(M),$ $\id(M)=\id(\Y),$ $\Y=\omega^\wedge$ and $W\GP_\omega=\X={}^\perp M.$ Finally, in Theorem \ref{Cotil4},  we consider $\Modu(R),$ for a ring $R$ which is left perfect, left noetherian and right coherent. We characterize in this case, when $R$ is a cotilting module in $\Modu\,(R)$ and we also give several relations between the different homological dimensions introduced in this paper.

\section{Auslander-Buchweitz approximation theory}
\

We start this section by collecting all the background material that will be necessary in the sequel. First, we
introduce some general notation. Next, we recall the notion
of relative projective dimension and resolution dimension of a given
class of  objects in an abelian category $\A.$ Finally, we also recall  definitions and basic properties  we need of Auslander-Buchweitz 
approximation theory. In all that follows, we are taking as a main reference the papers \cite{AuB} and \cite{BMPS}.
\

We remark that M. Auslander and R. O. Buchweitz assumed in \cite{AuB} that a given class $\X\subseteq\A$  is a resolving  and 
an additively closed subcategory, which is also  closed under direct summands in $\A.$  In a very carefully revision of their proofs in \cite{AuB}, it can be seen that some of the assumed  hypothesis are not used. In order to give nice applications of AB-approximation theory to the relative Gorenstein  theory, we give a review of such theory by putting in each statement the minimum needed hypothesis. Of course, these results also have dual versions which we will freely use in the sequel. For more details, we recommend the reader to see \cite{BMPS}.
\

Throughout the paper, $\A$ will be an abelian category and $\X\subseteq\A$ a class of objects of $\A.$ 
We denote by $\pd\,X$ the  {\bf{projective dimension}} of $X\in\A.$ Similarly, $\id\,X$ denotes the {\bf{injective dimension}} of $X\in\A.$ 
For any non-negative integer $n,$ we set 
$$\Q_n(\A):=\{X\in\A\;:\;\pd X\leq n\}.$$ 
In particular $\Proj(\A):=\Q_0(\A)$ is the class of all the projective objects in $\A.$ The classes  $\I_n(\A)$  and $\Inj(\A)$ are defined dually. 
\

Let now  $\X$ be a subclass of objects in $\A$. We denote by $\add\,(\X)$   the class of all  objects isomorphic to direct summands of finite direct sums 
of objects in $\mathcal{X}.$ Moreover, for each positive integer $i,$ we consider the right  orthogonal classes
$$\X^{\perp_i}:=\{M\in\A\;:\;\Ext^i_\A(-,M)|_{\X}=0\}\quad {\rm and}\quad  \X^\perp:=\cap_{i>0}\,\X^{\perp_i}.$$ Dually, we have the left orthogonal classes ${}^{\perp_i}\X$ and ${}^{\perp}\X.$
\ 

By following \cite{BMPS}, we recall the notions to be considered in the paper.  Let $\X\subseteq\A$ be a subclass of objects in $\A.$ It  is said that $\X$ is a {\bf{pre-resolving}} class if it is closed under extensions 
and kernels of epimorphisms between its objects. A pre-resolving class 
is said to be  {\bf{resolving}} if it contains $\Proj(\A).$  If the dual properties hold true, then we get  {\bf{pre-coresolving}} and 
 {\bf{coresolving}} subclasses of $\A.$ A {\bf{left thick}} (respectively, {\bf{right thick}}) class is a pre-resolving (respectively, pre-coresolving) 
 class which is closed under direct summands in $\A.$ A {\bf{ thick}} class is  both a right and left thick class.
 A {\bf{left saturated}} (respectively, {\bf{right saturated}}) class is a resolving (respectively, coresolving) 
 class which is closed under direct summands in $\A.$ A {\bf{saturated}} class is  both a right saturated and left saturated class.
For example, $\Proj(\A)$  and ${}^{\perp}\X$ are left saturated subclasses of $\A$, while $\Inj(\A)$ and $\X^\perp$ are right saturated 
 subclasses of $\A$.
\

\bigskip

{\sc Relative homological dimensions.} Given a class $\X\subseteq\A$ and $M\in\A,$ the {\bf{relative projective dimension}} of $M,$ with 
 respect to $\X,$ is defined as $$\pd_{\X}\,(M):=\min\{n\in\mathbb{N}\,:\,\Ext_\A^j(M,-)|_{\X}=0  \text{ for any } j>n\}.$$
 We set by definition that $\min\emptyset:=\infty.$   Dually, we 
  denote by $\mathrm{id}_{\X}\,(M)$ the  {\bf{relative injective dimension}} of
  $M,$ with respect to $\X.$ Furthermore, for any class $\Y\subseteq\A,$ we set $$\pd_\X\,(\Y):=\mathrm{sup}\,\{\pd_\X (Y)\;:\; Y\in\Y\}\text{ and }\id_\X\,(\Y):=\mathrm{sup}\,\{\id_\X(Y)\;:\; Y\in\Y\}.$$ 
 It can be shown that $\pd_\X\,(\Y)=\id_\Y\,(\X).$ If $\X=\A,$ we just write $\pd\,(\Y)$ and   $\id\,(\Y)$.
\

\bigskip

{\sc Resolution and coresolution dimension.}
Let $M\in\A$ and $\X$ be a class of objects in $\A.$ The $\X$-{\bf{coresolution dimension}} $\coresdim_\X\,(M)$ of $M$ is the minimal 
non-negative integer $n$ such that there is an exact sequence $$0\to M\to X_0\to X_1\to\cdots\to X_n\to 0$$ with $X_i\in\X$ for 
$0\leq i\leq n.$ If such $n$ does not exist, we set $\coresdim_\X\,(M):=\infty.$ Also, we denote by $\mathcal{X}^{\vee}$ the  class of objects 
in $\A$ having finite $\mathcal{X}$-coresolution. 

Dually, we have the $\X$-{\bf{resolution dimension}} $\resdim_\X\,(M)$ of $M,$ and the class $\mathcal{X}^{\wedge}$ of objects in 
$\A$ having finite $\mathcal{X}$-resolution. 

Given a class $\Y\subseteq\A,$ we set  
$$\coresdim_\X\,(\Y):=\mathrm{sup}\,\{\coresdim_\X\,(Y)\;
:\; Y\in\Y\},$$ and  $\resdim_\X\,(\Y)$ is defined dually.
\

\bigskip

{\sc Approximations.}  Let  $\X$ be a class of objects in $\A.$ A morphism $f:X\rightarrow M$ is  an $\X$-{\bf{precover}} if $X\in\X$ and $\mathrm{Hom}_\A(Z,f):\mathrm{Hom}_\A(Z,X)\rightarrow\mathrm{Hom}_\A(Z,M)$
is surjective for any $Z\in\X.$ Furthermore, an $\X$-precover $f:X\rightarrow M$ is {\bf{special}} if $\Coker\,(f)=0$ and $\Ker\,(f)\in\X^{\perp_1}.$ 
We will  freely use  the dual notion of (special) $\X$-{\bf{preenvelope}}.
\

Finally, we recall  the notion of cotorsion pair which was introduced by L. Salce in \cite{S}. It is the analog of a torsion pair where the functor 
$\Hom_\A(-,-)$ is replaced by $\Ext^1_\A(-,-).$

\begin{defi}\cite{S} Let $\X$ and $\Y$ be classes of objects in the abelian category $\A.$ The pair $(\X,\Y)$ is  a {\bf left cotorsion pair} (respectively,  a {\bf right cotorsion pair}) if 
$\X={}^{\perp_1}\Y$ (respectively $\X^{\perp_1}=\Y$). We say that $(\X,\Y)$ is a {\bf cotorsion pair} if it is both a left and right cotorsion pair. 
\end{defi}

The notion of relative cotorsion pair, as was introduced in \cite{BMPS}, will play an important role in this paper. For more details in the study of these cotorsion pairs, we recommend to the reader to see in \cite{BMPS}.

\begin{defi}\cite{BMPS} A $\mathcal{Z}$-cotorsion pair, in an abelian category $\A,$ consists of the following data: 
\begin{itemize}
\item[(a)] a thick subclass  $\mathcal{Z}$ of $\A;$ 
\item[(b)] a pair of clases of objects $(\F,\G)$ in $\mathcal{Z}$ satisfying the following conditions
 \begin{itemize}
 \item[(b1)] $\F={}^{\perp_1}\G\cap\mathcal{Z}$ and $\G=\F^{\perp_1}\cap\mathcal{Z},$
 \item[(b2)]  for any $Z\in\mathcal{Z}$ there are exact sequences $0\to G\to F\to Z\to 0$ and $0\to Z\to G'\to F'\to 0$ with $F,F'\in\F$ and $G,G'\in\G.$
 \end{itemize}
\end{itemize}
\end{defi}

We will use several kinds of pairs $(\X,\Y)$ of classes of objects in $\A.$ A pair $(\X,\Y)\subseteq\A^2$ is {\bf left complete} (respectively, {\bf right complete}) if for any $A\in\A,$ 
there is an exact sequence $0\to Y\to X\to A\to 0$ (respectively, $0\to A\to Y\to X\to 0),$   where $X\in\X$ and $Y\in\Y.$ We say that the pair $(\X,\Y)$ 
is {\bf complete} if it is both a left and right complete. Finally, the pair $(\X,\Y)$ is 
{\bf hereditary} if $\id_\X\,(\Y)=0.$

\begin{lem} \cite[Corollary 2.4]{S} For a cotorsion pair $(\X,\Y)$ in the abelian category $\A,$ with enough projectives and injectives, the 
 following conditions are equivalent.
 \begin{itemize}
  \item[(a)] Every object in $\A$ has a special $\X$-precover.
\vspace{.2cm}
  \item[(b)] Every object in $\A$ has a special $\Y$-preenvelope.
 \end{itemize}
\end{lem}

\begin{lem} \cite{GR} For a cotorsion pair $(\X,\Y)$ in the abelian category $\A,$ with enough projectives and injectives, the following conditions are equivalent.
 \begin{itemize}
  \item[(a)] $\X$ is resolving.
\vspace{.2cm}
  \item[(b)] $\Y$ is coresolving.
\vspace{.2cm}
  \item[(c)] The pair $(\X,\Y)$ is  hereditary.
 \end{itemize}
\end{lem}

Let $(\X,\Y)$ be a hereditary right cotorsion pair in an abelian category $\A.$ Note that, in this case,  the class $\Y$ is right saturated.

{\sc Abelian categories with aditional structure.}
In some places of the paper, we consider abelian categories with some additional conditions that were introduced by  A. Grothendieck 
 \cite{Gro}. We are particularly interested in the conditions  AB4* and AB4. An abelian category $\A$ is an {\bf AB4*-abelian} category if 
 $\A$ has products and the product of any non empty set of epimorphisms is also an epimorphism. Dually, an {\bf AB4-abelian} category 
 is an abelian category $\A$ which has coproducts and the coproduct of any non empty set of monomorphisms is also a monomorphism. 
  For a nice treatment of this kind of categories, we recommend the readers to see in \cite{P}.
\

Let $\X$ be a class of objects in an abelian category $\A.$ We denote by $\Prod(\X)$ (respectively, $\Add(\X)$) the class of objects in 
$\A$ which are direct summands of products (respectively, coproducts) of elements of $\X.$ In the case of a single object $\X=\{X\},$ 
 for simplicity we just write $\Prod(X)$ and $\Add(X).$
 \
 
  Let $\A$ be an abelian category with coproducts. An object $M\in\A$ is 
 {\bf $\Sigma$-orthogonal} if $\Ext_\A^i(M,M^{(I)})=0$ for all $i\geq 1$ and any set $I.$ Dually, in an  abelian category $\A$ with products, an object $M\in\A$ is 
 {\bf $\Pi$-orthogonal} if $\Ext_\A^i(M^{I},M)=0$ for all $i\geq 1$ and any set $I.$

\begin{rk}\label{B1} Let $\A$ be an AB4*-abelian category with enough injectives. In this case, the product of any set of exact sequences is an exact sequence \cite[Proposition 8.3]{P}. Then, it can be shown that 
$\Ext_\A^i(A,\prod_{\alpha\in\Lambda}B_\alpha)\simeq \prod_{\alpha\in\Lambda}\Ext_\A^i(A,B_\alpha)$ for any $i\geq 0.$ As a consequence of the above, 
we have that $\id(\prod_{\alpha\in\Lambda}B_\alpha)=\sup_{\alpha\in\Lambda}\,\id(B_\alpha).$ In particular, it follows that 
$$\id(\Prod(M))=\id(M)\quad\text{and}\quad {}^\perp\Prod(M)={}^\perp M,$$ for any object $M\in\A.$ If $M$ is $\Pi$-orthogonal, then $\Ext^i_\A(\Prod(M),\Prod(M))=0$ for any $i\geq 1.$
\end{rk}

{\sc Some fundamental results in AB-theory.} \\
 Let $(\X,\omega)$ be a pair of classes of objects in an abelian category $\A.$ The class $\omega$ 
is {\bf $\X$-injective} if $\id_\X(\omega)=0.$ It is said that $\omega$ is a {\bf relative quasi-cogenerator} in $\X$ if for 
any $X\in\X$ there is an exact sequence $0\to X\to W\to X'\to 0,$ with $W\in\omega$ and $X'\in\X.$ If in addition, the inclusion $\omega\subseteq \X$ holds true, the class $\omega$ is called  {\bf relative cogenerator} in $\X.$  
Dually, we have the notions of {\bf $\X$-projective} and  {\bf relative (quasi) generator} in $\X.$ 
\

Let $\X$ and $\omega$ be classes of objects in the abelian category $\A.$ We recall from \cite{BMPS}, that $(\X,\omega)$ is a {\bf left Frobenius pair} if $\X$ is left thick, 
$\omega$ is closed under direct summands in $\A$ and $\omega$ is an $\X$-injective relative cogenerator in $\X.$ If in addition, $\omega$ is also an $\X$-projective relative generator in $\X,$ then it is said that $(\X,\omega)$ is a {\bf strong left Frobenius pair}.

\begin{lem}\label{AB1}  Let $\X$ and $\Y$ be classes of objects in $\A.$ Then
$$\pd_{\Y}\,(\X^{\vee})=\pd_{\Y}\,(\X)$$
\end{lem}
\begin{dem} The proof given in  \cite[Lemma 2.13]{MS} can be carried up to the abelian categories.
\end{dem}

\begin{pro}\label{AB2} Let $(\X,\omega)$ be a pair of classes of objects in $\A$ such that $\omega$ is $\X$-injective.
Then, the following statements hold true.
\begin{itemize}
\item[(a)] $\omega^\wedge$ is $\X$-injective.
\item[(b)] If $\omega$ is closed under direct summands in $\A$ and it is a relative cogenerator in $\X,$ then
$$\omega=\{X\in \X\mid \id_\X(X)=0\}=\X\cap\omega^{\wedge},$$
$$\X\cap \omega^{\vee}=\{X\in \X\mid\id_\X(X)<\infty\}.$$
Furthermore, we have that $\id_\X(M)=\coresdim_\omega(M)$ for any $M\in\X\cap \omega^{\vee}.$
\end{itemize}
\end{pro}
\begin{proof} See in \cite[Proposition  2.7]{BMPS}.
\end{proof}

In the following result, which goes back to M. Auslander and R. O. Buchweitz \cite{AuB},  the expression $\resdim_\omega(K)=-1$ just means that $K=0.$

\begin{teo}\label{AB4} Let $(\X,\omega)$ be a pair of classes of objects in $\A$ such that $\X$ is closed under extensions, $0\in\X$ and 
$\omega$ is a relative quasi-cogenerator in $\X.$ Then, the following statements hold true, for any  $C\in\A$ with $\resdim_\X(C)=n<\infty.$ 
\begin{itemize}
\item[(a)]There exist exact sequences in $\A$
\begin{center}
$0\to K\to X\stackrel{\varphi}\to C\to 0,$
\end{center}
with $\resdim_\omega(K)\leq n-1$ and $X\in\X,$ and
\begin{center}
$0\to C\stackrel{\varphi'}\to H\to X'\to 0,$
\end{center}
with $\resdim_\omega(H)\leq n$ and $X'\in\X.$ If $\omega\subseteq \X$ then $\resdim_\omega(K)= n-1.$
\item[(b)] If $\omega$ is $\X$-injective, then
  \begin{itemize}
  \item[(i)] $\varphi:X\to C$ is an $\X$-precover and $K\in\X^{\perp},$
  \item[(ii)] $\varphi':C\to H$ is an $\omega^{\wedge}$-preenvelope and $X'\in{}^\perp(\omega^{\wedge}).$
  \end{itemize}
  \end{itemize}
\end{teo}
\begin{dem} The proof given in \cite[Theorem 1.1]{AuB} can be adapted to this statements.
\end{dem}

\begin{cor}\label{AB6} Let $\X\subseteq\A$ be a pre-resolving class  and let
$\omega$ be a relative cogenerator in $\X,$  closed under isomorphisms. Then, the followings statements are equivalent, for any  $C\in\A$ 
and $n\geq 0.$ 
\begin{itemize}
\item[(a)] $\resdim_\X(C)\leq n.$ 
\item[(b)]There is an exact sequence $0\to K\to X\stackrel{\varphi}\to C\to 0,$  with $\resdim_\omega(K)\leq n-1$ and $X\in\X.$
\item[(c)]There is an exact sequence $0\to C\stackrel{\varphi'}\to H\to X'\to 0,$ with $\resdim_\omega(H)\leq n$ and $X'\in\X.$
\end{itemize}
\end{cor}
\begin{dem} It follows from Theorem \ref{AB4}, see \cite[Proposition 1.5]{AuB}
\end{dem}

\begin{pro}\label{AB7}  Let $\X$ and $\Y$ be classes of objects  in $\A.$  Then, the following statements hold true.  
 \begin{itemize}
\vspace{.2cm}
  \item[(a)]  $\id_\X\,(L)\leq$ $\id_\X\,(\Y) + \coresdim_\Y\,(L)$  for every $L\in\A.$ 
\vspace{.2cm}
  \item[(b)] Assume that $\Y=\X^{\perp_1}$ or that $\Y$ is a subclass of $\X$ which is closed under direct summands in $\A.$ If 
  $\id_\X\,(\Y)=0$ then
$$\id_\X\,(L)= \coresdim_\Y\,(L)\quad\text{for any }\;L\in\Y^\vee.$$
  \item[(c)] Let $\A$ be with enough injectives, and let $\Y=\X^{\perp_1}$ or $\Y=\X^{\perp}.$
 Then,  for any $M\in\A,$  $$\coresdim_{\Y}\,(M)\leq\id_\X(M).$$ 
In particular, $\coresdim_{\Y}\,(\A)\leq\pd\,\X.$
 \end{itemize}
\end{pro}
\begin{dem} The proof given in  \cite[Theorem 2.1 and Lemma 3.3]{MS} also works for abelian categories.
\end{dem}

The following two results will be very useful in this paper. 

\begin{lem}\label{Ldebildimfin} Let $\A$ be an abelian category with enough projectives and $\B\subseteq\A.$ Then 
$\resdim_{{}^\perp\B}(M)=\pd_\B(M)$ for any $M\in\A.$
\end{lem}
\begin{dem} Let $M\in\A.$ By the dual of Proposition \ref{AB7} (c), we get that $\resdim_{{}^\perp\B}(M)\leq \pd_\B(M).$ On the other hand, 
  since $\pd_{\B}({}^\perp\B)=0,$ it follows from the dual of Proposition \ref{AB7} (a) that $\pd_\B(M)\leq \resdim_{{}^\perp\B}(M).$ 
  Thus $\pd_\B(M)= \resdim_{{}^\perp\B}(M).$
\end{dem}

\begin{rk}\label{resdimPD} Let $\A$ be an abelian category. By the dual of Proposition \ref{AB7} (b), $\pd(M)=\resdim_{\Proj(A)}(M)$ 
for any $M\in\Proj(\A)^\wedge.$ Moreover, in the case that $\A$ has enough projectives, it follows from Lemma \ref{Ldebildimfin} that 
$\pd(M)=\resdim_{\Proj(\A)}(M)$ for any $M\in\A.$
\end{rk}

The following result, whose proof can be found in \cite[Proposition 2.1]{AuB},  establishes a connection between resolutions and relative projective dimensions.

\begin{teo}\label{AB8}  Let $(\X,\omega)\subseteq\A^2$  be such that $\X$ is closed under extensions and direct summands in $\A;$ and let $\omega$ be an $\X$-injective relative cogenerator in $\X,$ which is closed under direct summands in $\A.$ Then
$$\pd_{\omega^{\wedge}}(C)=\pd_{\omega}\,(C)=\resdim_{\X}(C) \quad \forall\, C\in \X^{\wedge}\!.$$
\end{teo}

\begin{pro}\label{AB10} Let $(\X,\omega)$ be a left Frobenius pair in  $\A.$ Then, for any $C\in\X^\wedge$ 
and $n\geq 0,$ the following statements are equivalent.
\begin{itemize}
\item[(a)] $\resdim_\X(C)\leq n.$
\item[(b)] If $0\to K_n\to X_{n-1}\to\cdots \to X_1\to X_0\to C\to 0$ is an exact sequence, with $X_i\in\X$  $\forall\,i\in[0,n-1],$ then $K_n\in\X.$
\end{itemize}
\end{pro}
\begin{proof} \cite[Proposition 3.3]{AuB}.
\end{proof}

\begin{pro}\label{AB12} Let $(\X,\omega)$ be a left Frobenius pair in  $\A.$ Then, $\omega^\wedge$ is 
a right thick class   in $\A.$ 
\end{pro}
\begin{dem}  \cite[Proposition 3.8]{AuB}
\end{dem}

\section{General properties of the relative Gorenstein objects}
Throughout this section, we assume that $\A$ is an abelian category. Consider a class $\X$ of objects in $\A.$ We say that $\X$ is {\bf $\X$-epic} in 
$\A$  if for  any 
$A\in\A$ there is an epimorphism $X\to A,$ with $X\in\X.$ Note that, if $\A$ has enough projectives, then  for the class $\Q_0(\A)$ of 
projective objects in $\A,$ we have that $\Q_0(\A)$  is $\Q_0(\A)$-epic in $\A$.
Dually, we have the notion saying when $\X$  is $\X$-monic in $\A.$ For example, the class $\I_0(\A),$ of 
 injective objects in $\A,$ is $\I_0(\A)$-monic in $\A$ if $\A$ has enough injectives.

\begin{defi} A pair $(\X,\Y)$ of classes of objects in $\A$ is {\bf weak GP-admissible} if $\pd_\Y(\X)=0$ and $\X$ is $\X$-epic in $\A.$ If in addition, the pair $(\X,\Y)$ satisfies the following two conditions
\begin{itemize}
 \item[\normalfont{(a)}]  $\X$ and  $\Y$ are closed under finite coproducts  in $\A$, and $\X$ is closed under extensions;
 \item[\normalfont{(b)}] $\omega := \X \cap \Y$ is a relative cogenerator in  $\X;$
\end{itemize}
 we say that  $(\X,\Y)$ is {\bf GP-admissible}.
\end{defi}

\begin{defi}\label{defiGP} Let $(\X,\Y)\subseteq \A^2.$  A left complete $(\X,\Y)$-resolution is an acyclic complex 
$$\eta:\quad \cdots \to X_1\to X_0\to X^0\to X^1\to\cdots$$
with $X_i,X^i\in\X$ and such that the complex $\Hom_\A(\eta,Y)$ is acyclic for any $Y\in\Y.$ The object 
$M:=\Ima(X_0\to X^0)$ is called {\bf  $(\X,\Y)$-Gorenstein  projective}   or $GP_{(\X,\Y)}$-object. The class of all 
$GP_{(\X,\Y)}$-objects is denoted by $\GP_{(\X,\Y)}(\A)$ or $\GP_{(\X,\Y)}.$ We also say that $\X$ is the {\bf approximation class} and $\Y$ is 
the {\bf testing class} in $\GP_{(\X,\Y)}.$
\end{defi}

Note that $\X\subseteq \GP_{(\X,\Y)}$ if $0\in\X.$ The notion of left complete $(\X,\Y)$-resolution was already considered in \cite[Definition 2.1]{PC}, but only for the case that $\Q_0(\A)\subseteq \X$ and $\A=\Modu(R),$ the category of left $R$-modules over some ring $R.$ 

\begin{defi} For any $(\X,\Y)\subseteq\A^2$ and $M\in\A,$ the  $(\X,\Y)$-Gorenstein  projective dimension of $M$ is
$$\Gpd_{(\X,\Y)}(M):=\resdim_{\GP_{(\X,\Y)}}(M).$$
For any class $\mathcal{Z}\subseteq\A,$ we set $\Gpd_{(\X,\Y)}(\mathcal{Z}):=\sup\{\Gpd_{(\X,\Y)}(Z)\;:\;Z\in\mathcal{Z}\}.$
If $\X=\Y,$ for simplicity, we set $\GP_\X:=\GP_{(\X,\X)}$ and $\Gpd_{\X}(M):=\Gpd_{(\X,\X)}(M).$
\end{defi}

\begin{ex} Let $(\X,\Y)\subseteq\A^2$.
\begin{itemize}
\item[(1)] If $\X=\Q_0(\A)=\Y,$  the relative $(\X,\Y)$-Gorenstein  projective objects are just the usual 
Gorenstein projective objects in $\A.$ In this case, we write
 \begin{center} $\GP(\A):=\GP_{\Q_0(\A)}$ and $\Gpd(M):=\Gpd_{\Q_0(\A)}(M).$\end{center}
\item[(2)] If $\Q_0(\A)\subseteq\X=\Y,$  the relative $(\X,\Y)$-Gorenstein  projective objects are just called 
$\X$-Gorenstein projective objects  \cite{BO, Ta}. 
 \item[(3)]  For $\A:=\Modu\,(R),$  $\X:=\Proj\,(R):=\Q_0(\A)$ and  the class $\Y:=\Flat\,(R)$ of
flat $R$-modules,  the  $(\X,\Y)$-Gorenstein  projective modules are just known as the {\bf Ding-projective} modules \cite{Gi}.
In this case, we write
 \begin{center} $\DP(R):=\GP_{(\Proj\,(R),\Flat\,(R))}$ and $\Dpd\,(M):=\Gpd_{(\Proj(R),\Flat\,(R))}(M).$\end{center}
\end{itemize}
  
\end{ex}
\begin{lem}\label{GP1} Let $(\X,\Y)\subseteq\A^2$ with $0\in\X.$ Then,  the class $\X$ is a relative  generator and cogenerator in $\GP_{(\X,\Y)}.$
\end{lem}
\begin{dem}  Let $\cdots\to X_1\xrightarrow{d_1} X_0\xrightarrow{d_0} X^0\xrightarrow{d^0} X^1\xrightarrow{d^1}\cdots$ be a left complete $(\X,\Y)$-resolution 
such that $M=\Ima\,(d_0).$ Then, we have the exact sequences $0\to\Ima(d_1)\to X_0\to M\to 0$ and 
$0\to M\to X^0\to\Ima(d^0)\to 0,$ where $\Ima(d_1)$ and $\Ima(d^0)$ are $GP_{(\X,\Y)}$-objects.
\end{dem}

\begin{defi} A pair $(\X,\Y)$ of classes of objects in $\A$ is  {\bf weak GI-admissible} if $\id_\X(\Y)=0$ and $\Y$ is $\Y$-monic in 
$\A.$  If in addition, the pair $(\X,\Y)$ satisfies the following two conditions
\begin{itemize}
 \item[\normalfont{(a)}]  $\X$ and  $\Y$ are closed under finite coproducts  in  $\A$, and $\Y$ is closed under extensions;
 \item[\normalfont{(b)}] $\omega := \X \cap \Y$ is a relative generator in  $\Y;$
\end{itemize}
 we say that  $(\X,\Y)$ is {\bf GI-admissible}.
\end{defi}

\begin{defi} Let $(\X,\Y)\subseteq\A^2.$ A  right complete  $(\X,\Y)$-resolution is an acyclic complex 
$$\eta:\quad \cdots \to Y_1\to Y_0\to Y^0\to Y^1\to\cdots$$
with $Y_i,Y^i\in\Y$ and such that the complex $\Hom_\A(X,\eta)$ is acyclic for any $X\in\X.$ The object 
$M:=\Ima(Y_0\to Y^0)$ is called {\bf $(\X,\Y)$-Gorenstein  injective} or $GI_{(\X,\Y)}$-object. The class of all 
$GI_{(\X,\Y)}$-objects is denoted by $\GI_{(\X,\Y)}(\A)$ or $\GI_{(\X,\Y)}.$ We also say that $\Y$ is the approximation class and 
$\X$ is the testing class in $\GI_{(\X,\Y)}.$
\end{defi}

If $\X=\I_0(\A)=\Y,$  the relative $(\X,\Y)$-Gorenstein  injective objects are just the usual 
Gorenstein injective objects in $\A.$ In the case that $\A=\Modu\,(R),$  $\X$ is the class of the FP-injective  $R$-modules (i.e. those $E$ such that 
$\Ext^1(F,E)=0$ for any finitely presented $R$-module $F$) and $\Y:=\I_0(\A),$ we have that the  $(\X,\Y)$-Gorenstein  injective modules are just the {\bf Ding-injective} modules \cite{Gi}.

\begin{defi} For any $(\X,\Y)\subseteq\A^2$ and $M\in\A,$ the  $(\X,\Y)$-Gorenstein  injective dimension of $M$ is
$$\Gid_{(\X,\Y)}(M):=\coresdim_{\GI_{(\X,\Y)}}(M).$$
\end{defi}

\begin{rk} Let $(\X,\Y)\subseteq\A^2.$  It can be seen that $(\X,\Y)$ is $GI$-admissible in $\A$ if and only if 
$(\Y^{op},\X^{op})$ is $GP$-admissible in the opposite category $\A^{op}.$ Moreover, 
$\GI_{(\X,\Y)}(\A)=(\GP_{(\Y^{op},\X^{op})}(\A^{op}))^{op}.$ Therefore, any obtained result for $GP_{(\X,\Y)}$-objects can be translated 
into a result for $GI_{(\X,\Y)}$-objects. So, in what follows, we are dealing only with the relative Gorenstein projective objects.
\end{rk}

The following  result \cite[Lemma 4.1.1]{Chris} is an useful tool for studying the class of relative Gorenstein  projective objects. 

\begin{lem}\label{Chris} Let $(X^{\bullet},d_{X^{\bullet}})$ be an acyclic cochain complex of objects in $\A,$ and let $N\in\A$ be such that 
$X^i\in{}^{\perp}N$ for any $i\in\mathbb{Z}.$ Then, the following statements are equivalent, for 
${Z^i}_{\!\!X^{\bullet}}:=\Ker\,({d^i}_{\!\!X^{\bullet}}).$
\begin{itemize}
\item[(a)] The complex $\Hom_\A(X^{\bullet},N)$ is acyclic.
\item[(b)] ${Z^i}_{\!\!X^{\bullet}}\in{}^{\perp_1}N$ for any $i\in\mathbb{Z}.$
\item[(c)] ${Z^i}_{\!\!X^{\bullet}}\in{}^{\perp}N$ for any $i\in\mathbb{Z}.$
\end{itemize}
\end{lem}

\begin{defi}\label{debildefi2} Let $(\X, \Y) \subseteq \A^2.$ We introduce the subclass $W\GP _{(\X,\Y)}$ of $\A$ whose objects 
are the {\bf weak $(\X, \Y)$-Gorenstein projectives} or  $WGP_{(\X,\Y)}$-objects. For $M \in \A,$ we say that $M$ is a 
$WGP_{(\X,\Y)}$-object if  
$M \in{}^{\perp} \Y $ and there is an exact sequence $\xi : 0\to M \to X^0 \to X^1 \to \cdots,$ with $X^{i}\in \X$ and 
$\Ima\, (X^{i} \to X^{i+1}) \in{}^{\perp} \Y$ for any $i \in \mathbb{N}$. If $\X = \Y,$ by simplicity, we just write $W\GP _{\X}$ instead of 
$W\GP _{(\X, \X)}$. In this case, any object in $W\GP _{\X}$ is called {\bf weak $\X$-Gorenstein projective}.
\

The weak $(\X,\Y)$-Gorenstein  projective dimension of $M$ is
$$\WGpd_{(\X,\Y)}(M):=\resdim_{W\GP_{(\X,\Y)}}(M).$$
For any class $\mathcal{Z}\subseteq\A,$ we set $\WGpd_{(\X,\Y)}(\mathcal{Z}):=\sup\{\WGpd_{(\X,\Y)}(Z)\;:\;Z\in\mathcal{Z}\}.$
If $\X=\Y,$ we set  $\WGpd_{\X}(M):=\WGpd_{(\X,\X)}(M).$ Dually, we have the class $W\GI _{(\X,\Y)}$ of the  {\bf weak $(\X, \Y)$-Gorenstein injectives} or  $WGI_{(\X,\Y)}$-objects in $\A,$ and the  {\bf weak $(\X,\Y)$-Gorenstein  injective dimension} $\WGid_{(\X,\Y)}(M)$ of $M.$
\end{defi}

\begin{ex} (1) \cite{BGRO} \label{GCProj}  The class $G_CP(R)$ of the $G_C$-projective $R$-modules is introduced in \cite[Definition 2.2]{BGRO} and also it is studied its homological properties for the case that $C$ be a weakly-Wakamatsu tilting $R$-module (i.e $C$ is $\Sigma$-orthogonal and ${}_RR\in W\GP_{\Add(C)}).$
\

 Let $C$ be $\Sigma$-orthogonal. Then,   by \cite[Proposition 2.4]{BGRO} and Lemma \ref{Chris}, we get that 
$G_CP(R)=W\GP_{\Add(C)}.$ As we will see through out the paper, many of the results obtained in \cite{BGRO} are particular cases of the 
developed theory in this paper. 
\

(2) \cite{BR} Let $\A$ be an abelian category and $\omega\subseteq\A$ be such that $\id_\omega(\omega)=0.$ In this case, the objects in $W\GP_\omega$ are called Cohen-Macaulay objects in $\A$ and this class of objects is denoted in \cite{BR} by 
$\mathrm{CMC}(\omega).$
\end{ex}

\begin{rk}\label{WGorro} For any $(\X, \Y) \subseteq \A^2,$ we have that $W\GP _{(\X,\Y)}=W\GP _{(\X,\Y^\wedge)}.$ Indeed, by the dual of Lemma \ref{AB1}, we 
know that $\pd_\Y(M)=\pd_{\Y^\wedge}(M)$ for any $M\in \A,$ and thus $^{\perp} \Y={}^{\perp}(\Y^\wedge).$
\end{rk}

Note that, in general, for any arbitrary pair $(\X,\Y) \subseteq \A ^2$, the class $\GP_{(\X,\Y)}$ does not have to be equal  to $W\GP_{ (\X,\Y)}.$ However, we can stablish the following relationship between them.

\begin{pro}\label{GP-WGP} For any $(\X, \Y) \subseteq \A^2,$  the following statements hold true.
\begin{itemize}
\item[(a)] If $\pd _{\Y} (\X)=0$ then $\GP _{(\X,\Y)} \subseteq W \GP _{(\X,\Y)}.$
\item[(b)] If $0\in\X$ and $\GP _{(\X,\Y)} \subseteq W \GP _{(\X,\Y)},$ then $\pd _{\Y} (\X)=0.$
\end{itemize} 
\end{pro}
\begin{dem} (a)
Let $ G\in \GP _{(\X, \Y)}$. Then there is a  complete left $(\X,\Y)$-resolution 
 $$\eta : \cdots \to X_1 \to X_0 \to X^0 \to X^1 \to \cdots ,$$ 
such that $ G = \Ker ( X^0 \to X^1)$. In particular, for the exact sequence $$\eta ' : 0 \to G \to X^0 \to X^1 \to \cdots ,$$ we have that  the complex
$\Hom _{\A} (\eta ' ,Y )$ is acyclic for any $Y \in \Y$.  Therefore, by Lemma \ref{Chris} it follows that $\Ker (X^{i} \to X^{i+1}) \in{}^{\perp} \Y$ for any integer $i,$ and thus  $ G \in W \GP _{(\X , \Y)}.$
\

(b) Since $0\in\X,$ we have that $\X\subseteq \GP _{(\X,\Y)}.$ Thus, the result follows using that  
$ W \GP _{(\X,\Y)}\subseteq{}^\perp \Y.$
\end{dem}

\begin{cor}\label{GP3} Let $(\X, \Y) \subseteq \A^2$ be such that $\pd_\Y(\X)=0.$ Then
\begin{itemize}
\item[(a)] $\pd_{\Y^{\wedge}}(\GP_{(\X,\Y)})=\pd_{\Y^{\wedge}}(W\GP_{(\X,\Y)})=\pd_{\Y^{\wedge}}(\X)=0,$
\item[(b)] $\GP_{(\X,\Y)}\subseteq {}^\perp\Y\cap {}^\perp(\Y^\wedge)$ and $W\GP_{(\X,\Y)}\subseteq {}^\perp\Y\cap {}^\perp(\Y^\wedge).$
\end{itemize}
\end{cor}
\begin{dem} By Proposition \ref{GP-WGP}, we get that $\pd_\Y(\GP_{(\X,\Y)})=0.$ On the other hand, by hypothesis, we have that 
$\pd_\Y(\X)=0=\pd_\Y(W\GP_{(\X,\Y)}).$ Hence the item (a) follows from the dual of Lemma \ref{AB1}. Finally, (b) follows from (a) and the 
equality $\pd_\Y(\GP_{(\X,\Y)})=0.$
\end{dem}
\vspace{0.2cm}

The following result is a generalization of \cite[Proposition 2.4]{BGRO},  \cite[Proposition 2.3]{Holm}, \cite[Proposition 3.8]{Ta} and  
   \cite[Proposition 2.4]{PC}.  

\begin{pro}\label{GP2} Let $(\X,\Y)$ be a  weak $GP$-admissible pair in $\A.$ Then,  for $M\in\A,$ the following conditions are equivalent.
\begin{itemize}
\item[(a)] $M\in\GP_{(\X,\Y)}.$
\item[(b)] $M\in{}^\perp\Y$ and there is an exact sequence $\varepsilon:\quad 0\to M\to X^0\to X^1\to\cdots,$ with $X^i\in\X$ and such 
that the complex $\Hom_\A(\varepsilon,Y)$ is acyclic for any $Y\in\Y.$
\item[(c)] $M\in W\GP_{(\X,\Y)}.$
\end{itemize}
\end{pro}
\begin{dem} (a) $\Rightarrow$ (c) It is Lemma \ref{GP-WGP} (a).
\

(b) $\Leftrightarrow$ (c) It is a direct consequence of Lemma \ref{Chris}.
\

(b) $\Rightarrow$ (a) Since $\X$  is $\X$-epic in $\A,$ we can construct an exact sequence as follows
$$\varepsilon':\quad\cdots\to X_2\stackrel{d_2}\to X_1\stackrel{d_1}\to X_0\stackrel{d_0}\to M\to 0,$$
where $X_i\in\X$ for any $i\in\mathbb{N}.$
Let $K_i:=\Ker\,(d_i).$ Then, by applying the functor $\Hom_\A(-,Y),$ for any $Y\in\Y,$ to the exact sequence $0\to K_0\to X_0\to M\to 0,$
we get the following exact sequence
$$\Ext^i_\A(X_0,Y)\to \Ext^i_\A(K_0,Y)\to\Ext^{i+1}_\A(M,Y).$$
But now, the facts that $\pd_\Y(\X)=0$ and $M\in{}^\perp\Y$ imply that $K_0\in{}^\perp\Y.$ So, by Lemma \ref{Chris}, we have that $K_0$ satisfies the same hypothesis as $M$ does in (b). Then, we can  replace $M$ 
by $K_0.$ Therefore, by repeating this procedure and using again Lemma \ref{Chris}, we get that the complex $\Hom_\A(\varepsilon',Y)$ 
is acyclic for any $Y\in\Y.$ Hence, by putting together $\varepsilon'$ and $\varepsilon,$ we obtain a left complete $(\X,\Y)$-resolution 
$\cdots\to X_1\stackrel{d_1}\to X_0\to X^0\to X^1\to\cdots$ such that $M=\Ima\,(X_0\to X^0).$
\end{dem}

\begin{cor}\label{TGP} Let $\A$ be an abelian category with enough injectives and let $(\X,\Y)$ be a hereditary cotorsion pair  in $\A$ which is right complete. Then, the following statements hold true.
\begin{itemize}
\item[(a)] $\X$ is left thick and $\Y$ is right thick. Moreover,  $\X\cap\Y$ is an $\X$-injective relative cogenerator in $\X.$
\item[(b)]  $W\GP_{(\X,\Y)}=\GP_{(\X,\Y)}=\X={}^\perp\Y.$ 
\end{itemize}
\end{cor}
\begin{dem} (a) Since $\X^{\perp_1}=\Y,$ ${}^{\perp_1}\Y=\X$ and $\id_\X(\Y)=0,$ it follows that $\X$ is left thick and $\Y$ is right thick. On the other hand, the fact that  $(\X,\Y)$ is right complete implies that $\X\cap\Y$ is a relative cogenerator in $\X.$ Finally, it is 
clear that $\id_{\X}(\X\cap\Y)=0.$
\

(b) Since $\A$ has enough injectives and  $\Y$ is coresolving, it follows that ${}^{\perp_1}\Y={}^\perp\Y$ and thus 
$\X={}^\perp\Y.$ Then, (b) follows from the inclusions $\X\subseteq \GP_{(\X,\Y)}\subseteq{}^\perp\Y$ and 
$\X\subseteq W\GP_{(\X,\Y)}\subseteq{}^\perp\Y.$
\end{dem}

As a consequence of the above corollary, it follows that complete hereditary cotorsion pairs can be seen as particular cases of the relative Gorenstein theory in abelian categories with enough injectives and projectives. More specifically, we have the following remark.

\begin{rk} Let $\A$ be an abelian category with enough injectives and projectives,  and let $(\X,\Y)$ be a complete hereditary cotorsion pair  in $\A.$ Then, the pair $(\X,\Y)$ is both $GP$-admissible and $GI$-admissible. Furthermore, $\GP_{(\X,\Y)}=\X$ and $\GI_{(\X,\Y)}=\Y.$
\end{rk}

\begin{pro}\label{GP4} For  $(\X,\Y)\subseteq\A^2, $  the following statements are equivalent.
\begin{itemize}
\item[(a)] $\X$ is $\GP_{(\X,\Y)}$-injective.
\item[(b)] $\pd_\X(\X)=0$ and $\GP_{(\X,\Y)}\subseteq\GP_\X.$
\end{itemize}
If one of the above conditions holds true, then  $\Gpd_\X(M)\leq\Gpd_{(\X,\Y)}(M),$ for any $M\in\A.$
\end{pro}
\begin{dem} (a) $\Rightarrow$ (b) We start by proving that $\pd_\X(\X)=0.$ Indeed, by using that $\X$ is $\GP_{(\X,\Y)}$-injective , we obtain
$$\pd_\X(\X)\leq \pd_{\X}(\GP_{(\X,\Y)})=0.$$
Let $M\in\GP_{(\X,\Y)}.$ Then, there is  
a left complete $(\X,\Y)$-resolution
$$\eta:\quad \cdots \to X_1\to X_0\to X^0\to X^1\to\cdots$$
such that  $M=\Ima(X_0\to X^0).$  We assert that the complex 
$\Hom_\A(\eta,X)$ is acyclic for any $X\in\X.$ Indeed, since all the cycles of the complex $\eta$ are $GP_{(\X,\Y)}$-objects and 
$\pd_\X(\GP_{(\X,\Y)})=0,$ we obtain from Lemma \ref{Chris} that $\Hom_\A(\eta,X)$ is acyclic for any $X\in\X.$ Hence 
$M\in\GP_\X.$
\

(b) $\Rightarrow$ (a) Since $\pd_\X(\X)=0,$ it follows from Proposition \ref{GP-WGP} (a) that $\GP_\X\subseteq {}^\perp\X.$ Thus 
$\GP_{(\X,\Y)}\subseteq {}^\perp\X$ and then $\X$ is $\GP_{(\X,\Y)}$-injective.
\end{dem}

\begin{lem}\label{G4.5} Let  $(\X,\Y)\subseteq\A^2$ be such that $\pd_\Y(\X)=0.$ If $\X\subseteq\Y^\wedge$  then $\X$ is $\GP_{(\X,\Y)}$-injective.
\end{lem}
\begin{dem} Let $\X\subseteq\Y^\wedge.$ Then, by  Corollary \ref{GP3}  
$\pd_\X(\GP_{(\X,\Y)})\leq \pd_{\Y^{\wedge}}(\GP_{(\X,\Y)})=0,$
proving that $\X$ is $\GP_{(\X,\Y)}$-injective.
\end{dem}
\vspace{0.2cm}

The following result is a generalization of \cite[Theorem 2.3]{BO},  \cite[Theorem 2.5]{Holm}, \cite[Theorem 2.5]{PC} and \cite[Theorem 3.11]{Ta} to the context of relative Gorenstein objects. Another possible generalization of it is given in  Corollary \ref{ThickGP}.

\begin{teo}\label{GP6}  Let $(\X,\Y)\subseteq\A^2$ be a weak $GP$-admissible pair in an abelian category $\A,$ with enough projectives,   such that $\X$ is $\GP_{(\X,\Y)}$-injective. Then, 
the following statements hold true.
\begin{itemize}
\item[(a)] If $\X$ is closed under finite coproducts, then $\GP_{(\X,\Y)}$ is a pre-resolving class.
\item[(b)] If $\A$ is AB4 and $\X$ is closed under coproducts, then $\GP_{(\X,\Y)}$ is  closed under coproducts and a left thick class in 
$\A.$
\end{itemize}
\end{teo}
\begin{dem} By using Proposition \ref{GP2} and the dual of Remark \ref{B1}, the proof given in \cite[Theorem 2.3]{BO} can be adapted to our situation.
\end{dem}
\vspace{0.2cm}

The aim in what follows is to prove that the class $\GP _{(\X,\Y)}$ (respectively,  $W\GP _{(\omega,\Y)}$) is a  left thick class in $\A$ if the pair 
$(\X,\Y)$ is GP-admissible (respectively, $\omega\subseteq\Y$ and $\omega$ is closed under finite coproducts in $\A$). In order to 
do that, we start with a series of Lemmas,  propositions and theorems.

\begin{lem}\label{oobss2} Let $(\X, \Y)\subseteq\A^2$ be a weak GP-admissible pair and  $M \in \A$. Then, $M \in \GP _{(\X,\Y)}$ iff there is an exact sequence $ 0 \to M \to X \to G \to 0$  in $\A,$   with $X \in \X$ and $G \in \GP_{(\X,\Y)}.$
 \end{lem}
 \begin{dem} $(\Rightarrow)$ Let $M \in \GP _{(\X ,\Y)}.$ Then, there is a left complete  $(\X,\Y)$-resolution 
  $$\cdots \to X_1 \to X_0 \to X ^0 \to X^1 \to \cdots$$ 
  such that  $M = \Ker (X^0 \to X^1)$ and moreover  $G:= \Ker (X^1 \to X^2) \in \GP_{(\X,\Y)}.$ Thus, we get an exact sequence
   $0\to M \to X^0 \to G \to 0$, with $X^0\in \X$ and $G \in \GP_{(\X,\Y)}.$
  \
  
  $(\Leftarrow)$ Let  $\eta : 0\to M \to X \to G \to 0$ be an exact sequence in $\A,$  with $X \in  \X$ and $G \in \GP _{(\X, \Y)}$. By 
  Proposition  \ref{GP2}, we get  $G\in{}^{\perp} \Y$ and an exact sequence  
  $$\xi :0 \to G \to X^0 \to X ^1 \to \cdots ,$$ 
  with  $X^{i} \in \X ,$ $\forall i \in \mathbb{N},$  such that the complex $\Hom _{\A} (\xi , Y)$ is acyclic for any $Y \in \Y$. By putting together 
  $\eta$ and $\xi ,$ we get the exact sequence 
   $$ \epsilon : 0\to  M \to X \to X^0 \to X^1 \to \cdots .$$ 
  In order to see that  the complex $\Hom_{\A} (\epsilon, Y)$ is acyclic $\forall \; Y \in \Y,$ it is enough to check the same for $\Hom _{\A} (\eta, Y ).$ 
 Let $Y \in \Y. $ By applying the functor $\Hom(-,Y)$ to  $\eta,$ we obtain the exact sequence  
 $$0 \to \Hom _{\A} (G,Y) \to \Hom _{\A} (X, Y) \to \Hom _{\A} (M,Y) \to \Ext ^1 (G , Y) .$$ 
 Note that $\Ext ^1 (G , Y) = 0,$ since $G \in \GP _{(\X ,\Y)} \subseteq{}^{\perp} \Y$. 
 \
 
 We assert that $M \in{}^{\perp} \Y $. Indeed, take 
 $Y \in \Y$ and apply $\Hom(-,Y)$ to $\eta.$ Then, we get the exact sequence  
 $$ \Ext ^{i}_{\A} (G, Y) \to \Ext ^{i}_{\A} (X,Y) \to \Ext ^{i}_{\A} (M,Y) \to \Ext ^{i+1}_{\A} (G, Y) ,$$ 
 where $ \Ext ^{i}_{\A} (G, Y)=0=\Ext ^{i+1}_{\A} (G, Y).$ Using the fact that  $\pd_{\Y} (\X) =0,$  it follows that 
 $M\in{}^{\perp}\Y.$ Finally, by  Proposition \ref{GP2}, we conclude that $M \in \GP_{(\X,\Y)}.$
 \end{dem}
\vspace{0.2cm}

The following Lemma is a generalization of \cite[Lemma 2.4]{Xu}.

\begin{lem}\label{CasiExten} 
Let $(\X,\Y) \subseteq \A^2$ be  such that  $\pd_\Y(\X)=0$ and $\X$ is closed under extensions, and let $0 \to A \to B\to C \to 0$ be an exact sequence in 
$\A$. If $A \in W\GP _{(\X , \Y)}$ and $C \in \X,$ then   $B \in W\GP _{(\X ,\Y)}.$ 
\end{lem}
\begin{dem} Let $A \in W\GP _{(\X , \Y)}$ and $C \in \X.$ In particular, there is an exact sequence $0 \to A \to X \to G \to 0$, with $X \in \X$ and 
$G \in W\GP _{(\X, \Y)}.$ Thus, we get an exact and commutative diagram in $\A$
$$\xymatrix{& 0\ar[d] & 0\ar[d] &  &\\
0\ar[r]  & A\ar[r]\ar[d] & B \ar[r]\ar[d] & C\ar[r]\ar@{=}[d] & 0\\
0\ar[r] & X \ar[r]\ar[d] & Q \ar[r]\ar[d] & C\ar[r] & 0\\
& G \ar@{=}[r]\ar[d] & G \ar[d] & &\\
& 0 & \;0. & &}$$
Since $\X$ is closed under  extensions and $X ,C \in \X,$ we have that $Q \in \X .$ Therefore, we obtain an exact sequence  
$\varepsilon:\; 0 \to B \to Q \to G \to 0,$ where $Q \in \X$ and $ G \in W\GP_{(\X,\Y)}.$ By using the exact sequence $\varepsilon$ and  
the fact that $Q,G\in{}^\perp\Y,$ it can be shown that $B\in{}^\perp\Y.$ Thus  $B \in W\GP _{(\X,\Y)}.$
\end{dem}

\begin{pro}\label{cogenera}
Let $(\X, \Y)\subseteq \A^2$ be such that $\pd_\Y(\X)=0$ and $\X$ is closed under extensions, and let $\omega:= \X \cap \Y$ be  closed under finite coproducts in $\A$ and a relative cogenerator in $\X.$ Then, the following statements hold true.
\begin{itemize}
\item[(a)] The class $\omega$ is closed under extensions,  a $W\GP _{(\X ,\Y)}$-injective 
relative cogenerator in  $W\GP _{(\X ,\Y)}$ and $\id_\omega(\omega)=0.$
\item[(b)] If $\omega$ is closed under direct summands in $\A,$ then $\omega=\Y\cap W\GP_{(\omega,\Y)}.$
\end{itemize}
\end{pro}
\begin{dem} (a) Since $\pd_\Y(\X)=0,$ it follows that $\Ext^i_\A(\omega,\omega)=0$ for any $i\geq 1.$ Then, any exact sequence $0\to W\to E\to W'\to 0,$ with $W,W'\in\omega,$ splits and thus $E=W\oplus W'\in\omega$ since $\omega$ is closed under finite coproducts in $\A.$
\

We prove, now, that $\omega$ is a relative cogenerator in $W\GP _{(\X,\Y)}.$ Indeed, let $G \in W\GP _{(\X,\Y)}.$ In particular, there exists  an exact sequence $0 \to G \to X \to G' \to 0,$ with $X \in \X$ and $G' \in W\GP _{(\X,\Y)}.$ Since $\omega$ is a relative cogenerator 
in  $\X,$  there is an exact sequence $0 \to X \to W \to X' \to 0,$ with $X' \in \X$ and  $W \in \omega.$ Hence, we get the 
following exact and commutative diagram
$$\xymatrix{&  & 0\ar[d] & 0\ar[d] &\\
\; \;  0\ar[r]  & G\ar[r] \ar@{=}[d]  & X\ar[r]\ar[d] & G' \ar[r] \ar[d] & 0\\
\eta : 0\ar[r] & G\ar[r] & W \ar[r]\ar[d] & T \ar[r] \ar[d] & 0\\
& & X' \ar@{=}[r]\ar[d] & X' \ar[d]  &\\
&  & 0 & \;0.& }$$
Using that $G' \in W\GP _{(\X,\Y)},$  $X' \in \X$ and Lemma \ref{CasiExten}, we get that $T \in W\GP _{(\X,\Y)}.$ Moreover, from the inclusions
$\omega\subseteq\X\subseteq W\GP _{(\X,\Y)}$ and the exact sequence $\eta,$ it follows that $\omega$ is a relative cogenerator in 
$W\GP _{(\X,\Y)}.$
Finally, we have that $\id _{W\GP_{(\X,\Y)}} (\omega)=0,$ since $W\GP _{(\X,\Y)} \subseteq {}^{\perp} \Y \subseteq {}^{\perp} \omega.$ 
\

(b) Let $\omega$ be closed under direct summands in $\A.$ We prove that $\omega=\Y\cap W\GP_{(\omega,\Y)}.$ It is clear that 
$\omega\subseteq\Y\cap W\GP_{(\omega,\Y)}.$ On the other hand, 
let $M\in \Y\cap W\GP_{(\omega,\Y)}.$ Then, by (a), there is an exact sequence $\varepsilon:\; 0\to M\to W\to C\to 0,$ with 
$W\in\omega$ and $C\in W\GP_{(\omega,\Y)}.$ Since $\Ext_\A^1(W\GP_{(\omega,\Y)},\Y)=0,$ it follows that $\varepsilon$ splits and then $M\in\omega.$
\end{dem}

\begin{cor}\label{GPcogenera} Let $(\X ,\Y)$ be a GP-admisible pair in $\A$ and $\omega := \X \cap \Y$. Then, the following statements hold true.
\begin{itemize}
\item[(a)] The class $\omega$ is closed under extensions,  a $\GP _{(\X ,\Y)}$-injective 
relative cogenerator in  $\GP _{(\X ,\Y)}$ and $\id_\omega(\omega)=0.$
\item[(b)] If $\omega$ is closed under direct summands in $\A,$ then $\omega=\Y\cap W\GP_{(\omega,\Y)}.$
\end{itemize}
\end{cor}
\begin{dem} Since $(\X ,\Y)$ is GP-admisible, we get from Proposition \ref{GP2} that $\GP _{(\X ,\Y)}=W\GP _{(\X ,\Y)}.$ Then, the 
result follows from Proposition \ref{cogenera}.
\end{dem} 

 \begin{pro}\label{GXY-GXW}
 Let $(\X ,\Y)$ be a GP-admisible pair in $\A$ and $\omega := \X \cap \Y$. Then, the pair $(\X , \omega)$ is GP-admisible and $\GP _{(\X ,\Y)} \subseteq \GP _{(\X ,\omega)}.$
 \end{pro}
 \begin{proof} Since $\omega\subseteq\Y$ and $\pd _{\Y} (\X) =0,$ it follows that $\pd _{\omega} (\X)=0$ and hence  $(\X, \omega)$ is  GP-admisible. 
 \
 
 Let  $M \in \GP _{(\X,\Y)}.$ Then $M\in {}^{\perp}  \Y \subseteq {}^{\perp} \omega.$ Moreover, by Lema \ref{Chris} there exists an exact sequence  $ \xi ^{+} : 0 \to M \to X_0 \to X_1 \to \cdots ,$ with $X_i \in \X,$ such that the complex $\Hom (\xi ^{+} ,Y)$ is acyclic for any $Y \in \Y.$ In particular, $\Hom (\xi ^{+} ,W)$ is acyclic for any $W\in \omega.$ Therefore, from Lema \ref{Chris}, we get that 
 $M \in \GP _{(\X ,\omega)}.$
 \end{proof}

In what follows, we study the category $W\GP _{(\omega, \Y)}$ (see Definition \ref{debildefi2}), in the case that $\omega \subseteq \Y.$ A particular situation of that was firstly considered by  M. Auslander and I. Reiten  \cite{AuR}, for $\omega = \add (T) = \Y,$ where  $T$ is a self-ortogonal 
(i.e. $\Ext ^{i} _{\Lambda} (T,T) =0 \; \forall \; i \geq 1$) finitely generated left 
$\Lambda$-module and   $\Lambda$ is an Artin algebra.

\begin{rk} \label{cogenera3}
For any pair $(\omega,\Y) \subseteq \A ^2$ such that $\omega \subseteq \Y,$ the following statements hold true.
\begin{itemize}
  \item[(a)]  $\omega$ is $W\GP _{(\omega, \Y)}$-injective and a relative quasi-cogenerator in $W\GP _{(\omega, \Y)}.$
  \item[(b)]  If $\pd _{\Y} (\omega) =0$ and $ 0 \in \omega,$ then $\omega$ is a relative cogenerator in $W\GP_{(\omega, \Y)}.$
   \item[(c)] If $\omega$ is a relative cogenerator in $W\GP_{(\omega, \Y)},$ then $\pd _{\Y} (\omega) =0.$
\end{itemize}
\end{rk}

\begin{lem}\label{GXY-WGW}
For any GP-admissible pair $(\X , \Y)\subseteq \A^2$ and $\omega:=\X \cap \Y,$ it follows that    
$\GP _{(\X,\Y)} \subseteq W \GP _{\omega}.$
\end{lem}
\begin{dem}
Let $G \in \GP_{(\X,\Y)} .$ By Corollary \ref{GPcogenera} (a), there is an exact sequence $0\to G \to W_0 \to G_0 \to 0,$ with $W_0 \in \omega$ and 
 $G_0 \in \GP _{(\X, \Y)}.$ By doing the same with $G_0$ and repeating this procedure, we can construct an exact sequence 
 $0 \to G \to W_0 \to W_1 \to W_2 \to  \cdots,$ with $W_i \in \omega$ and 
 $\Ima (W_i \to W _{i+1}) \in \GP _{(\X, \Y)}\subseteq {}^\perp\Y\subseteq {}^\perp\omega,$
 for any non negative integer $i.$ 
\end{dem}
\vspace{0.2cm}

The following result generalizes \cite[Proposition 2.9 and Corollary 2.10]{BGRO}.
\begin{teo}\label{thickcat} 
Let $\omega \subseteq \Y \subseteq \A,$ where $\omega$ is closed under finite coproducts. Then, the class  $W \GP _{(\omega, \Y)}$ is left thick.
\end{teo}
\begin{dem} We will carry out the proof of the theorem by following several steps.
\

(i) $W \GP _{(\omega,\Y)}$ is closed under extensions.

Let  $0\to A \to B \to C \to 0$ be an exact sequence, with $A ,C \in W \GP _{(\omega, \Y)} \subseteq {}^{\perp}\Y$. Then $B \in{}^{\perp} \Y $ and there 
exist the following exact sequences
\begin{eqnarray*}
\eta : 0 \rightarrow A \stackrel{\epsilon }{\longrightarrow}  W_0   \longrightarrow L \to 0,\;\;\;\\
\;\;\; \eta ' :0 \to C \stackrel{\xi }{\longrightarrow}  W_0 '  \longrightarrow K \to 0,
\end{eqnarray*}
with $W_0 , W_0 ' \in{\omega} $ and $L,K \in W \GP _{(\omega,\Y)} \subseteq {}^{\perp} \Y.$ 
Consider the following exact and commutative diagram
$$\xymatrix{&0\ar[d]  & 0\ar[d] &  &\\
  0\ar[r]  & A \ar[r]^{\alpha} \ar[d]^{\epsilon}  & B \ar[r]^{\beta} \ar[d]^{\epsilon '} & C \ar[r] \ar@{=}[d] & 0\\
\gamma : 0\ar[r] & W_0 \ar[d] \ar[r]^{\alpha '} & U \ar[r]\ar[d] & C \ar[r]  & 0\\
& L \ar[d] \ar@{=}[r]& L \ar[d] &   &\\
& 0 & \; 0. & & }$$
Since $C \in W \GP _{(\omega,\Y )} \subseteq {}^{\perp} \Y \subseteq{} ^{\perp} \omega,$ it follows that $\Ext _\A^1 (C, W_0 ) = 0.$ 
Thus  $\gamma$ splits and so $U = W_0 \oplus C.$ By using Snake's Lemma and the exact sequence $  0 \to W_0 \oplus C \stackrel{\delta }{\longrightarrow}  W_0 \oplus W_0 '  \longrightarrow K \to 0,$ where $\delta := 1_{W_0} \oplus \xi $ and $W_0 \oplus W_0 ' \in \omega,$ we get the 
following exact and commutative diagram
$$\xymatrix{&  & 0\ar[d] &  0\ar[d]&\\
  0\ar[r]  & B \ar[r]^{\epsilon '\;\;\;\;\;\;}  \ar@{=}[d]  & W_0 \oplus C \ar[d]^{\delta} \ar[r] & L \ar[r] \ar[d] & 0\\
0\ar[r] & B  \ar[r]^{\delta \epsilon'\;\;\;\;\;\;\;\;} & W_0 \oplus W_0 ' \ar[r]\ar[d] & V \ar[r] \ar[d] & 0\\
&  & K \ar[d] \ar@{=}[r]&  K \ar[d] &\\
&  & 0 & \; 0.& }$$
Furthermore, we have that $V \in{}^{\perp} \Y,$ since $L, K \in{}^{\perp} \Y.$ Then, the exact sequence $0 \to B \to W_0 \oplus W_0 ' \to V \to 0$ 
satisfies that $W_0 \oplus W_0 ' \in \omega $ and  $B, V \in{}^{\perp} \Y $. Moreover, in the exact sequence $0 \to L \to V\to K\to 0,$ we have that 
 $L,K \in W \GP _{(\omega,\Y)}$. We can repeat  this  procedure  with  $V,$ and so we get the desired exact sequence for $B$.
\

(ii) $W \GP _{(\omega,\Y)}$ is closed under direct summands and kernels of epimorphisms between its objects.

Let $\eta: 0 \to A \to B \to C \to 0$ be an exact sequence, with $B \in W \GP _{(\omega,\Y) }.$   Since  $B \in W\GP _{(\omega, \Y)},$ there is an exact sequence $\eta ' : 0 \to B \to W_0 \to E_0 \to 0,$ with $W_0 \in \omega$ and 
 $E_0 \in W \GP _{(\omega,\Y)}.$ Therefore, we get the following exact and commutative diagram
$$\xymatrix{&  &   & 0 \ar[d]&\\
& 0 \ar[d] &  &  C' \ar[d]^{t}&\\
  0\ar[r]  & A \ar[r]  \ar[d]^{f}  & W_0  \ar@{=}[d] \ar[r] & K_0 \ar[r] \ar[d]^{r} & 0\\
0\ar[r] & B  \ar[r] \ar[d]^{g} & W_0 \ar[r] & E_0 \ar[r] \ar[d] & 0\\
& C \ar[d] &  &  0& \\
& \;0, &  &  & }$$
and thus, by Snake's Lemma  $C \simeq C'.$ 
\

Let $C\in W\GP _{(\omega,\Y) }.$  We prove now that $A \in W\GP _{(\omega, \Y)}.$ Indeed,  firstly we have that $A\in{}^\perp\Y.$ Since   
$ C' \simeq C,E_0 \in W \GP _{(\omega,\Y)} ,$ by  (i) and the third 
row of the diagram above, we get that $K_0 \in W \GP _{(\omega,\Y)} .$ Then, the exact sequence $0\to A \to W_0 \to K_0 \to 0$  and  $W_0 \in \omega ,$ imply that $A \in W \GP _{(\omega,\Y)} $.
\

Now, assume that the exact sequence  $\eta$ splits. By the third 
row of the diagram above, we get the following exact sequence 
 \[ 0  \longrightarrow A \oplus C \stackrel{\left(\begin{array}{cc} 1_A & 0 \\0 & t \end{array}\right) } {\longrightarrow} A\oplus K_0  \stackrel{\left(\begin{array}{c} 0  \\ r \end{array}\right) }{\longrightarrow} E_0 \longrightarrow 0.\]
Since $A \oplus C = B \in W \GP _{(\omega,\Y)} $ and $E_0 \in W \GP _{(\omega,\Y)},$ we conclude from (i) that 
$A \oplus K_0 \in W \GP _{(\omega,\Y)} \subseteq {}^{\perp} \Y $ and thus  $K_0 \in{}^{\perp} \Y.$ Using that $A \oplus K_0 \in W \GP _{(\omega,\Y)},$ we obtain an exact and commutative diagram 

$$\xymatrix{&  &   & 0 \ar[d]&\\
& 0 \ar[d] &  &  \overline{C} \ar[d]^{t_1}&\\
  0\ar[r]  & K_0 \ar[r]  \ar[d]  & W_1  \ar@{=}[d] \ar[r] & K_1 \ar[r] \ar[d]^{r_1} & 0\\
0\ar[r] & A \oplus K_0  \ar[r] \ar[d] & W_1 \ar[r] & E_1 \ar[r] \ar[d] & 0\\
& A \ar[d] &  &  0& \\
& \;0, &  &  & }$$
where $E_1 \in W \GP _{(\omega,\Y)} $, $W_1 \in \omega$ and $A\simeq \overline{C}.$ Consider the exact sequence
\[ 0  \longrightarrow A \oplus K_0\stackrel{\left(\begin{array}{cc} t_1 & 0 \\0 & 1_{K_0} \end{array}\right) } {\longrightarrow} K_1\oplus K_0  \stackrel{\left(\begin{array}{c} r_1 \\ 0 \end{array}\right) }{\longrightarrow} E_1 \longrightarrow 0.\]
Since $E_1, A \oplus K_0 \in W \GP _{(\omega,\Y)}$, we get that $K_1 \oplus K_0 \in W \GP _{(\omega,\Y)} \subseteq{}^{\perp} \Y $ and thus  $K_1 \in{}^{\perp} \Y.$ Therefore we can do, with $K_1 \oplus K_0$, the same procedure done with $A\oplus K_0$ in order to get the desired exact 
sequence  $0 \to A \to W_0 \to W_1 \to W_2 \to \cdots .$ 
\end{dem}
\vspace{0.2cm}

Let $(\X,\Y) \subseteq \A ^2.$  It is quite natural to ask what can be obtained, in terms of relative Gorenstein projective objects, if we consider the pairs 
$(\GP _{(\X,\Y)}, \Y)$ and $(W\GP _{(\X,\Y)}, \Y)$ That is, to study the classes $\GP ^2 _{(\X,\Y)} := \GP _{(\GP _{(\X,\Y)} , \Y)}$ and 
$W\GP ^2 _{(\X,\Y)} := W\GP _{(W\GP _{(\X,\Y)} , \Y)}.$ In the following result, we consider  weakly Gorenstein projective objects.

\begin{teo}\label{Wiguales} Let $(\X, \Y)\subseteq \A^2$ be such that $\pd_\Y(\X)=0$ and $\X$ be closed under extensions, and let 
$\omega: = \X \cap \Y$ be  closed under finite coproducts in $\A$ and a relative cogenerator in $\X.$ Then, the class $W\GP_{(\X,\Y)}$ 
is left thick and
$$W \GP _{(\X, \Y)}=W \GP _{(\omega, \Y)} =W\GP ^{2} _{(\X,\Y)}.$$
\end{teo}
\begin{dem}   Let $M \in W\GP_{(\X,\Y)}$. By Proposition \ref{cogenera} (a), there is an exact sequence $0 \to M \to W \to G \to 0,$ with $W \in \omega$ and 
$G \in W\GP_{(\X,\Y)} \subseteq{}^{\perp}\Y. $ By repeating the above with $G$ and so on, it can be seen that $M \in W\GP_{(\omega,\Y)}.$ Now, the inclusion  $\omega \subseteq \X$ give us that  $W\GP _{(\omega, \Y)} \subseteq W\GP _{(\X,\Y)};$ proving that 
$W\GP _{(\omega, \Y)} = W\GP _{(\X,\Y)}.$ In particular, by Theorem \ref{thickcat} we get that the class $W\GP_{(\X,\Y)}$ 
is left thick.
\

 Let $G \in W\GP _{(\X,\Y)}$. By considering the exact sequence $ 0 \to G \xrightarrow{1_G} G \to 0 \to \cdots ,$ we have that 
 $G \in W\GP ^{2} _{(\X,\Y)}.$
 \
 
 Let $M \in W\GP ^2 _{(\X,\Y)}.$ Then,  $M\in{}^\perp\Y$ and  there is an exact sequence  
 $\eta : 0 \to M \xrightarrow{f^0} G^0 \xrightarrow{f^1} G^1 \to\cdots , $ with 
 $ G^{i} \in W\GP _{(\X,\Y)}$ and  $L^{i} : = \Coker (f^i) \in{^{\perp} \Y},$ for any $i \in \mathbb{N}$. Since $G^{0} \in W\GP _{(\X,\Y)},$ there is an exact sequence 
 $0 \to G^{0} \to F^{0} \to K^1 \to 0,$ with $F^{0} \in \X$ and $K^1 \in W\GP_{(\X,\Y)}.$ By the Snake's Lemma, we get the following exact and commutative diagram in   $\A$
$$\xymatrix{&  & 0\ar[d] & 0\ar[d] &\\
 0\ar[r]  & M\ar[r] \ar@{=}[d]  & G^0 \ar[r]\ar[d] & L^0 \ar[r] \ar[d] & 0\\
0\ar[r] & M \ar[r] & F^0 \ar[r]\ar[d] & T^0 \ar[r] \ar[d] & 0\\
& & K^1 \ar@{=}[r]\ar[d] & K^1 \ar[d]  &\\
&  & 0 & \;0.& }$$
Since $K^1 \in W\GP_{(\X,\Y)} \subseteq \; ^{\perp } \Y$ and $L^{0} \in {^{\perp} \Y},$ we obtain $T^{0} \in {^{\perp} \Y}.$ Now, we consider the following  exact and commutative diagram in $\A$
$$\xymatrix{&0\ar[d]  & 0\ar[d] &  &\\
  0\ar[r]  & L^0 \ar[r] \ar[d]  & G^1 \ar[r]\ar[d] & L^1 \ar[r] \ar@{=}[d] & 0\\
0\ar[r] & T^0 \ar[d] \ar[r] & U^1 \ar[r]\ar[d] & L^1 \ar[r]  & 0\\
& K^1 \ar[d] \ar@{=}[r]& K^1 \ar[d] &   &\\
& 0 & \; 0. & & }$$
Since $K^1 , G^1 \in W\GP _{(\X,\Y)}$ and the class $ W\GP _{(\X,\Y)}$ is closed under extensions, it follows that $U^1 \in W\GP_{(\X,\Y)}.$ 
Thus, we can construct  a complex $\xi ^{+}$ as follows 
$$\xymatrix{\xi ^{+} : 0\ar[r]  & T^0 \ar[r]  & U^1 \ar[r]\ar[dr] & G^2 \ar[r]  & G^3 \ar[r]& \cdots\\
0 \ar[r]  & L^0  \ar[r] & G^1 \ar[r] & L^1 \ar[u] \ar[dr]  \ar[r] & 0\\
&  & &0 \ar[u] &\;0,  & }$$
where  $T^0 \in W\GP ^2 _{(\X,\Y)}.$ Thus, we have the exact sequence $0\to M \to F^0 \to T^0 \to 0$ with $F^0 \in \X,$  
$M \in {}^{\perp} \Y$ and $T^0 \in W\GP ^2_{(\X,\Y)}.$ Therefore, we can repeat the above procedure for $T^0$ to obtain an exact sequence  $0 \to M \to F^0 \to F^1 \to F^2 \to \cdots ,$ having cokernels in $^{\perp} \Y$ and $F^i \in \X.$ 
\end{dem}
\vspace{0.2cm}

The following corollary is a generalization of \cite[Proposition 2.16]{BGRO}.

\begin{cor}\label{C1Wiguales} Let $\A$ be an AB4-abelian category, with enough projectives, and let $M\in\A$ be a $\Sigma$-orthogonal object. Then
$W\GP_{\Add(M)}=W\GP^2_{(\Add(M),\Add(M))}.$
\end{cor}
\begin{dem} It follows from the dual of Remark \ref{B1} and Theorem \ref{Wiguales}.
\end{dem}
\vspace{0.2cm}

The following result is a generalization of \cite[Theorem 3.8]{Xu}. 

\begin{teo}\label{iguales}
Let $(\X,\Y)$  be a GP-admissible pair in $\A$ and  $\omega := \X \cap \Y.$ Then, the pair $(\GP _{(\X,\Y)}, \Y)$ is GP-admissible and
 $$W \GP _{(\X, \Y)}=W \GP _{(\omega, \Y)} = \GP_{(\X,\Y)} = \GP ^{2} _{(\X,\Y)}=W\GP ^{2} _{(\X,\Y)}.$$
\end{teo}
\begin{dem}  By Proposition \ref{GP2}, we know that $W \GP _{(\X, \Y)}=\GP _{(\X, \Y)}.$  Furthermore, from Theorem \ref{Wiguales}, 
we have that $W \GP _{(\X, \Y)}=W \GP _{(\omega, \Y)} =W\GP ^{2} _{(\X,\Y)}.$ Then, by Proposition  \ref{GP2}, in order to prove that 
$W\GP ^{2} _{(\X,\Y)}=\GP ^{2} _{(\X,\Y)},$ it is enough to show that the pair $(\GP _{(\X,\Y)}, \Y)$ is GP-admissible.
\

Let us prove that  $(\GP _{(\X,\Y)}, \Y)$ is GP-admissible. Indeed,   by Theorem \ref{thickcat}, it follows that $\GP _{(\X,\Y)}$ is closed under extensions. On the other hand, the fact that $\X$ is $\X$-epic  in $\A$ and 
$\X \subseteq \GP _{(\X,\Y)}$ give us that $\GP_{(\X,\Y)}$ is $\GP_{(\X,\Y)}$-epic in $\A$. We also know that 
$\GP_{(\X,\Y)} \subseteq {}^{\perp} \Y$ and so $\pd _{\Y} (\GP_{(\X,\Y)})=0.$ 
\

Now, we show  that $\GP_{(\X,\Y)} \cap \Y$ is a relative cogenerator in $\GP_{(\X,\Y)}.$ Indeed, let $G \in \GP_{(\X,\Y)}.$ By Lemma \ref{oobss2}, there is an exact sequence $0\to G \to X \to G' \to 0,$ with $X \in \X $ and $G' \in \GP_{(\X,\Y)}.$ Since $\X \cap \Y$ is a relative cogenerator in $\X,$ 
there is an exact sequence $0\to X \to E \to X' \to 0,$ with $X' \in \X$ and $E \in \X \cap \Y.$ Then, we have the following exact and commutative diagram 
$$\xymatrix{&  & 0\ar[d] & 0\ar[d] &\\
 0\ar[r]  & G\ar[r] \ar@{=}[d]  & X \ar[r]\ar[d] & G' \ar[r] \ar[d] & 0\\
0\ar[r] & G \ar[r] & E \ar[r]\ar[d] & Q \ar[r] \ar[d] & 0\\
& & X' \ar@{=}[r]\ar[d] & X' \ar[d]  &\\
&  & 0 & \;0.& }$$
Since $G' \in \GP_{(\X,\Y)},$  $X' \in \X \subseteq \GP_{(\X,\Y)}$ and $\GP_{(\X,\Y)}$ is closed under extensions, it follows that $Q \in \GP_{(\X,\Y)}.$ 
Therefore, for the exact sequence  $0 \to G \to E \to Q \to 0,$ we have that $Q \in \GP_{(\X,\Y)}$ and $E \in \X \cap \Y \subseteq \GP_{(\X,\Y)} \cap \Y;$ proving that $(\GP _{(\X,\Y)}, \Y)$ is GP-admissible.
\end{dem}

The following result generalizes \cite[Proposition 2.7]{Xu}.

\begin{cor}\label{ThickGP} If $(\X,\Y)$  is a GP-admissible pair in $\A,$ then $\GP_{(\X,\Y)}$ is left thick.
\end{cor}
\begin{dem} It follows from Theorem \ref{thickcat} and Theorem \ref{iguales}.
\end{dem}

\begin{teo}\label{CThickGP} For a GP-admissible pair $(\X,\Y)$  in $\A$ and  $\omega := \X \cap \Y,$ the following statements hold true. 
\begin{itemize}
\item[(a)] The pairs  $(\GP _{(\X,\Y)}, \Y),$ $(\GP_{(\X,\Y)}, \omega )$ and $(\X,\Y^\wedge)$  are  GP-admissible. 
\item[(b)] Let $\Y$ be closed under direct summands in $\A$ and $\Y\subseteq \GP _{(\X,\Y)}.$ Then, the pair $(\GP _{(\X,\Y)}, \Y)$ is left Frobenius, the elements of $\Y$ are direct summands of objects in $\X$ and $\Y^\wedge$ is right thick.
\item[(c)] If $\omega$ is closed under direct summands in $\A,$ then $\GP _{(\X,\Y)} \cap \Y = \omega.$
\item[(d)] If $\Y^\wedge,$ $\omega$ and $\X\cap\Y^\wedge$ are closed under direct summands in $\A,$ then 
$$\GP_{(\X,\Y)}=\GP_{(\X,\Y^\wedge)}\quad\text{and}\quad\GP _{(\X,\Y)} \cap \Y^\wedge = \omega=\X\cap\Y^\wedge.$$
\end{itemize}
\end{teo}
\begin{dem} (a) The fact that $(\GP _{(\X,\Y)}, \Y)$ is GP-admissible was shown in Theorem \ref{iguales}.
\

Let us prove that $(\GP_{(\X,\Y)}, \omega )$ is  GP-admissible. Since $(\X,\Y)$ is GP-admissible, we get that $\omega := \X \cap \Y$ is closed 
under finite coproducts in $\A.$ Moreover, since $(\GP _{(\X,\Y)}, \Y)$ is GP-admissible,  we have in particular that $\GP_{(\X,\Y)}$ is closed 
under extensions and finite coproducts in 
$\A,$  and it is 
$\GP_{(\X,\Y)}$-epic in $\A.$ By Corollary \ref{GPcogenera} (a) and the equality $\GP_{(\X,\Y)} \cap \omega = \omega,$ we get that 
$\GP_{(\X,\Y)} \cap \omega$ is a relative cogenerator in $\GP_{(\X,\Y)}.$ On the other hand,  
$\GP_{(\X,\Y)} \subseteq {}^{\perp} \Y \subseteq {} ^{\perp} \omega$ implies that $\pd _{\omega} (\GP_{(\X,\Y)}) = 0.$ Thus, it follows 
that $(\GP_{(\X,\Y)}, \omega )$ is  GP-admissible.
\

Let us show that $(\X,\Y^\wedge)$  is  GP-admissible. Indeed, by the dual of Lemma \ref{AB1} we get that $\pd_{\Y^\wedge}(\X)=0.$ Furthermore, 
$\X\cap\Y^\wedge$ is a relative cogenerator in $\X,$ since $\omega\subseteq \X\cap\Y^\wedge$ and $\omega$ is a relative cogenerator in $\X.$ Finally,  $\Y^\wedge$ is closed under finite coproducts in $\A,$ since $\Y$ has this property.
\

(b) Assume that $\Y$ is closed under direct summands in $\A$ and $\Y\subseteq \GP _{(\X,\Y)}.$ Let us prove that $(\GP _{(\X,\Y)}, \Y)$ is left Frobenius. By Corollary \ref{ThickGP} we have 
that $\GP _{(\X,\Y)}$ is left thick, and by Proposition \ref{GP2} we get that $\Y$ is $\GP _{(\X,\Y)}$-injective. Moreover,  Corollary  \ref{GPcogenera} (a) and 
the inclusion  $\omega\subseteq \Y$ imply that  $\Y$ is a relative cogenerator in $\GP _{(\X,\Y)}.$ Once we have that $(\GP _{(\X,\Y)}, \Y)$ is left Frobenius, we conclude from Proposition \ref{AB12} that $\Y^\wedge$ is right thick. Finally, let $Y\in\Y.$ Then there is an exact sequence 
$\varepsilon:\;0\to Y\to X\to G\to 0,$ where $X\in\X$ and $G\in\GP _{(\X,\Y)}.$ Since $\pd_\Y(\GP _{(\X,\Y)})=0,$ by Proposition \ref{GP-WGP} (b), we 
have that $\varepsilon$ splits and so $Y$ is a direct summand of $X.$
\

(c) Let $\omega$ be closed under direct  summands in $\A.$ Then, by Corollary \ref{GPcogenera} (b) $\omega=\Y\cap W\GP_{(\omega,\Y)}.$ Moreover, from Theorem \ref{iguales}, we know that $W\GP_{(\omega,\Y)}=\GP_{(\X,\Y)}$ and thus (c) follows.
 \
 
 (d) Let $\Y^\wedge,$ $\omega$ and $\X\cap\Y^\wedge$ be closed under direct summands in $\A.$ By (a), we know that $(\X,\Y^\wedge)$  is  
 GP-admissible. By applying the item (c) to the pair $(\X,\Y^\wedge)$ and by Remark \ref{WGorro} and 
 Proposition \ref{GP2}, we get that $\GP_{(\X,\Y)}=\GP_{(\X,\Y^\wedge)}$ and $\GP_{(\X,\Y)}\cap\Y^\wedge=\X\cap\Y^\wedge.$ 
 \
 
 Since $(\X,\Y^\wedge)$  is  GP-admissible, it follows from (a) that the pair $(\GP_{(\X,\Y^\wedge)}, \Y^\wedge)$ is GP-admissible. Then, by Corollary \ref{GPcogenera} (a),  we get  that $\GP_{(\X,\Y^\wedge)} \cap \Y^\wedge$  is a relative cogenerator in  $\GP _{(\GP_{(\X,\Y^\wedge)}, \Y^\wedge)}$ which is $\GP _{(\GP_{(\X,\Y^\wedge)}, \Y^\wedge)}$-injective . But, from Teorem \ref{iguales}, we know that 
$$\GP _{(\GP_{(\X,\Y^\wedge)}, \Y^\wedge)}  = \GP _{(\X, \Y^\wedge)}.$$
Then, by the equality $\GP_{(\X,\Y)}=\GP_{(\X,\Y^\wedge)},$ we obtain that $\GP_{(\X,\Y)} \cap \Y^\wedge$  is a relative cogenerator in $\GP_{(\X,\Y)}$  which is $\GP_{(\X,\Y)}$-injective. Thus, from
 \cite[Proposition 2.7]{BMPS} it follows that $\GP_{(\X,\Y)} \cap \Y^\wedge=\omega = \GP_{(\X,\Y)} \cap \Y.$
\end{dem}

\begin{cor}\label{CThickGP1} Let $(\X,\Y)$ be a GP-admissible pair in $\A$ such that $\X$ and $\Y$ are closed under direct summands in 
$\A$ and $\Y\subseteq \GP _{(\X,\Y)}.$ Then, for $\omega := \X \cap \Y,$ the following statements hold true.
\begin{itemize}
\item[(a)] The pair  $(\GP _{(\X,\Y)}, \Y^\wedge)$  is a  $\GP_{(\X,\Y)}^\wedge$-cotorsion pair in $\A$ and $\Y\subseteq \X.$ 
\item[(b)] $\GP _{(\X,\Y)} \cap \Y=\GP _{(\X,\Y)} \cap \Y^\wedge = \omega=\X\cap\Y^\wedge.$
\item[(c)]  $\GP _{(\X,\Y)}=\GP _{(\X,\Y)}^\wedge\cap {}^\perp(\Y^\wedge)$ and $\Y^\wedge=\GP _{(\X,\Y)}^\wedge\cap \GP _{(\X,\Y)}^\perp.$ 
\end{itemize}
\end{cor}
\begin{dem} It follows from Theorem \ref{CThickGP} and  \cite[Proposition 2.14 and Theorem 3.6]{BMPS}.
\end{dem}

\begin{cor}\label{CThickGP2} Let $\A$ be an abelian category with enough projectives, and let $\Proj(\A)\subseteq\omega\subseteq\A$ 
 be such that $\add(\omega)=\omega$ and $\id_\omega(\omega)=0.$ Then, the following statements hold true.
\begin{itemize}
\item[(a)] The pair  $(\GP _\omega, \omega^\wedge)$  is a  $\GP_\omega^\wedge$-cotorsion pair in $\A.$  
\item[(b)] $\GP _\omega \cap \omega^\wedge=\omega.$
\item[(c)]  $\GP _\omega=\GP _\omega^\wedge\cap {}^\perp(\omega^\wedge)$ and $\omega^\wedge=\GP _\omega^\wedge\cap \GP _\omega^\perp.$ 
\end{itemize}
\end{cor}
\begin{dem} By the given hypothesis, we have that the pair $(\omega,\omega)$ satisfies the needed conditions to apply Corollary \ref{CThickGP1}.
\end{dem}
\vspace{0.5cm}

Finally, we close this section by giving sufficient conditions in order to construct  strong left Frobenius pairs from GP-admissible pairs in abelian categories. The importance of the strong left Frobenius pairs relies on the fact that they give us exact model structures on exact categories \cite[Section 4]{BMPS}.

\begin{lem}\label{SLF1} For a  GP-admissible pair $(\X,\Y)$ in an abelian category $\A,$ $\B:=\GP_{(\X,\Y)}\cap\Y^\perp$ and $\omega:=\X\cap\Y,$ the 
following statements hold true.
\begin{itemize}
\item[(a)] If $\omega$ is a relative generator in $\X,$ then $\omega$ is a relative generator in $\GP_{(\X,\Y)}.$
\item[(b)] If $\id_\Y(\omega)=0,$ then $\omega$ is a relative cogenerator in $\B.$
\item[(c)] If $\GP_{(\X,\Y)}\subseteq \Y^{\perp_1},$ then $\B$ is closed under kernels of epimorphisms between its objects.
\item[(d)]  If $\id_\Y(\omega)=0$ and  $\omega$ is a relative generator in $\X,$ then $\omega$ is a relative generator in $\B.$ 
\end{itemize}
\end{lem}
\begin{dem} (a) Let $\omega$ be a relative generator in $\X.$ Consider $G\in \GP_{(\X,\Y)}.$ Then, there is an exact sequence $0\to G'\to X\to G\to 0,$ where $X\in\X$  and $G'\in \GP_{(\X,\Y)}.$ Using now that $\omega$ is a relative generator in $\X,$ there is an exact sequence 
$0\to X'\to W\to X\to 0,$ where $X'\in\X$ and $W\in\omega.$ From the preceding two exact sequences and the pull-back construction, we get the following exact and commutative diagram
$$\xymatrix{&  0\ar[d] & 0\ar[d] &  &\\
   & X'  \ar[d]\ar@{=}[r]  & X ' \ar[d] & & \\
0\ar[r] & L \ar[r] \ar[d] & W \ar[r]\ar[d] & G \ar[r] \ar@{=}[d] & 0\\
 0 \ar[r] &G'  \ar[r] \ar[d]& X \ar[r] \ar[d] & G  \ar[r] & 0\\
& 0 & 0. & & }$$
Since $X',G'\in \GP_{(\X,\Y)},$ it follows from Corollary \ref{ThickGP} that $L\in \GP_{(\X,\Y)}.$ Thus, we have an exact sequence $0\to L\to W\to G\to 0,$ with $L\in \GP_{(\X,\Y)}$ and $W\in\omega,$ proving (a).
\

(b) Let $\id_\Y(\omega)=0.$ Consider $B\in\B.$ Then, by Corollary \ref{GPcogenera} (a),  there is an exact sequence $0\to B\to W\to G\to 0,$ with 
 $W\in\omega$ and $G\in\GP_{(\X,\Y)}.$ Note that $B,W\in\Y^\perp$ and thus $G\in\B,$ proving (b).
 \
 
 (c) Let $\GP_{(\X,\Y)}\subseteq \Y^{\perp_1}.$  Let us show that $\B$ is closed under kernels of epimorphisms between its objects. Consider an exact sequence $0\to K\to B'\to B\to 0,$ where 
 $B,B'\in\B;$ and hence by Corollary \ref{ThickGP} $K\in \GP_{(\X,\Y)}.$ Moreover, for any $Y\in\Y,$ we get the exact sequence 
 $$\Ext^i_\A(Y,B)\to \Ext^{i+1}_\A(Y,K)\to \Ext^{i+1}_\A(Y,B').$$
 From the preceding exact sequence and the facts that $B,B'\in\Y^\perp$ and $\GP_{(\X,\Y)}\subseteq \Y^{\perp_1},$ it follows that 
 $K\in\Y^\perp;$ proving that $K\in\B.$
 \
 
 (d) Let $\id_\Y(\omega)=0$ and $\omega$ be a relative generator in $\X.$  We show that $\omega$ is a relative generator 
 in $\B.$ Indeed, let $B\in\B.$ In particular, $B\in \GP_{(\X,\Y)}$ and thus by (a), there is an exact sequence $0\to G\to W\to B\to 0,$ where 
 $W\in\omega$ and $G\in \GP_{(\X,\Y)}.$ Then, for any $Y\in\Y,$ we get the exact sequence 
 $$\Ext^i_\A(Y,B)\to \Ext^{i+1}_\A(Y,G)\to \Ext^{i+1}_\A(Y,W).$$
 From the above exact sequence and since  $B\in\Y^\perp,$ $\id_\Y(\omega)=0$ and $\GP_{(\X,\Y)}\subseteq \Y^{\perp_1},$ we get that 
 $G\in\Y^\perp;$ proving that $\omega$ is a relative generator in $\B.$ 

\end{dem}

\begin{pro}\label{SLF2} For a  GP-admissible pair $(\X,\Y)$ in an abelian category $\A,$  such that $\GP_{(\X,\Y)}\subseteq \Y^{\perp_1},$ $\B:=\GP_{(\X,\Y)}\cap\Y^\perp$ and a relative generator $\omega:=\X\cap\Y$ in $\X,$  the 
following statements hold true.
\begin{itemize}
\item[(a)] If $\id_\Y(\omega)=0$ and $\omega$ is closed under direct summands in $\A,$ then $(\B,\omega)$ is a strong left Frobenius pair.
\item[(b)] If $\Y$ is closed under direct summands in $\A$ and $\Y\subseteq \GP_{(\X,\Y)},$  then $(\B,\Y)$ is a strong left Frobenius pair and 
$\id_\Y(\omega)=0.$
\end{itemize}
\end{pro}
\begin{dem} By Corollary \ref{ThickGP}, it is clear that $\B$ is closed under extensions and direct summands in $\A.$
\

 (a) Let $\omega$ be closed under direct summands in $\A.$ Since $\omega\subseteq\X\subseteq \GP_{(\X,\Y)}$ and $\id_\Y(\omega)=0,$ we get 
 that $\omega\subseteq\B.$ Moreover $\B\subseteq \GP_{(\X,\Y)}\subseteq{}^\perp\Y\subseteq{}^\perp\omega$ and thus, by Lemma \ref{SLF1} (b), 
 $\omega$ is a $\B$-injective relative cogenerator in $\B.$ Furthermore, by Lemma \ref{SLF1} (c) and (d), $\B$ is closed under kernels of epimorphisms between its objects and $\omega$ is a relative generator in $\B.$ Finally, $\pd_\B(\omega)=0$ since $\B\subseteq\Y^\perp\subseteq\omega^\perp.$
 \
 
 (b) Since $\Y\subseteq \GP_{(\X,\Y)}\subseteq{}^\perp\Y,$ it follows that $\Y\subseteq \GP_{(\X,\Y)}^\perp\subseteq \Y^\perp.$ Therefore 
 $\Y\subseteq\B.$ Moreover, using that $\omega\subseteq\Y\subseteq\Y^\perp,$ we get  $\id_\Y(\omega)=0.$
 \
 
 By Lemma \ref{SLF1} (b) and (d), and since $\omega\subseteq\Y,$ we get that $\Y$ is a relative generator and 
 cogenerator in $\B.$ Moreover, by  Lemma \ref{SLF1} (c), it follows that $\B$ is closed under kernels of epimorphisms between its objects. Note that $\Y$ is $\B$-injective, since 
 $\id_\B(\Y)=\pd_\Y(\B)\leq\pd_\Y(\GP_{(\X,\Y)})=0.$ Finally, $\Y$ is $\B$-projective since $\B\subseteq\Y^\perp.$
\end{dem}

As we have seen in Proposition \ref{SLF2}, the inclusion   $\GP_{(\X,\Y)}\subseteq \Y^{\perp_1}$ plays an important role in the proof of such result.  In what follows, we give sufficient conditions on a GP-admissible pair $(\X,\Y)$ to get that $\GP_{(\X,\Y)}\subseteq \Y^{\perp}.$ Note firstly that, by Lemma \ref{SLF1} (c), we should have at least that $\GP_{(\X,\Y)}\cap\Y^\perp$ be closed under kernels of epimorphism between its objects.

\begin{pro} Let $(\X,\Y)$ be a GP-admissible pair in an abelian category $\A,$ and let $\omega:=\X\cap\Y$ satisfying the following conditions:
\begin{itemize}
\item[(a)] $\B:=\GP_{(\X,\Y)}\cap\Y^\perp$ is closed under kernels of epimorphism between its objects;
\item[(b)] $\id_\Y(\omega)=0$ and $\omega$ is closed under direct summands in $\A;$
\item[(c)] $\id_{\GP_{(\X,\Y)}}(G)<\infty$ $\forall\,G\in\GP_{(\X,\Y)}.$
\end{itemize}
Then $\GP_{(\X,\Y)}\subseteq\omega^\vee\subseteq \Y^{\perp}.$ 
\end{pro}
\begin{dem} By Corollary \ref{GPcogenera} (a), we know that $\omega$ is a $\GP_{(\X,\Y)}$-injective relative cogenerator in $\GP_{(\X,\Y)}.$ Then, by 
\cite[Lemma 4.3]{AuB}, we get that 
$\GP_{(\X,\Y)}\cap\omega^\vee=\{G\in\GP_{(\X,\Y)}\;:\;\id_{\GP_{(\X,\Y)}}(G)<\infty\}$ and $\id_{\GP_{(\X,\Y)}}(G)=\coresdim_\omega(G),$ for any 
$G\in\GP_{(\X,\Y)}\cap\omega^\vee.$ Then, by (c), we conclude that $\GP_{(\X,\Y)}\cap\omega^\vee=\GP_{(\X,\Y)}.$ Thus, $\GP_{(\X,\Y)}\subseteq\omega^\vee.$
\

Let $M\in\omega^\vee.$ Then, $n:=\coresdim_\omega(M)$ is finite. In particular, there is an exact sequence $0\to M\to W_0\to W_1\to\cdots\to W_n\to 0,$ where $W_i\in\omega$ for each $i.$ Since $\omega\subseteq \GP_{(\X,\Y)}$ and by (b) we know that $\omega\subseteq \Y^\perp,$ we get that 
$W_i\in\B$ for each $i.$ Thus, by condition (a), we obtain that $M\in\B\subseteq \Y^\perp.$
\end{dem}

\section{Relative Gorenstein homological dimensions}

In this section we develop, in an unified way, the theory of the relative Gorenstein homological dimensions. For each pair of classes of objects in an abelian category $\A,$ satisfying certain natural conditions, we stablish relationships between different kinds of relative homological dimensions, namely: (weak) relative Gorenstein projective, relative projective and resolution dimensions. By taking different pairs of classes of objects in $\A,$ as an application, we obtain the well known results which hold true in each particular classical case. 
\

In the following result, the equality $\resdim_\omega(K)=-1$ just means that $K=0.$ This theorem generalizes 
\cite[Proposition 3.9]{ChZ}, \cite[Theorem 2.10]{Holm}, \cite[Theorem 3.11 and  Propositon 3.16]{MP}, 

\begin{teo}\label{GPAB4} Let $(\X,\Y)$ be a GP-admissible pair in abelian category $\A,$  and let
$\omega:=\X\cap\Y.$ Then,  for any  $C\in\A$ with $\Gpd_{(\X,\Y)}(C)=n<\infty,$ the following statements hold true. 
\begin{itemize}
\item[(a)] There exist exact sequences in $\A,$ with $G,G'\in \GP_{(\X,\Y)}$
\begin{center}
$0\to K\to G\xrightarrow{\varphi}C\to 0\;$ and $\;0\to K'\to G'\xrightarrow{\varphi'}C\to 0,$
\end{center}
where $\varphi:G\to C$ is an $\GP_{(\X,\Y)}$-precover, $\resdim_\omega(K)= n-1=\resdim_\X(K')$ and $K\in\GP_{(\X,\Y)}^{\perp}.$
\item[(b)] There exist exact sequences in $\A,$ with $\overline{G},\overline{G}'\in \GP_{(\X,\Y)}$
\begin{center}
$0\to C\xrightarrow{\psi} H\to \overline{G}\to 0\;$ and $\;0\to C\xrightarrow{\psi'} H'\to \overline{G}'\to 0,$
\end{center}
where $\psi:C\to H$ is an $\omega^{\wedge}$-preenvelope, $\max(\resdim_\omega(H),\resdim_\X(H'))\leq n$ and $\overline{G}\in{}^\perp(\omega^{\wedge}).$
\item[(c)] Let $\X$ be $\GP_{(\X,\Y)}$-injective. Then, $\varphi':G'\to C$ is an $\GP_{(\X,\Y)}$-precover and $K'\in\GP_{(\X,\Y)}^{\perp}.$ Moreover, $\psi':C\to H'$ is an $\X^{\wedge}$-preenvelope and $\overline{G}'\in{}^\perp(\X^{\wedge}).$ 
\end{itemize}
\end{teo}
\begin{dem} By Corollary \ref{ThickGP} we know that $\GP_{(\X,\Y)}$ is left thick. Moreover, from Corollary \ref{GPcogenera} (a),  we have that $\omega$ is 
$\GP_{(\X,\Y)}$-injective and a relative cogenerator in $\GP_{(\X,\Y)}.$ On the other hand, by Lemma \ref{GP1}, we have that $\X$ is a relative cogenerator in $\GP_{(\X,\Y)}.$ Then the result follows now by applying twice the Theorem \ref{AB4}.
\end{dem}
\vspace{0.2cm}

In case of the weak Gorenstein projective objects, we have the following result, which is a generalization of \cite[Theorem 3.5]{BGRO}.

\begin{teo}\label{WGPAB4} Let $(\omega,\Y)$ be a  pair in $\A,$ where $\omega\subseteq\Y$ and $\omega$ is closed under finite coproducts in 
$\A.$  Then,  for any  $C\in\A$ with $\WGpd_{(\omega,\Y)}(C)=n<\infty,$ the following statements hold true. 
\begin{itemize}
\item[(a)] There exist an exact sequence in $\A,$ with $G\in W\GP_{(\omega,\Y)}$
\begin{center}
$0\to K\to G\xrightarrow{\varphi}C\to 0\,$ 
\end{center}
where $\varphi:G\to C$ is an $W\GP_{(\omega,\Y)}$-precover, $\resdim_\omega(K)\leq n-1$ and $K\in W\GP_{(\omega,\Y)}^{\perp}.$
\item[(b)] There exist an exact sequence in $\A,$ with $G'\in \GP_{(\X,\Y)}$
\begin{center}
$\;0\to C\xrightarrow{\psi} H\to G'\to 0,$
\end{center}
where $\psi:C\to H$ is an $\omega^{\wedge}$-preenvelope, $\resdim_\omega(H)\leq n$ and 
$G'\in{}^\perp(\omega^{\wedge}).$
\item[(c)] If $\pd_\Y(\omega)=0$ then $\resdim_\omega(K)= n-1,$ $\WGpd_{(\omega,\Y)}(K)\leq n-1$ and $\WGpd_{(\omega,\Y)}(H)\leq n.$
\end{itemize}
\end{teo}
\begin{dem} By Theorem \ref{thickcat} we know that $W\GP_{(\omega,\Y)}$ is left thick. Moreover, from Remark \ref{cogenera3} we have that 
$\omega$ is 
$W\GP_{(\omega,\Y)}$-injective and a relative quasi-cogenerator in $W\GP_{(\omega,\Y)}.$ Then the result follows now by applying Theorem \ref{AB4} 
and the fact that $\WGpd_{(\omega,\Y)}(A)\leq\resdim_\omega(A),$ for any $A\in\A,$ since $\omega\subseteq W\GP_{(\omega,\Y)}.$
\end{dem}

\begin{cor}\label{ThickGP1} Let $(\X,\Y)$  be a GP-admissible pair in $\A,$ and let $\omega:=\X\cap\Y.$ Then, the followings statements are equivalent, for any  $C\in\A$ and $n\geq 0.$ 
\begin{itemize}
\item[(a)] $\Gpd_{(\X,\Y)}(C)\leq n.$ 
\item[(b)]There is an exact sequence $0\to K\to G\to C\to 0,$  with $\resdim_\omega(K)\leq n-1$ and $G\in\GP_{(\X,\Y)}.$
\item[(c)]There is an exact sequence $0\to C\to H\to \overline{G}\to 0,$ with $\resdim_\omega(H)\leq n$ and $\overline{G}\in\GP_{(\X,\Y)}.$
\item[(d)]There is an exact sequence $0\to K'\to G'\to C\to 0,$  with $\resdim_\X(K')\leq n-1$ and $G'\in\GP_{(\X,\Y)}.$
\item[(e)]There is an exact sequence $0\to C\to H'\to \overline{G}'\to 0,$ with $\resdim_\X(H')\leq n$ and $\overline{G}'\in\GP_{(\X,\Y)}.$
\end{itemize}
\end{cor}
\begin{dem} By Corollary \ref{ThickGP} we know that $\GP_{(\X,\Y)}$ is left thick. Moreover, from Corollary \ref{GPcogenera} (a), we have that $\omega$ is a  relative cogenerator in $\GP_{(\X,\Y)}.$ Note that $\omega$ is closed under isomorphisms in $\A,$ since it is closed under finite coproducts in $\A.$  On the other hand, by Lemma \ref{GP1}, we have that $\X$ is a relative cogenerator in $\GP_{(\X,\Y)}.$ Then the result follows now by applying twice the Corollary \ref{AB6}.
\end{dem}

\begin{rk} Let $R$ be a ring and $\A:=\Modu\,(R).$ If $\X=\Proj\,(R)\subseteq\Y,$ then 
\cite[Proposition 3.11]{MP} is a particular case of Corollary \ref{ThickGP1}.
\end{rk}

\begin{defi}\label{WGPadmi}We say that the pair $(\omega, \Y)\subseteq \A ^2$ is  WGP-\textbf{admissible} if $\omega \subseteq \Y,$ $\pd _{\Y} (\omega) =0$ and $\omega$ is closed under finite co-products in $\A.$ Dually, a pair $(\X,\nu)$ is WGI-\textbf{admissible} if $\nu \subseteq \X,$ $\id _{\X} (\nu) =0$ and $\nu$ is closed under finite co-products in $\A.$
\end{defi}

\begin{ex} (1) Let $R$ be a ring and $M,N\in\Modu(R)$ be such that $M$ is $\Sigma$-orthogonal and $N$ is $\Pi$-orthogonal. Then, by Remark 
\ref{B1} and its dual, we have that the pairs $(\Add(M),\Add(M))$ and $(\Prod(N),\Prod(N))$ are both WGP-admissible and WGI-admissible. Note that $\Add(M)$ is a precovering class; and therefore $\Add(M)$ is $\Add(M)$-epic in $\A$ if and only if $\Proj(R)\subseteq\Add(M).$ In particular, if ${}_RR\not\in\Add(M),$ then $(\Add(M),\Add(M))$ is not GP-admissible.
\

(2) Let $(\X,\Y)$ be a hereditary pair of classes of objects in an abelian category $\A$ such that $\X$ and $\Y$ are closed  under finite co-products in $\A.$ Then, for $\omega:=\X\cap\Y,$ we have that $(\omega,\Y)$ is WGP-admissible and $(\X,\omega)$ is WGI-admissible.  
\end{ex}

Note that a pair $(\X,\nu)\subseteq\A^2$ is WGI-admissible in $\A$ if, and only if, the pair $(\nu^{op},\X^{op})$ is  WGP-admissible in $\A^{op}.$ Therefore, any result or notion related with WGP-admissible pairs can be translated in terms of WGI-admissible pairs. These pairs are related with the GP-admissible pairs, as can be seen below.

\begin{rk}\label{debiladmisiblegrueso} If $(\X,\Y)\subseteq  \A^2$ is a GP-admisible pair, then  $(\X \cap \Y, \Y)$ is a WGP-admissible pair and $(\X,\X \cap \Y)$ is a WGI-admissible pair.
 \end{rk}

\begin{cor}\label{WThickGP1} Let $(\omega,\Y)$ be a WGP-admissible pair in $\A.$ Then, the following statements are equivalent, for any  $C\in\A$ and $n\geq 0.$ 
\begin{itemize}
\item[(a)] $\WGpd_{(\omega,\Y)}(C)\leq n.$ 
\item[(b)]There is an exact sequence $0\to K\to G\to C\to 0,$  with $\resdim_\omega(K)\leq n-1$ and $G\in W\GP_{(\omega,\Y)}.$
\item[(c)]There is an exact sequence $0\to C\to H\to \overline{G}\to 0,$ with $\resdim_\omega(H)\leq n$ and 
$\overline{G}\in W\GP_{(\omega,\Y)}.$
\end{itemize}
\end{cor}
\begin{dem} From Theorem \ref{thickcat} we know that $W\GP_{(\omega,\Y)}$ is left thick. By Remark \ref{cogenera3}, we have that 
$\omega$ is 
$W\GP_{(\omega,\Y)}$-injective and a relative cogenerator in $W\GP_{(\omega,\Y)}.$ Then, the result follows now by applying  Corollary \ref{AB6}.
\end{dem}
\vspace{0.2cm}

\begin{cor}\label{WThickGP2}  Let $(\omega, \Y) \subseteq \A ^2$ be a WGP-admissible pair, with $\omega$ closed under direct summands in $\A.$ Then, 
the pair $(W\GP _{(\omega, \Y)}, \omega)$ is left Frobenius and the followings statements are equivalent, for any  $C\in W\GP_{(\omega,\Y)}^{\wedge}$ and $n\geq 0.$ 
\begin{itemize}
\item[(a)] $\WGpd_{(\omega,\Y)}(C)\leq n.$
\item[(b)] If $0\to K_n\to G_{n-1}\to\cdots \to G_1\to G_0\to C\to 0$ is an exact sequence, with $G_i\in W\GP_{(\omega,\Y)},$ then 
$K_n\in W\GP_{(\omega,\Y)}.$
\end{itemize}
\end{cor}
\begin{proof} By Theorem \ref{thickcat}, we know that  $W\GP _{(\omega, \Y)}$ is left thick.  Moreover, from Remark \ref{cogenera3}, we get that $(W\GP _{(\omega, \Y)}, \omega)$ is a left Frobenius pair in $\A.$ Then, the 
equivalence between (a) and (b) follows from Proposition \ref{AB10}.
\end{proof}

\begin{cor}\label{ThickGP2} Let $(\X,\Y)$  be a GP-admissible pair in $\A,$ and let $\omega:=\X\cap\Y$ be closed under direct summands in $\A.$ Then,  the pair $(\GP_{(\X,\Y)},\omega)$  is left Frobenius and the followings statements are equivalent, for any  $C\in\GP_{(\X,\Y)}^{\wedge}$ and $n\geq 0.$ 
\begin{itemize}
\item[(a)] $\Gpd_{(\X,\Y)}(C)\leq n.$
\item[(b)] If $0\to K_n\to G_{n-1}\to\cdots \to G_1\to G_0\to C\to 0$ is an exact sequence, with $G_i\in\GP_{(\X,\Y)},$ then $K_n\in\GP_{(\X,\Y)}.$
\end{itemize}
\end{cor}
\begin{dem} From Remark \ref{debiladmisiblegrueso} and Theorem \ref{iguales}, we get that the pair $(\omega, \Y) $ is  WGP-admissible and 
$W\GP _{(\omega, \Y)}=\GP_{(\X,\Y)}.$  Thus, the result follows now from Corollary \ref{WThickGP2}.
\end{dem}

\begin{cor}\label{ThickGP3} Let $(\X,\Y)$  be a GP-admissible pair in $\A$ and  $\omega:=\X\cap\Y.$ 
\begin{itemize}
\item[(a)] Let $\omega$ be closed under direct summands in $\A.$ Then
$$\pd_{\omega^{\wedge}}(C)=\pd_{\omega}\,(C)=\Gpd_{(\X,\Y)}(C) \quad \forall\, C\in \GP_{(\X,\Y)}^{\wedge}\!.$$
\item[(b)] Let $\X$ be $\GP_{(\X,\Y)}$-injective and  closed under direct summands in $\A.$ Then
$$\pd_{\X^{\wedge}}(C)=\pd_{\X}\,(C)=\Gpd_{(\X,\Y)}(C) \quad \forall\, C\in \GP_{(\X,\Y)}^{\wedge}\!.$$
\end{itemize}
\end{cor}
\begin{dem} (a) By Theorem \ref{thickcat} and Theorem \ref{iguales}, we know that $W\GP_{(\omega,\Y)}=\GP_{(\X,\Y)}$ is left thick. 
By Corollary \ref{GPcogenera} (a), we have that $\omega$ is a $\GP_{(\X,\Y)}$-injective  relative cogenerator in $\GP_{(\X,\Y)}.$ Then, the result follows now by applying the Theorem \ref{AB8} to the pair $(\GP_{(\X,\Y)},\omega).$
\

(b) Since $\GP_{(\X,\Y)}$ is left thick and, from  Lemma \ref{GP1}, we have that $\X$ is a relative cogenerator in $\GP_{(\X,\Y)},$ we get (b) by applying  Theorem \ref{AB8} to the pair $(\GP_{(\X,\Y)},\X).$
\end{dem}

\begin{rk} Let $R$ be a ring and $\A:=\Modu\,(R).$ The following results are particular cases of Corollaries \ref{ThickGP2} and \ref{ThickGP3}:
\begin{itemize}
\item[(1)]  \cite[Theorem 2.20]{Holm} by taking $\X=\Proj\,(R)=\Y,$
\item[(2)]  \cite[Proposition 2.8]{Yang} by taking $\X=\Proj\,(R)$ and $\Y=\Flat\,(R),$
\item[(3)]  \cite[Lemma 3.13 and Proposition 3.14]{MP} and \cite[Theorem 3.10]{ChZ}  by taking $\X=\Proj\,(R)\subseteq\Y.$
\end{itemize}
\end{rk} 

\begin{pro}\label{WThickGP9}  Let $(\omega,\Y)$  be a WGP-admissible pair in the abelian category $\A.$  Then, the following statements hold true.
\begin{itemize}
\item[(a)] If $\Y$ is closed under direct summands in $\A,$ then
 \begin{itemize}
 \item[(a1)] $\WGpd_{(\omega,\Y)}(M)=\pd_\Y(M)=\pd_{\Y^\wedge}(M)$ for any $M\in W\GP_{(\omega,\Y)}^\wedge,$
 \item[(a2)] $W\GP_{(\omega,\Y)}^\wedge\cap {}^\perp\Y=W\GP_{(\omega,\Y)}=W\GP_{(\omega,\Y)}^\wedge\cap {}^\perp(\Y^\wedge).$
 \end{itemize}
\item[(b)] If $\omega$ is closed under direct summands in $\A,$ then
 \begin{itemize}
 \item[(b1)] $\WGpd_{(\omega,\Y)}(M)=\pd_{\omega}\,(M)=\pd_{\omega^{\wedge}}(M) \quad \forall\, M\in W\GP_{(\omega,\Y)}^{\wedge},$
 \item[(b2)] $\WGpd_{(\omega,\Y)}(M)=\resdim_\omega(M)=\pd_\omega(M)=\pd_{\omega^\wedge}(M)$ for any $M\in\omega^\wedge,$
  \item[(b3)] $W\GP_{(\omega,\Y)}\cap\omega^\wedge=\omega$ and $W\GP_{(\omega,\Y)}^\wedge\cap {}^\perp\omega=W\GP_{(\omega,\Y)}=W\GP_{(\omega,\Y)}^\wedge\cap {}^\perp(\omega^\wedge).$
 \end{itemize}
\end{itemize}
\end{pro}
\begin{dem} (a)  Let $\Y$ be closed under direct summands in $\A.$ 
\

(a1) By Theorem \ref{thickcat} we know that $W\GP_{(\omega,\Y)}$ is left thick. In order to get (a1), it is enough to check the conditions in 
Theorem \ref{AB8} for the pair $(W\GP_{(\omega,\Y)},\Y).$ It is clear that $\Y$ is $W\GP_{(\omega,\Y)}$-injective. Moreover, $\Y$ is a relative cogenerator in $W\GP_{(\omega,\Y)},$ since $\omega$ is so (see Remark \ref{cogenera3}). Finally, by hypothesis we know that $\Y$ is closed under direct summands in $\A.$
\

(a2) By Corollary \ref{GP3} (b), we have that $W\GP_{(\omega,\Y)}\subseteq {}^\perp\Y\cap {}^\perp(\Y^\wedge).$ Then, (a2) follows from (a1).
\

(b) Let $\omega$ be closed under direct summands in $\A.$ 
\

(b1) By Corollary \ref{WThickGP2}, we know that  the pair $(W\GP_{(\omega,\Y)},\omega)$  is left Frobenius. Then, the item (b1) follows  by applying  Theorem \ref{AB8} to this left Frobenius pair.
\

(b2) Since $\pd_\Y(\omega)=0,$  we get that $\omega$ is $\omega$-injective and thus it is closed under extensions, since $\omega$ is closed under finite coproducts in $\A.$ Then, from Theorem \ref{AB8}, it follows that $\resdim_\omega(M)=\pd_\omega(M)$ for any $M\in\omega^\wedge.$ Therefore, (b2) follows from  (b1). 
\

(b3) The equality $W\GP_{(\omega,\Y)}\cap\omega^\wedge=\omega$ follows from (b2). On the other hand, from Corollary \ref{GP3} (b), we 
get that  $W\GP_{(\omega,\Y)}\subseteq {}^\perp\omega\cap {}^\perp(\omega^\wedge),$ since ${}^\perp\Y\cap {}^\perp(\Y^\wedge)\subseteq{}^\perp\omega\cap {}^\perp(\omega^\wedge).$ Thus, the equalities $W\GP_{(\omega,\Y)}^\wedge\cap {}^\perp\omega=W\GP_{(\omega,\Y)}=W\GP_{(\omega,\Y)}^\wedge\cap {}^\perp(\omega^\wedge)$ follow from (b1).
\end{dem}

\begin{rk} Corollary \ref{WThickGP2} and Proposition \ref{WThickGP9} generalize \cite[Theorem 3.8]{BGRO} which is given in the context of the  weakly Wakamatsu tilting $R$-modules.
\end{rk}

\begin{cor}\label{ThickGP9}  Let $(\X,\Y)$  be a GP-admissible pair in the abelian category $\A.$  Then, the following statements hold true
\begin{itemize}
\item[(a)] If $\Y$ is closed under direct summands in $\A,$ then
 \begin{itemize}
 \item[(a1)] $\Gpd_{(\X,\Y)}(M)=\pd_\Y(M)=\pd_{\Y^\wedge}(M)$ for any $M\in\GP_{(\X,\Y)}^\wedge,$
  \item[(a2)] $\GP_{(\X,\Y)}^\wedge\cap {}^\perp\Y=\GP_{(\X,\Y)}=\GP_{(\X,\Y)}^\wedge\cap {}^\perp(\Y^\wedge).$
 \end{itemize}
\item[(b)] If $\omega:=\X\cap\Y$ is closed under direct summands in $\A,$ then 
 \begin{itemize}
 \item[(b1)] $\Gpd_{(\X,\Y)}(M)=\pd_{\omega}\,(M)=\pd_{\omega^{\wedge}}(M) \quad \forall\, M\in \GP_{(\X,\Y)}^{\wedge},$
 \item[(b2)] $\Gpd_{(\X,\Y)}(M)=\resdim_\omega(M)=\pd_\omega(M)=\pd_{\omega^\wedge}(M)$ for any $M\in\omega^\wedge,$
  \item[(b3)] $\GP_{(\X,\Y)}\cap\omega^\wedge=\omega$ and $\GP_{(\X,\Y)}^\wedge\cap {}^\perp\omega=\GP_{(\X,\Y)}=\GP_{(\X,\Y)}^\wedge\cap {}^\perp(\omega^\wedge).$
 \end{itemize}
\item[(c)] If $\X$ is $\GP_{(\X,\Y)}$-injective and  closed under direct summands in $\A,$ then
\begin{itemize}
\item[(c1)] $\Gpd_{(\X,\Y)}(M)=\resdim_\X(M)=\pd_\X(M)=\pd_{\X^\wedge}(M)$ for any $M\in\X^\wedge,$
\item[(c2)] $\GP_{(\X,\Y)}\cap\X^\wedge=\X.$
\end{itemize}
\end{itemize}
\end{cor}
\begin{dem} From Remark \ref{debiladmisiblegrueso} and Theorem \ref{iguales}, we get that the pair $(\omega, \Y) $ is  WGP-admissible and 
$W\GP _{(\omega, \Y)}=\GP_{(\X,\Y)}.$  Thus, the  items (a) and (b) follow from Proposition \ref{WThickGP9}.
\

Let us prove the item (c). Assume that $\X$ is $\GP_{(\X,\Y)}$-injective and  closed under direct summands in $\A.$ By Proposition \ref{GP4} (b), we have that  $\X$ is $\X$-injective. Moreover, $\X$ is closed under extensions, since  $(\X,\Y)$  is a GP-admissible. Then, from Theorem \ref{AB8}, it follows that $\resdim_\X(M)=\pd_\X(M)$ for any $M\in\X^\wedge.$ Therefore, Corollary \ref{ThickGP3} (b) gives us (c1). Finally, the item (c2) follows directly from (c1).
\end{dem}
\vspace{0.2cm}

\begin{defi}  For any pair $(\X,\Y)$ of classes of objects in an abelian category $\A,$ we consider the following FINITISTIC homological dimensions.
\begin{itemize}
\item[(1)] The {\bf finitistic  $(\X,\Y)$-Gorenstein projective dimension} of $\A$ 
\begin{center} $\FGPD_{(\X,\Y)}(\A):=\Gpd_{(\X,\Y)}(\GP_{(\X,\Y)}^\wedge).$\end{center}
 In particular, the {\bf finitistic Gorenstein  projective dimension} 
of the abelian category $\A$ is $\FGPD(\A):= \FGPD_{(\Proj(\A),\Proj(\A))}(\A).$
\item[(2)] The {\bf finitistic projective dimension} of $\A$ is $\FPD(\A):=\pd(\Proj(\A)^\wedge).$
\item[(3)] The {\bf weak finitistic  $(\X,\Y)$-Gorenstein projective dimension} of $\A$  
\begin{center}$\WFGPD_{(\X,\Y)}(\A):=\WGpd_{(\X,\Y)}(W\GP_{(\X,\Y)}^\wedge).$\end{center}
\end{itemize}
Similarly, we have  
$\FGID_{(\X,\Y)}(\A)$ which is the {\bf finitistic  $(\X,\Y)$-Gorenstein injective dimension}  of $\A,$  and 
$\WFGID_{(\X,\Y)}(\A)$ which is the {\bf weak finitistic  $(\X,\Y)$-Gorenstein injective dimension} of $\A.$  
\end{defi}

\begin{defi}  For any pair $(\X,\Y)$ of classes of objects in an abelian category $\A,$ we consider the following GLOBAL homological dimensions.
\begin{itemize}
\item[(1)] The {\bf global  $(\X,\Y)$-Gorenstein projective dimension} of $\A$ 
\begin{center} $\glGPD_{(\X,\Y)}(\A):=\Gpd_{(\X,\Y)}(\A).$\end{center}
 In particular, the {\bf global Gorenstein  projective dimension} 
of the abelian category $\A$ is $\glGPD(\A):= \glGPD_{(\Proj(\A),\Proj(\A))}(\A).$
\item[(2)] The {\bf global projective dimension} of $\A$ is $\glPD(\A):=\pd(\A).$
\item[(3)] The {\bf weak global  $(\X,\Y)$-Gorenstein projective dimension} of $\A$  
\begin{center}$\glWGPD_{(\X,\Y)}(\A):=\WGpd_{(\X,\Y)}(\A).$\end{center}
\end{itemize}
Similarly, we have  
$\glGID_{(\X,\Y)}(\A)$ which is the {\bf global  $(\X,\Y)$-Gorenstein injective dimension}  of $\A,$  and 
$\glWGID_{(\X,\Y)}(\A)$ which is the {\bf weak global  $(\X,\Y)$-Gorenstein injective dimension} of $\A.$  
\end{defi}

In case of some ring $R,$ the finitistic  $(\X,\Y)$-Gorenstein 
projective dimension of $R$ is $\FGPD_{(\X,\Y)}(R):=\FGPD_{(\X,\Y)}(\Modu\,R),$  the weak finitistic  $(\X,\Y)$-Gorenstein 
projective dimension of $R$ is $\WFGPD_{(\X,\Y)}(R):=\WFGPD_{(\X,\Y)}(\Modu\,R)$ and the finitistic projective dimension of $R$ is 
$\FPD(R):=\FPD(\Modu\,R).$ We also have the following homological dimensions of the ring $R.$ Namely,  the  {\bf finitistic Ding projective} dimension  $\FDPD(R):= \FGPD_{(\Proj(R),\Flat(R))}(\Modu\,R)$ and the  finitistic Gorenstein  projective dimension $\FGPD(R):= \FGPD(\Modu\,R).$ 
\

We also have the so called relative global dimensions of the ring $R,$ namely, 
\begin{center}$\glGPD_{(\X,\Y)}(R):=\glGPD_{(\X,\Y)}(\Modu\,R)$\end{center}
 which is the global $(\X,\Y)$-Gorenstein 
projective dimension of $R,$ and 
\begin{center}$\glGID_{(\X,\Y)}(R):=\glGID_{(\X,\Y)}(\Modu\,R)$\end{center} which is the global $(\X,\Y)$-Gorenstein 
injective dimension of $R,$ and so on. 

\begin{cor}\label{ThickGP5} Let $(\X,\Y)$  be a GP-admissible pair in $\A,$  such that $\omega:=\X\cap\Y=\Proj\,(\A).$ Then, the following statements hold true.
\begin{itemize}
\item[(a)] $\Gpd_{(\X,\Y)}(M)=\pd(M)$ for any $M\in\Proj\,(\A)^\wedge.$
\item[(b)] $\Gpd_{(\X,\Y)}(M)=\pd_{\omega}\,(M)=\pd_{\omega^{\wedge}}(M) \quad \forall\, M\in \GP_{(\X,\Y)}^{\wedge}.$
\item[(c)] $\GP_{(\X,\Y)}\cap\omega^\wedge=\omega$ and $\GP_{(\X,\Y)}^\wedge\cap {}^\perp\omega=\GP_{(\X,\Y)}=\GP_{(\X,\Y)}^\wedge\cap {}^\perp(\omega^\wedge).$
\item[(d)] $\FPD(\A)\leq\FGPD_{(\X,\Y)}(\A)\leq \glGPD_{(\X,\Y)}(\A)\leq \glPD(\A).$
\end{itemize}
\end{cor}
\begin{dem} Since $\Proj\,(\A)\subseteq\X\subseteq\GP_{(\X,\Y)},$ it follows from Remark \ref{resdimPD}  that 
$\Gpd_{(\X,\Y)}(M)\leq\pd(M),$ for any $M\in\Proj(\A)^\wedge.$ In particular $\Proj\,(\A)^\wedge\subseteq\GP_{(\X,\Y)}^\wedge.$
Thus, the result follows directly from Corollary \ref{ThickGP9} (b).
\end{dem}

\begin{pro}\label{ThickGP5.5} Let $(\X,\Y)$  be a GP-admissible pair in $\A$  such that $\omega:=\X\cap\Y$ is closed under direct summands in $\A$ and $\X$ is left thick. Then 
$$\FGPD_{(\X,\Y)}(\A)\leq\resdim_\X(\X^\wedge).$$
\end{pro}
\begin{dem} Let $\alpha := \resdim _{\X} (\X ^\wedge)$ and $\beta := \mbox{FGPD} _{(\X,\Y )} (\A)$. We may assume that  $\alpha$ is finite.
\

We assert that  $\beta$ is finite.  Indeed, by  Theorem \ref{GPAB4},  for any $M \in \GP _{(\X,\Y)} ^\wedge,$ there exists $H \in \A$ with $\resdim _{\X} (H) = \Gpd _{(\X,\Y)} (M) -1.$ Therefore 
$\Gpd _{(\X,\Y)} (M) \leq \alpha+1$ and thus $\beta$ is finite. Since $\alpha\geq 0,$ in order to prove that $\beta\leq\alpha,$ we may assume that 
$\beta>0.$
\

 Fix some $M  \in \GP _{(\X,\Y)} ^\wedge$ such that $\Gpd _{(\X,\Y)} (M) = \beta.$ By  
Theorem \ref{GPAB4} there is an exact sequence $0 \to K \to G \to M \to 0,$ with $G \in \GP _{(\X,\Y)}$ and $\resdim _{\X} ( K ) = \beta-1.$ Since 
$G \in \GP _{(\X, \Y)},$ there is an exact sequence $0 \to G \to Q \to G' \to 0$ with $Q \in \X$ and $G' \in \GP _{(\X,\Y)}.$ So we get a monomorphism 
 $K \to Q$ and then an exact sequence $0 \to K \to Q \to L \to 0.$  Therefore, we have the following exact and 
 commutative diagram
$$\xymatrix{&0\ar[d]  & 0 \ar[d] & &\\
 & K \ar@{=}[r] \ar[d] & K \ar[d] &  & \\
0\ar[r] & G \ar[r]\ar[d] & Q \ar[r]\ar[d] & G' \ar@{=}[d] \ar[r] & 0\\
   0\ar[r] &  M  \ar[r] \ar[d] & L \ar[d] \ar[r] & G' \ar[r] & 0\\
& 0 & \;0. & &}$$
From the above diagramm, we have the exact sequence $ 0\to M \to L \to G' \to 0,$ where $G' \in \GP _{(\X,\Y)}.$ If  $L \in  \GP _{(\X,\Y)}$ then 
$M \in  \GP _{(\X,\Y)}$ (see Corollary \ref{ThickGP}), contradicting that $ \Gpd _{(\X,\Y)} (M) = \beta > 0.$ Then $L \not \in  \GP _{(\X,\Y)}$ and so 
$L \not \in \X$. Finally, from Proposition \ref{AB10} and the exact sequence  $0 \to K \to Q \to L \to 0,$ we conclude that 
$\resdim _{\X} (L ) = \resdim_{\X} (K) +1=\beta;$ proving that $\beta\leq\alpha.$ 
\end{dem}
\vspace{0.2cm}

In the case of the abelian category $\A:=\Modu(R)$ of left $R$-modules, for some ring $R,$ the following result is a generalization of  \cite[Proposition 2.27, Theorem 2.28]{Holm}.

\begin{cor}\label{ThickGP5.55} Let $\A$ be an abelian category with enough projectives. Then, for $\omega:=\Proj(\A),$ the following statements hold true. 
\begin{itemize}
\item[(a)] $\Gpd (M)=\pd(M)$ for any $M\in\Proj\,(\A)^\wedge.$
\item[(b)] $\Gpd(M)=\pd_{\omega}\,(M)=\pd_{\omega^{\wedge}}(M) \quad \forall\, M\in \GP(\A)^{\wedge}.$
\item[(c)] $\GP(\A)\cap \omega^\wedge=\omega$ and  $\GP(\A)^\wedge\cap {}^\perp\omega=\GP(\A)=\GP(\A)^\wedge\cap {}^\perp(\omega^\wedge).$
\item[(d)]  $\FPD(\A)=\FGPD(\A).$
\item[(e)] If $\glGPD(\A)<\infty$ then $(\GP(\A),\omega^\wedge)$ is a hereditary complete cotorsion pair in $\A.$
\end{itemize}
\end{cor}
\begin{dem} Consider the GP-admissible pair $(\Proj\,(\A),\Proj\,(\A)).$ Then, the items from (a) to (d) follow from Corollary \ref{ThickGP5} and Proposition \ref{ThickGP5.5}. Finally, the item (e) follows from Corollary \ref{CThickGP2}.
\end{dem}

\begin{pro}\label{WThickGP5.5} Let $(\omega,\Y)$  be a WGP-admissible pair in $\A,$  such that $\omega$ is closed under direct summands and kernels of epimorphisms between its objects. Then, the pair $(\omega, \omega)$ is left Frobenius,  $\omega^\wedge$ is a  thick class in $\A$ and 
$$\WFGPD_{(\omega,\Y)}(\A)\leq\resdim_\omega(\omega^\wedge).$$
\end{pro}
\begin{dem} Since $\pd_\Y(\omega)=0$ and $\omega$ is closed under finite coproducts in $\A,$  it follows that $\omega$ is closed under 
extensions. Therefore, $(\omega, \omega)$ is a left Frobenius pair in $\A.$ Moreover, from \cite[Theorem 2.11]{BMPS}, we have that 
$\omega^\wedge$ is a thick class in $\A$. By using  Theorem \ref{WGPAB4}, Theorem \ref{thickcat} and Proposition \ref{AB10}, we can adapt the proof given in Proposition \ref{ThickGP5.5}  to obtain a proof of this result.
\end{dem}
\vspace{0.2cm}

The following result generalizes \cite[Theorem 3.10, Proposition 3.11 and Lemma 5.1]{ChZ}.

\begin{teo}\label{WThickGP5} Let $\Y$ be a class of objects in $\A$  such that 
$\omega:=\Proj\,(\A)\subseteq \Y.$ Then,  the following statements hold true.
\begin{itemize}
\item[(a)] $W\Gpd_{(\omega,\Y)}(M)=\pd(M)$ for any $M\in\Proj\,(\A)^\wedge.$
\item[(b)] $W\Gpd_{(\omega,\Y)}(M)=\pd_{\omega}\,(M)=\pd_{\omega^{\wedge}}(M) \quad \forall\, M\in W\GP_{(\omega,\Y)}^{\wedge}.$
\item[(c)] $W\GP_{(\omega,\Y)}\cap\omega^\wedge=\omega$ and $W\GP_{(\omega,\Y)}^\wedge\cap {}^\perp\omega=W\GP_{(\omega,\Y)}=W\GP_{(\omega,\Y)}^\wedge\cap {}^\perp(\omega^\wedge).$
\item[(d)] $\FPD(\A)=\WFGPD_{(\omega,\Y)}(\A).$
\item[(e)] Let $\Y$ be closed under direct summands in $\A.$ Then 
 \begin{itemize}
  \item[(e1)] $\WGpd_{(\omega,\Y)}(M)=\pd_\Y(M)=\pd_{\Y^\wedge}(M)$ for any $M\in W\GP_{(\omega,\Y)}^\wedge;$
 \item[(e2)] $W\GP_{(\omega,\Y)}^\wedge\cap {}^\perp\Y=W\GP_{(\omega,\Y)}=W\GP_{(\omega,\Y)}^\wedge\cap {}^\perp(\Y^\wedge);$ 
 \item[(e3)]  if $\A$ has enough projectives, then 
 \begin{center} $W\GP_{(\omega,\Y)}=\GP_{(\omega,\Y)},$  $\omega=\Y\cap \GP_{(\omega,\Y)}$ and\end{center} 
 $$\FGPD_{(\omega,\Y)}(\A)=\FPD(\A)=\WFGPD_{(\omega,\Y)}(\A).$$
 \end{itemize}
  \item[(f)] If $\glWGPD_{(\omega,\Y)}(\A)<\infty$ then $(W\GP_{(\omega,\Y)},\omega^\wedge)$ is a hereditary complete cotorsion pair in $\A.$
\end{itemize}
\end{teo}
\begin{dem}  Note that the pair $(\Proj\,(\A),\Y)$ is  WGP-admissible and $\Proj(\A)$ is a left thick class in $\A.$
Since $\Proj\,(\A)\subseteq W\GP_{(\Proj\,(\A),\Y)},$  it follows that 
$$\WGpd_{(\Proj\,(\A),\Y)}(M)\leq\resdim_{\Proj\,(\A)}(M),$$ for any $M\in\A.$ In particular $\Proj\,(\A)^\wedge\subseteq W\GP_{(\Proj\,(\A),\Y)}^\wedge.$
Thus, by Remark \ref{resdimPD} and Proposition \ref{WThickGP9} (b), we obtain (a), (b), (c)  and 
{\small $\FPD(\A)\leq\WFGPD_{(\Proj\,(\A),\Y)}(\A).$}  Moreover, by Proposition \ref{WThickGP5.5}  we have
\begin{center} $\WFGPD_{(\Proj\,(\A),\Y)}(\A)\leq \FPD(\A)$\end{center}
 and thus (d)  holds true. 
 \
 
 Let $\Y$ be closed under  direct summands in $\A.$  Then, by Proposition \ref{WThickGP9} (a), we get (e1) and (e2). Assume now that 
 $\A$ has enough projectives. Then, the pair $(\Proj\,(\A),\Y)$ is  weak GP-admissible, and then by Proposition \ref{GP2} we conclude 
 that $W\GP_{(\omega,\Y)}=\GP_{(\omega,\Y)}.$ Hence by (d), the equalities $\FGPD_{(\omega,\Y)}(\A)=\FPD(\A)=\WFGPD_{(\omega,\Y)}(\A)$ hold true. Finally, from Theorem \ref{CThickGP} (c), we have $\omega=\Y\cap \GP_{(\omega,\Y)};$ proving 
 (e3).
 \
 
 Let $\glWGPD_{(\omega,\Y)}(\A)<\infty.$ In particular, it follows that $W\GP_{(\omega,\Y)}^\wedge=\A.$ On the other hand, from 
 Corollary \ref{WThickGP2},  we get that $(\GP_{(\omega,\Y)},\omega)$ is left Frobenius. Then, the item (f) follows from 
 \cite[Proposition 2.14 and Theorem 3.6]{BMPS}.
\end{dem}

\begin{teo}\label{ThickGP5.6} Let $(\X,\Y)$  be a GP-admissible pair in $\A$  such that $\X\cap\Y=\Proj(\A),$  $\X$ be left thick and 
$\X\subseteq\Proj(\A)^\wedge.$ Then, the following statements hold true.
\begin{itemize}
\item[(a)] $\FGPD_{(\X,\Y)}(\A)=\resdim_\X(\X^\wedge)=\resdim_\X(\Proj(\A)^\wedge)=\FPD(\A).$
\item[(b)] If $\A$ has enough projectives, then $\FGPD_{(\Proj(\A),\Y)}(\A)=\FGPD_{(\X,\Y)}(\A).$
\end{itemize}
\end{teo}
\begin{dem} (a)  Note that $\X^\wedge\subseteq\Proj(\A)^\wedge,$ since $\X\subseteq\Proj(\A)^\wedge.$ Hence, by Proposition \ref{ThickGP5.5} it follows that
$$\FGPD_{(\X,\Y)}(\A)\leq\resdim_\X(\X^\wedge)\leq\resdim_\X(\Proj(\A)^\wedge).$$
On the other hand, the inclusion  $\Proj(\A)\subseteq\X$ and Remark \ref{resdimPD} imply that 
$\resdim_\X(\Proj(\A)^\wedge)\leq\resdim_{\Proj(\A)}(\Proj(\A)^\wedge)=\FPD(\A).$ Moreover, by Corollary \ref{ThickGP5} (d)  $\FPD(\A)\leq \FGPD_{(\X,\Y)}(\A)$ and thus 
$\FGPD_{(\X,\Y)}(\A)=\resdim_\X(\X^\wedge)=\resdim_\X(\Proj(\A)^\wedge)=\FPD(\A).$
\

(b) Since $(\X,\Y)$  is GP-admissible pair in $\A,$  $\X\cap\Y=\Proj(\A)$ and $\A$ has enough projectives, it follows that 
$(\Proj(\A),\Y)$  is GP-admissible and satisfies the same hypothesis as the pair $(\X,\Y)$ does. Then, by (a) we get that  
$\FGPD_{(\Proj(\A),\Y)}(\A)=\FPD(\A).$ 
\end{dem}

\begin{cor}\label{ThickGP8}  For any ring $R$ and $\omega:=\Proj\,(R),$ the following statements hold true.
\begin{itemize}
\item[(a)] $\DP(R)=W\GP_{(\omega,\Flat\,(R))}(\Modu\,R)=\GP_{(\omega,\Flat\,(R)^\wedge)}(\Modu\,R).$
\item[(b)] $\Gpd (M)=\pd(M)=\Dpd(M)$ for any $M\in\Proj\,(R)^\wedge.$
\item[(c)]  For any $M\in \DP(R)^{\wedge},$ we have 
\begin{center}$\Dpd(M)=\pd_{\omega}\,(M)=\pd_{\omega^{\wedge}}(M) =\pd_{\Flat(R)}(M)=\pd_{\Flat(R)^\wedge}(M).$\end{center}
\item[(d)] $\DP(R)\cap\omega^\wedge=\omega=\DP(R)\cap\Flat(R)=\DP(R)\cap\Flat(R)^\wedge.$
\item[(e)] $\DP(R)^\wedge\cap {}^\perp\omega=\DP(R)=\DP(R)^\wedge\cap {}^\perp(\omega^\wedge).$
\item[(f)] $\DP(R)^\wedge\cap {}^\perp\Flat(R)=\DP(R)=\DP(R)^\wedge\cap {}^\perp(\Flat(R)^\wedge);$
\item[(g)]  $\FDPD(R)=\FPD(R)=\FGPD(R)=\WFGPD_{(\Proj\,(R),\Flat\,(R))}(R).$
\end{itemize}
\end{cor}
\begin{dem} By \cite[Proposition 8.4.19]{EJ},  we get that the class $\Flat(R)^\wedge$ is closed under direct summands in $\Modu\,R.$ Then, the result  follows directly from Theorem \ref{WThickGP5}, Theorem \ref{CThickGP} and Corollary \ref{ThickGP5.55}, by taking $\Y:=\Flat\,(R)$ and $\X:=\Proj\,(R)$  in the abelian category 
$\Modu\,R.$ 
\end{dem}

\begin{teo}\label{WThickGP6}  Let $(\omega,\Y)$  be a WGP-admissible pair in $\A,$ with $\omega$ closed under direct summands in 
$\A.$ Then,
$$\WFGPD_{(\omega,\Y)}(\A)=\resdim_\omega(\omega^\wedge)=\pd_\omega(\omega^\wedge)=\pd_{\omega^\wedge}(\omega^\wedge).$$
\end{teo}
\begin{dem} Let $\alpha:=\resdim_\omega(\omega^\wedge)$ and $\beta:=\WFGPD_{(\omega,\Y)}(\A).$ Since $\omega^\wedge\subseteq W\GP_{(\omega,\Y)}^\wedge,$ it follows from Proposition \ref{WThickGP9} (b2) that $\alpha\leq\beta$ and $\alpha=\pd_\omega(\omega^\wedge).$
\

We prove that $\beta\leq\alpha.$ In order to do that, we can assume that $\alpha<\infty.$ We assert that $\beta<\infty.$ Indeed, let 
$C\in W\GP_{(\omega,\Y)}^\wedge$ and $n:=\WGpd_{(\omega,\Y)}(C).$ Then, by Theorem \ref{WGPAB4} there is an exact sequence $0\to H \to T\to C\to 0,$ where 
$\resdim_\omega(H)=n-1.$ Hence  $\WGpd_{(\omega, \Y)}(C)=\resdim_\omega(H)+1\leq\alpha+1<\infty,$ proving that $\beta<\infty.$
\

To conclude the proof of $\beta\leq\alpha,$ it is enough to see the existence of some $L\in\omega^\wedge$ such that $\resdim_\omega(L)=\beta.$ Indeed, since $\beta<\infty,$ there is some $M\in W\GP_{(\omega,\Y)}^\wedge$ with $\WGpd_{(\omega,\Y)}(M)=\beta.$ Then, by Theorem \ref{WGPAB4} there is an exact sequence 
$0\to K\to G\to C\to 0,$ where $G\in W\GP_{(\omega,\Y)}$ and $\resdim_\omega(K)=\beta-1.$ Furthermore, since $\omega$ is a relative cogenerator in $W\GP_{(\omega,\Y)},$  there is an exact 
sequence $0\to G\to W\to G'\to 0$ with $G'\in W\GP_{(\omega,\Y)},$ $W\in\omega$ and $\id_\omega(\omega)=0.$ Then, we get an exact sequence 
$\eta:\;0\to K\to W\to L\to 0.$ Since $\resdim_\omega(K)=\beta-1,$ by $\eta$ we conclude that $\resdim_\omega(L)\leq\beta<\infty.$ Hence, Corollary \ref{ThickGP9} (b2) give us that
\begin{center}
$\pd_\omega(K)=\resdim_\omega(K)=\beta-1\quad\text{and}\quad
\pd_\omega(L)=\resdim_\omega(L).$ 
\end{center}
Applying the functor $\Hom_\A(-,W')$ to the exact sequence $\eta,$ with $W'\in\omega,$ we get the exact sequence 
$$\Ext^i_\A(W,W')\to\Ext^i_\A(K,W')\to\Ext^{i+1}_\A(L,W')\to\Ext^{i+1}_\A(W,W').$$
Since  $\id_\omega(\omega)=0,$ it follows that $\Ext^i_\A(K,W')\simeq\Ext^{i+1}_\A(L,W')$  for any $W'\in\omega$ and $i\geq 1.$ Therefore 
$\resdim_\omega(L)=\pd_\omega(L)=\pd_\omega(K)+1=\beta.$
\end{dem}

\begin{cor}\label{ThickGP6}  Let $(\X,\Y)$  be a GP-admissible pair in $\A,$ and let $\omega:=\X\cap\Y$ be closed under direct summands in 
$\A.$ Then
\begin{center}$\WFGPD_{(\omega,\Y)}(\A)=\FGPD_{(\X,\Y)}(\A)=\resdim_\omega(\omega^\wedge)=\pd_\omega(\omega^\wedge)=\pd_{\omega^\wedge}(\omega^\wedge),$\end{center}
\begin{center}$\WFGID_{(\X,\omega)}(\A)=\coresdim_\omega(\omega^\vee)=\id_\omega(\omega^\vee)=\id_{\omega^\vee}(\omega^\vee).$\end{center}
\end{cor}
\begin{dem} From Remark \ref{debiladmisiblegrueso} and Theorem \ref{iguales}, we get that the pair $(\omega, \Y) $ is  WGP-admissible and $W\GP _{(\omega, \Y)}=\GP_{(\X,\Y)}.$  Thus, the first list of equalities follows from Theorem \ref{WThickGP6}.
\

On the other hand, by Remark \ref{debiladmisiblegrueso}, we have that $(\X,\omega)$ is WGI-admissible and then, by dual of Theorem \ref{WThickGP6}, we get the second list of equalities.
\end{dem}

Note that the following result generalizes \cite[Proposition 2.28]{Holm}.

\begin{cor}\label{ThickGP7}   Let $(\X,\Y)$  be a GP-admissible pair in an abelian category $\A,$  such that $\omega:=\X\cap\Y=\Proj\,(\A).$ Then,
$$\FGPD_{(\X,\Y)}(\A)=\FPD(\A)=\WFGPD_{(\omega,\Y)}(\A)=\pd_\omega(\omega^\wedge)=\pd_{\omega^\wedge}(\omega^\wedge).$$
\end{cor}
\begin{dem} It follows directly from Corollary \ref{ThickGP6}.
\end{dem} 

\begin{defi} Let $\A$ be an abelian category. For any class $\Y\subseteq\A,$  the {\bf $\Y$-finitistic projective dimension} of $\A$ is 
$\FPD_\Y(\A):=\pd_\Y(\Q_\Y^{<\infty}),$ where $\Q_\Y^{<\infty}:=\{M\in\A\;:\;\pd_\Y(M)<\infty\}.$ For a ring $R$ and  $\Y\subseteq\Modu\,R,$ 
the $\Y$-finitistic projective dimension of $R$ is $\FPD_\Y(R):=\FPD_\Y(\Modu\,R).$ Dually, we have the class $\I_\Y^{<\infty},$ 
the {\bf $\Y$-finitistic injective dimension} $\FID_\Y(\A)$ of $\A,$ and the $\Y$-finitistic injective dimension $\FID_\Y(R)$ of the ring $R.$ 
\end{defi}

\begin{lem}\label{ThickGP8.5} Let $(\X,\Y)$  be a pair of classes of objects  in an abelian category $\A,$ with enough projectives, such that 
$\pd_\Y(\X)=0.$ Then, for $\omega:=\X\cap\Y$ and $\mathcal{Z}\in \{\omega, \Y, \omega^\wedge,\Y^\wedge\},$ the following statements hold true.
\begin{itemize}
\item[(a)] If ${}^{\perp}\mathcal{Z}\subseteq\GP_{(\X,\Y)},$ then $\GP_{(\X,\Y)}={}^{\perp}\mathcal{Z}$ and $\Gpd_{(\X,\Y)}(M)\leq \pd_\mathcal{Z}(M)$ for any $M\in\A.$
\item[(b)] If ${}^{\perp}\mathcal{Z}\subseteq W\GP_{(\X,\Y)},$  then $W\GP_{(\X,\Y)}={}^{\perp}\mathcal{Z}$ and $\WGpd_{(\X,\Y)}(M)\leq \pd_\mathcal{Z}(M)$ for any $M\in\A.$
\end{itemize}
\end{lem}
\begin{dem} (a) Assume that ${}^{\perp}\Y\subseteq\GP_{(\X,\Y)}.$ Let $M\in\A.$ Then, by the dual of Proposition \ref{AB7} (c), we have 
\[\Gpd_{(\X,\Y)}(M)=\resdim_{\GP_{(\X,\Y)}}(M)\leq \resdim_{{}^{\perp}\Y}(M)\leq \pd_\Y(M).\]
For the other elections of  $\mathcal{Z},$ the same arguments used in the previous one  work well, since by Corollary \ref{GP3} (b), we know that 
$\GP_{(\X,\Y)}\subseteq {}^{\perp}\Y\cap{}^{\perp}(\Y^\wedge)\subseteq {}^{\perp}\omega\cap{}^{\perp}(\omega^\wedge).$
\

(b) It can be proven in the same way as we did in (a).
\end{dem}

\begin{teo}\label{WThickGP10}  For  a WGP-admissible pair $(\omega,\Y)$   in the abelian category $\A,$  the following statements hold true.
\begin{itemize}
\item[(a)] Let  $\Y$ be closed under direct summands in $\A,$ and $\mathcal{Z}\in \{\Y, \Y^\wedge\}.$ If $\Q_{\mathcal{Z}}^{<\infty}\subseteq W\GP_{(\omega,\Y)}^\wedge,$ then 
\[\WFGPD_{(\omega,\Y)}(\A)=\FPD_{\mathcal{Z}}(\A)\quad\text{and}\quad\Q_{\mathcal{Z}}^{<\infty}=W\GP_{(\omega,\Y)}^\wedge.\]
\item[(b)] Let  $\omega$ be closed under direct summands in $\A,$ and $\mathcal{Z}\in \{\omega, \omega^\wedge\}.$ If $\Q_{\mathcal{Z}}^{<\infty}\subseteq W\GP_{(\omega,\Y)}^\wedge,$ then 
\[\WFGPD_{(\omega,\Y)}(\A)=\FPD_{\mathcal{Z}}(\A)\quad\text{and}\quad\Q_{\mathcal{Z}}^{<\infty}=W\GP_{(\omega,\Y)}^\wedge.\]
\end{itemize}
\end{teo}
\begin{dem} (a) Assume that  $\Q_{\mathcal{Z}}^{<\infty}\subseteq W\GP_{(\omega,\Y)}^\wedge.$   Then, by  Proposition \ref{WThickGP9} (a1) we have $\Q_\mathcal{Z}^{<\infty}=W\GP_{(\omega,\Y)}^\wedge.$ By using that $\Q_\mathcal{Z}^{<\infty}=W\GP_{(\omega,\Y)}^\wedge$  and Proposition \ref{WThickGP9} (a1), we obtain 
\[\WFGPD_{(\omega,\Y)}(\A)=\WFGPD_{(\omega,\Y)}(W\GP_{(\omega,\Y)}^\wedge)=\pd_\mathcal{Z}(\Q_\mathcal{Z}^{<\infty})=\FPD_\mathcal{Z}(\A). \]
(b) It can be proven as in (a).
\end{dem}

\begin{cor}\label{ThickGP10} For a GP-admissible pair  $(\X,\Y)$  in the abelian category $\A,$  and $\omega:=\X\cap\Y,$  
the following statements hold true.
\begin{itemize}
\item[(a)] Let  $\Y$ be closed under direct summands in $\A,$ and $\mathcal{Z}\in \{\Y, \Y^\wedge\}.$ If $\Q_{\mathcal{Z}}^{<\infty}\subseteq \GP_{(\X,\Y)}^\wedge,$ then 
\[\FGPD_{(\X,\Y)}(\A)=\FPD_{\mathcal{Z}}(\A)=\WFGPD_{(\omega,\Y)}(\A)\quad\text{and}\quad\Q_{\mathcal{Z}}^{<\infty}=\GP_{(\X,\Y)}^\wedge.\]
\item[(b)] Let  $\omega$ be closed under direct summands in $\A,$ and $\mathcal{Z}\in \{\omega, \omega^\wedge\}.$ If $\Q_{\mathcal{Z}}^{<\infty}\subseteq \GP_{(\X,\Y)}^\wedge,$ then 
\[\FGPD_{(\X,\Y)}(\A)=\FPD_{\mathcal{Z}}(\A)=\WFGPD_{(\omega,\Y)}(\A)\quad\text{and}\quad\Q_{\mathcal{Z}}^{<\infty}=\GP_{(\X,\Y)}^\wedge.\]
\item[(c)] Let  $\X$ be closed under direct summands in $\A,$ and $\mathcal{Z}\in \{\X, \X^\vee\}.$ If $\I_{\mathcal{Z}}^{<\infty}\subseteq W\GI_{(\X,\omega)}^\vee,$ then 
\[\WFGID_{(\X,\omega)}(\A)=\FID_{\mathcal{Z}}(\A)\quad\text{and}\quad \I_{\mathcal{Z}}^{<\infty}=W\GI_{(\X,\omega)}^\vee.\]
\end{itemize}
\end{cor}
\begin{dem} (a) and (b): From Remark \ref{debiladmisiblegrueso} and Theorem \ref{iguales}, we get that the pair $(\omega, \Y) $ is  WGP-admissible and 
$W\GP _{(\omega, \Y)}=\GP_{(\X,\Y)}.$  Thus, the result follows from Theorem \ref{WThickGP10}.
\

(c) From Remark \ref{debiladmisiblegrueso}, we get that the pair $(\X, \omega) $ is  WGI-admissible. Then, (c) follows from the dual of Theorem \ref{WThickGP10} (a).
\end{dem}

\begin{cor}\label{ThickGP11}  Let $(\X,\Y)$  be a GP-admissible pair in the abelian category $\A,$ with enough projectives, such that 
$\Y$ is closed under direct summands in $\A,$ and let $\omega:=\X\cap\Y=\Proj(\A).$ Then, the following statements hold true.
\begin{itemize}
\item[(a)] Let $\mathcal{Z}\in \{\omega, \Y, \omega^\wedge, \Y^\wedge\}.$ If 
$\Q_{\mathcal{Z}}^{<\infty}\subseteq\GP_{(\X,\Y)}^\wedge$ then 
\begin{center}$\FGPD_{(\X,\Y)}(\A)=\FPD(\A)=\FGPD(\A)=\FPD_{\mathcal{Z}}(\A) =\FGPD_{(\omega,\Y)}(\A)=$\end{center} 
$=W\FGPD_{(\omega,\Y)}(\A)=\pd_{\omega}(\omega^\wedge)=\pd_{\omega^\wedge}(\omega^\wedge).$
\item[(b)] $\omega^\vee=\omega$ and $\I_{\omega}^{<\infty}=\I_{\omega^\vee}^{<\infty}=\A.$
\item[(c)]  $W\GI_{(\X,\omega)}^\vee=W\GI_{(\X,\omega)}.$
\end{itemize}
\end{cor}
\begin{dem} (a) By applying Corollary \ref{ThickGP6} to 
the pair $(\Proj(\A),\Proj(\A)),$ we obtain $\FPD(\A)=\FGPD(\A).$ Thus, the result follows by Corollary \ref{ThickGP7},  Corollary \ref{ThickGP10} and Theorem \ref{WThickGP5} (d).
\

(b) Since $(\X,\Y)$ is a GP-admissible pair, we have that $(\omega,\Y)$ is GP-admissible. Using that $\omega^\perp=\A$ and $\A$ has 
enough projectives, we get that $W\GI_{(\omega,\omega)}=\A.$ Then, by applying Corollary \ref{ThickGP10} to the pair 
$(\omega,\Y),$ the item (b) is true.
\

(c)  By Remark \ref{debiladmisiblegrueso}, we know that the pair $(\X,\omega)$ is WGI-admissible. Then, by the dual of Proposition \ref{WThickGP9} (b3), we have $W\GI_{(\X,\omega)}^\vee\cap \omega^\perp=W\GI_{(\X,\omega)}.$ Therefore, the equality in 
(c) is true,  since $\omega^\perp=\A.$
\end{dem}

\begin{cor}\label{ThickGP12}  For any ring $R$ such that $\DP(R)^\wedge=\Modu\,R,$ and  $\omega:=\Proj\,(R),$ the following statements hold true.
\begin{itemize}
\item[(a)] $\DP(R)={}^\perp\mathcal{Z}$ for any $\mathcal{Z}\in\{\omega, \Flat(R), \omega^\wedge,  \Flat(R)^\wedge\}.$
\item[(b)] $\Q_{\mathcal{Z}}^{<\infty}=\Modu\,(R)$ for any $\mathcal{Z}\in\{\omega, \Flat(R), \omega^\wedge,  \Flat(R)^\wedge\}.$
\item[(c)] For any $\mathcal{Z}\in\{\omega, \Flat(R), \omega^\wedge,  \Flat(R)^\wedge\},$ we have that 
\begin{center}$\WFGPD_{(\omega,\Flat(R))}(R)=\FDPD(R)=\FPD(R)=\FGPD(R)=$\end{center} $=\FPD_{\mathcal{Z}}(R)=
\pd_{\omega}(\omega^\wedge)=\pd_{\omega^\wedge}(\omega^\wedge).$
\end{itemize}
\end{cor}
\begin{dem} The item (a) follows from Corollary \ref{ThickGP8} (d), (e).   It is clear that $\Q_{\mathcal{Z}}^{<\infty}\subseteq \DP(R)^\wedge=\Modu\,(R),$ for any  $\mathcal{Z}\in\{\omega, \Flat(R), \omega^\wedge,  \Flat(R)^\wedge\}.$ Then, by Corollary \ref{ThickGP10} and Corollary \ref{ThickGP11}, we get (b) and (c), by considering the pair $(\Proj(R),\Flat(R)).$ 
\end{dem} 

Let $R$ be any ring. Denote by $\FP(R)$ the class of all finitely presented left $R$-modules. The {\bf FP-injective dimension} of $M\in\Modu\,R$ is 
$\FP\id(M):=\id_{\FP(R)}(M).$ We recall, see \cite{Gi}, that $R$ is a Ding-Cheng ring if $R$ is both left and right coherent and 
$\FP\id({}_RR)=\FP\id(R_R)$ is finite.

\begin{cor}\label{ThickGP13} Let $R$ be a Ding-Chen ring and $\omega:=\Proj\,(R).$ Then 
\begin{itemize}
\item[(a)] ${}^\perp(\Flat(R)^\wedge)=\DP(R)={}^\perp\Flat(R).$
\item[(b)] $\Q_{\Flat(R)^\wedge}^{<\infty}=\DP(R)^\wedge=\Q_{\Flat(R)}^{<\infty}.$ 
\item[(c)] For any $\mathcal{Z}\in\{\Flat(R), \Flat(R)^\wedge\},$   we have that 
\begin{center}$\WFGPD_{(\omega,\Flat(R))}(R)=\FDPD(R)=\FPD(R)=\FGPD(R)=$\end{center} $=\FPD_{\mathcal{Z}}(R)=
\pd_{\omega}(\omega^\wedge)=\pd_{\omega^\wedge}(\omega^\wedge).$
\item[(d)] $W\GI_{(\DP(R),\omega)}=\Flat(R)^\wedge.$  
\end{itemize}
\end{cor}
\begin{dem} We assert that  ${}^\perp(\Flat(R)^\wedge)={}^\perp\Flat(R).$  Indeed, let $M\in\Modu\,(R).$ Then, by the dual of Lemma \ref{AB1}, we have 
$$\pd_{\Flat(R)^\wedge}(M)=\id_M(\Flat(R)^\wedge)=\id_M(\Flat(R))=\pd_{\Flat(R)}(M),$$ 
 proving the assertion. On the other hand, by \cite[Theorem 4.7]{Gi} it follows that $(\DP(R),\Flat(R)^\wedge)$ is a hereditary complete cotorsion pair, and therefore  by Corollary \ref{TGP} (b)
 $\DP(R)={}^\perp(\Flat(R)^\wedge)={}^\perp\Flat(R).$
 \
 
 Consider the pair  $(\Proj(R),\Flat(R)).$ Let $\mathcal{Z}\in\{ \Flat(R),\Flat(R)^\wedge\}.$ Since the equality $\DP(R)={}^\perp\mathcal{Z}$ holds true,  it follows from Lemma \ref{ThickGP8.5} that $\Q_{\mathcal{Z}}^{<\infty}\subseteq \DP(R)^\wedge.$ 
  Then, the items (a), (b) and (c) follow from  Corollary \ref{ThickGP10} (a) and Corollary  \ref{ThickGP11}.
  \
  
  Since $(\DP(R),\Flat(R)^\wedge)$ is a hereditary complete cotorsion pair, we have that $\DP(R)^\perp=\Flat(R)^\wedge.$ Thus, the 
  class  $W\GI_{(\DP(R),\omega)}$ consists of all $M\in \Flat(R)^\wedge$ which admits a projective resolution $P_M$ of $M$ 
such that $\Omega_{P_M}^i(M)\in \Flat(R)^\wedge,$ for any $i\geq 1.$ Since $\Modu(R)$ has enough projectives and $\Flat(R)^\wedge$ 
is resolving, we get (d).
\end{dem} 

\begin{rk}  For a Ding-Cheng ring $R,$ it is proven in \cite[Theorem 3.1]{JW}  that 
$$\glDPD(R)\leq \FP\id({}_RR)+\pd\,(\Flat(R)).$$ 
Thus, for any Ding-Cheng ring $R$ with $\pd\,(\Flat(R))$ finite, it follows that $\DP(R)^\wedge=\Modu\,R.$ 
\end{rk}

\section{Cotorsion pairs and relative Gorenstein projective objects}

 In this section, as before, $\A$ stands for 
an abelian category. We also use, freely, the notation introduced in \cite{BMPS}.
\

We introduce the notion of relative tilting and cotilting pairs in $\A,$ and  show the strongly connection they have with relative cotorsion pairs. We recall from 
\cite{BMPS}, that a $\mathcal{Z}$-cotorsion pair, in the abelian category $\A,$ consists of the following data:  a thick subclass 
$\mathcal{Z}$ of 
$\A$ and a pair of clases of objects $(\F,\G)$ in $\mathcal{Z}$ such that $\F={}^{\perp_1}\G\cap\mathcal{Z},$ $\G=\F^{\perp_1}\cap\mathcal{Z}$
and for any $Z\in\mathcal{Z}$ there are exact sequences $0\to G\to F\to Z\to 0$ and $0\to Z\to G'\to F'\to 0$ with $F,F'\in\F$ and $G,G'\in\G.$

\begin{pro} \label{cotorpair} 
Let $(\omega, \Y)$ be a WGP-admissible pair in $\A$, with $\omega$ closed under direct summands. Then, the following statements hold true.
\begin{itemize}
\item[(a)]  $(W\GP _{(\omega, \Y)}, \omega ^{\wedge})$ is a  $W\GP _{(\omega, \Y)} ^{\wedge}$-cotorsion pair in $\A.$
\item[(b)]  $\omega ^{\wedge} = W\GP _{(\omega, \Y)}^{\perp} \cap W\GP _{(\omega, \Y)}^{\wedge}\;$ and $\;W\GP _{(\omega, \Y)}  = W\GP _{(\omega, \Y)}^{\wedge} \cap {}^{\perp} (\omega ^{\wedge}).$
\item[(c)] If $\glWGPD_{(\omega, \Y)}(\A)$ is finite, then $(W\GP _{(\omega, \Y)}, \omega ^{\wedge})$ is a hereditary complete 
cotorsion pair in $\A.$
\end{itemize}
\end{pro}
\begin{dem} By Corollary \ref{WThickGP2} we have that $(W\GP _{(\omega, \Y)}, \omega)$ is left Frobenius. Then,  (a) and (b)  follow from  \cite[Teorema 3.6 and Proposition 2.14]{BMPS}. Assume now that $\glWGPD_{(\omega, \Y)}(\A)$ is finite. Then 
$W\GP _{(\omega, \Y)}^\wedge=\A$ and thus (c) follows from (a) and (b).
\end{dem}

\begin{cor} \label{Gcotorpair} 
Let $(\X, \Y)$ be a GP-admissible pair in $\A$, with $\omega:=\X\cap\Y$ closed under direct summands. Then, the following statements hold true.
\begin{itemize}
\item[(a)]  $(\GP _{(\X, \Y)}, \omega ^{\wedge})$ is a  $\GP _{(\X, \Y)} ^{\wedge}$-cotorsion pair in $\A.$
\item[(b)]  $\omega ^{\wedge} = \GP _{(\X, \Y)}^{\perp} \cap \GP _{(\X, \Y)}^{\wedge}\;$ and $\;\GP _{(\X, \Y)}  = \GP _{(\X, \Y)}^{\wedge} \cap {}^{\perp} (\omega ^{\wedge}).$
\item[(c)] If $\glGPD_{(\X, \Y)}(\A)$ is finite, then $(\GP _{(\X, \Y)}, \omega ^{\wedge})$ is a hereditary complete 
cotorsion pair in $\A.$
\item[(d)]  $(\omega ^\vee, W\GI _{(\X, \omega)})$ is a  $W\GI _{(\X, \omega)} ^\vee$-cotorsion pair in $\A.$
\item[(e)]  $\omega ^\vee = {}^{\perp}W\GI _{(\X, \omega)} \cap W\GI_{(\X, \omega)}^{\vee}\;$ and $\;W\GI _{(\X, \omega)}  = W\GI _{(\X, \omega)}^{\vee} \cap (\omega ^{\vee})^{\perp} .$
\item[(f)] If $\glWGID_{(\X,\omega)}(\A)$ is finite, then $(W\GI _{(\X,\omega)}, \omega ^{\vee})$ is a hereditary complete 
cotorsion pair in $\A.$
\end{itemize}
\end{cor}
\begin{dem} By Remark \ref{debiladmisiblegrueso}, we have that $(\omega,\Y)$ is WGP-admissible and $(\X,\omega)$ is  WGI-admissible. Moreover, Theorem \ref{iguales} says us that 
$\GP _{(\X, \Y)}=W\GP _{(\omega, \Y)}.$ Thus, the result follows from Proposition \ref{cotorpair} and its dual.
\end{dem}

\begin{defi}\label{parcotilting} 
A  pair $(\omega,\Y) $ of classes of objects in $\A$ is {\bf $\W$-cotilting} if the following three conditions hold true.
  \begin{itemize}
  \item[\normalfont{(a)}]  $\W \subseteq \omega ^{\wedge}$ and $(\omega,\Y) $ is WGP-admissible.
  \item[\normalfont{(b)}] For any $C\in{^{\perp}\Y}$ there is an exact sequence $0\to C \to W \to C' \to 0$ in $\A,$ with $W \in \W.$
  \item[\normalfont{(c)}]  For any $C \in{^{\perp} \Y}$ there is a $\Y$-preenvelope $C \to C_{\Y}$ of $C,$ with  $C_{\Y} \in \omega.$
\end{itemize}
\end{defi}

\begin{ex} Let $\A$ be an abelian category with enough injectives and let $(\X,\Y)$ be a hereditary pair in $\A,$ which is left cotorsion and right complete and $\Y$ is closed under finite coproducts in $\A.$ Then, for $\omega:=\X\cap\Y,$  we have that $(\omega,\Y)$ is an $\omega$-cotilting pair. Indeed, since $\A$ has enough injectives and $\Y$ is coresolving, it follows that 
$\X={}^{\perp_1}\Y={}^\perp\Y;$ and using that, the  needed conditions can be checked easily.
\end{ex}

\begin{defi}\label{partilting} 
A  pair $(\X,\nu) $ of classes of objects in $\A$ is {\bf $\W$-tilting} if the following three conditions hold true.
  \begin{itemize}
  \item[\normalfont{(a)}]  $\W \subseteq \nu^{\vee}$ and $(\X,\nu) $ is WGI-admissible.
  \item[\normalfont{(b)}] For any $C\in\X^{\perp}$ there is an exact sequence $0\to C' \to V\to C \to 0$ in $\A,$ with $V \in \W.$
  \item[\normalfont{(c)}]  For any $C \in\X^{\perp}$ there is a $\X$-preecover $C_\X \to C$ of $C,$ with  $C_{\X} \in \nu.$
\end{itemize}
\end{defi}

 Note that a pair $(\X,\nu) $ of classes of objects in $\A$  is $\W$-tilting if, and only if, the pair $(\nu^{op},\X^{op})$ is $\W^{op}$-cotilting in the opposite category $\A^{op}.$ Thus, any obtained result for $\W$-cotilting pairs can be translated in terms of $\W$-tilting pairs.

\begin{teo}\label{debildimfin} For a $\W$-cotilting pair $(\omega, \Y)$ in an abelian category $\A,$ the following statements hold true. 
\begin{itemize}
  \item[\normalfont{(a)}]  $W\GP_{(\omega, \Y)} ={}^{\perp}\Y.$
  \item[\normalfont{(b)}] If $\A$ has enough projectives,   then $\WGpd_{(\omega, \Y)}(M)=\pd_\Y(M)$ for any $M\in\A.$ Moreover 
  $\glWGPD _{(\omega, \Y)} (\A) =\id\, (\Y) .$
\end{itemize}
\end{teo}
\begin{dem} (a) It is clear that $W \GP _{(\omega, \Y)} \subseteq {}^{\perp} \Y .$ Let $C \in{}^{\perp} \Y.$ By Definition \ref{parcotilting} (b),  
 there is an exact sequence  
$ 0\to C \xrightarrow{h_0} I \to C_0 \to 0,$ with $I \in \W.$ Since $\W \subseteq \omega ^{\wedge},$ there is an exact sequence  
$$0 \to W_n \stackrel{f_n}{\longrightarrow}W_{n-1}\to \cdots \to W_1\stackrel{f_1}{\longrightarrow} W_0 \stackrel{f_0}{\longrightarrow} I \to 0,$$ with 
$W_i \in \omega \subseteq W \GP _{(\omega,\Y)}$  for any $i \in [0, n].$ From the exact sequence $\eta : 0 \to L \to W_0 \xrightarrow{f_0}  I \to 0,$ by doing a pull-back construction and Snake's Lemma, we get the following commutative and exact diagram 
$$\xymatrix{&  & 0\ar[d] &  0\ar[d]&\\
  \eta ' : 0\ar[r]  & L \ar[r]  \ar@{=}[d]  & E  \ar[d]^{h_0 '} \ar[r]^{f_0 '} & C \ar[r] \ar[d]^{h_0} & 0\\
\eta : 0\ar[r] & L  \ar[r] & W_0 \ar[r]^{f_0} \ar[d] & I \ar[r] \ar[d] & 0\\
&  & C_0 \ar[d] \ar@{=}[r]&  C_0 \ar[d] &\\
&  & 0 & \; 0.& }$$
Let $X \in{}^{\perp} \Y \subseteq {}^{\perp} \omega.$ By applying the functor $\Hom_{\A}(X, -) $ to the exact sequence $0 \to W_n \to W_{n-1} \to \Ima (f_{n-1}) \to 0,$ we get the exact sequence  
$$ \Ext ^i _{\A} (X, W_{n-1}) \to \Ext^i _{\A} (X, \Ima (f_{n-1})) \to \Ext _{\A} ^{i+1} (X, W_n).$$ Hence $\Ext_{\A} ^{i} (X, \Ima (f_{n-1})) = 0$ 
for any $i\geq 1.$ By repeating this procedure, it follows that $0 = \Ext ^{i} _{\A} (X, \Ima (f_1)) = \Ext ^i _{\A} (X, L),$ for any $i\geq 1 $ and 
$X \in{}^{\perp} \Y.$ Therefore  $\eta '$ splits, since $C \in{}^{\perp} \Y;$ and thus, there is  $f_0 '' : C \to E$ such that $1_C = f_0 ' f_0 ''.$ Note that 
$h_0 ' f_0 '':C\to W_0$ is a monomorphism. Moreover, from Definition \ref{parcotilting} (c), there is a $\Y$-preenvelope $h_0 '': C \to W_0 ',$ with 
 $W_0' \in  \omega,$ and so the morphism
\[h:= \left(\begin{array}{c} h_0 ' f_0 '' \\ h_0''  \end{array}\right)  :C \to W_0 \oplus W_0 ',  
 \]
 is also a $\Y$-preenvelope of $C.$ Note that  $h$ is a monomorphism, since  $h_0 ' f_0 ''$ is a monomorphism.  Now, we consider the exact sequence
 $$\zeta:\; 0 \to C \to W_0 \oplus W_0 ' \to \Coker (h) \to 0 .$$
We assert that $\Coker (h) \in {}^{\perp} \Y.$ Indeed, let $Y \in \Y.$ By applying the functor $\Hom_{\A} (-,Y) $ to $\zeta$  and using that $\id _{\omega} (\Y)=0,$ we obtain the following exact sequences: \\
 
 (i) $\Hom_{\A} (W_0 \oplus W_0 ',Y)\stackrel{h^{*}}{\longrightarrow} \Hom_{\A} (C,Y) \rightarrow \Ext ^1 _{\A} (\Coker (h) , Y) \to 0,$\\ 
 
  (ii) $\Ext^i _{\A} (C,Y)\longrightarrow \Ext^{i+1} _{\A}  (\Coker (h) , Y) \to 0,$ for any $i \geq 1.$\\
  
  Since $h:C \to W_0 \oplus W_0 '$ is a $\Y$-preenvelope and $C\in{}^\perp\Y,$ we get from  (i)  and (ii) that  $\Ext _{\A} ^{i} (\Coker (h) , Y) =0,$ for any 
  $i\geq 1.$ Therefore $\Coker (h) \in{}^{\perp} \Y.$ By repeating this procedure with  $\Coker (h)$ and the exact sequence $\zeta,$ we can proof that 
  $C \in W\GP _{(\omega,\Y)}.$
  \
  
  (b) Let $M\in\A.$ Since ${}^\perp\Y=W\GP_{(\omega,\Y)},$ we get by Lemma \ref{Ldebildimfin}  
  $$\pd_\Y(M)= \resdim_{{}^\perp\Y}(M)=\WGpd_{(\omega,\Y)}(M).$$  Therefore 
  $\glWGPD_{(\omega,\Y)}(\A)=\pd_\Y(\A)=\id_\A(\Y)=\id(\Y).$
\end{dem}

\begin{cor}\label{GdebilCotil} For a WGP-admissible pair  $(\omega,\Y),$  in an abelian category $\A,$ the following statements are equivalent. 
\begin{itemize}
\item[(a)] $W\GP_{(\omega,\Y)}={}^\perp\Y.$
\item[(b)] The pair $(\omega,\Y)$ is $\W$-cotilting, for some class $\W\subseteq\A.$
\item[(c)] The pair $(\omega,\Y)$ is $\omega$-cotilting.
\end{itemize}
\end{cor}
\begin{dem} (a) $\Rightarrow$ (b) Let $W\GP_{(\omega,\Y)}={}^\perp\Y.$ We show that we can choose $\W:=\omega.$ Since 
$W\GP_{(\omega,\Y)}={}^\perp\Y,$ the conditions (a) and (b) in Definition \ref{parcotilting} follow easily. 
\

Let $C\in  {}^\perp\Y=W\GP_{(\omega,\Y)}.$ Then there is an exact sequence $\eta:\;0\to C\xrightarrow{\varphi}W\to Z\to 0,$ with 
$W\in\omega$ and $Z\in{}^\perp\Y.$ By applying the functor $\Hom_\A(-,Y)$ to $\eta,$ with $Y\in\Y,$ we get the exact sequence 
$$\Hom_\A(W,Y)\xrightarrow{(\varphi,Y)}\Hom_\A(C,Y)\to \Ext^1_\A(Z,Y).$$
Since $\Ext^1_\A(Z,Y)=0,$ it follows that $\varphi:C\to W$ is a $\Y$-preenvelope of $C.$
\

(b) $\Rightarrow$ (c) It follows from Theorem \ref{debildimfin} that $W\GP_{(\omega,\Y)}={}^\perp\Y.$ But as we have seen in the previous implication, in this case the pair $(\omega,\Y)$ is $\omega$-cotilting.
\

(c) $\Rightarrow$ (a) It follows from Theorem \ref{debildimfin}.
\end{dem}

\begin{rk}\label{RKGdebilCotil}  Let $\A$ be an abelian category and $\omega\subseteq\A$ be such that $\add\,(\omega)=\omega$ and 
$\id_\omega(\omega)=0.$ In this case, we have that $(\omega,\omega)$ is WGP-admissible. Thus, by Corollary \ref{GdebilCotil}, we get 
that $(\omega,\omega)$ is $\W$-cotilting if and only if any $C\in{}^\perp\omega$ admits a monic $\omega$-preenvelope $C\to W.$
\end{rk}

\begin{pro}\label{PGdebilCotil} Let $(\omega,\Y)$ be a WGP-admissible pair in an abelian category $\A,$ with enough projectives, such that $\Y$ is closed under direct summands in $\A.$ Then, the following statements hold true.
\begin{itemize}
\item[(a)] $W\GP_{(\omega,\Y)}^\wedge=\A$ if and only if $(\omega,\Y)$ is $\W$-cotilting and $\pd_\Y(M)<\infty$ for any $M\in\A.$
\item[(b)] $\glWGPD_{(\omega,\Y)}(\A)<\infty$ if and only if $(\omega,\Y)$ is $\W$-cotilting and $\id(\Y)<\infty.$
\end{itemize}
\end{pro}
\begin{dem} (a)  Let $W\GP_{(\omega,\Y)}^\wedge=\A.$ Then, by Proposition \ref{WThickGP9} (a2) $W\GP_{(\omega,\Y)}={}^\perp\Y,$ and 
thus by Corollary \ref{GdebilCotil} it follows that $(\omega,\Y)$ is $\W$-cotilting. Therefore, by Theorem \ref{debildimfin} (b) 
$ \pd_\Y(M)=\WGpd_{(\omega,\Y)}(M)<\infty,$ for any $M\in\A.$
\

Assume now that $(\omega,\Y)$ is $\W$-cotilting and $\pd_\Y(M)<\infty$ for any $M\in\A.$ Then, by Theorem \ref{debildimfin} (b) 
$\WGpd_{(\omega,\Y)}(M)= \pd_\Y(M)<\infty,$ for any $M\in\A;$ proving that $W\GP_{(\omega,\Y)}^\wedge=\A.$
\

(b) Let $\glWGPD_{(\omega,\Y)}(\A)<\infty.$ Then, by (a) we get that $(\omega,\Y)$ is $\W$-cotilting. Therefore from Theorem \ref{debildimfin} (b), $\id(\Y)=\glWGPD_{(\omega,\Y)}(\A)<\infty.$
\

Assume now that $(\omega,\Y)$ is $\W$-cotilting and $\id(\Y)<\infty.$ Then, by Theorem \ref{debildimfin} (b), $\WGpd_{(\omega,\Y)}(\A)=\id(\Y)<\infty.$
\end{dem}

\begin{cor} \label{APGdebil} Let $\A$ be an abelian category, with enough projectives, and let $\omega\subseteq\A$ be such that $\add\,(\omega)=\omega$ and 
$\id_\omega(\omega)=0.$ Then, the following conditions are equivalent.
\begin{itemize}
\item[(a)] $W\GP_\omega^\wedge=\A$ (respectively, $\glWGPD_{(\omega,\omega)}(\A)<\infty).$
\item[(b)] $(\omega,\omega)$ is $\W$-cotilting and $\pd_\omega(M)<\infty$ for any $M\in\A$ (respectively, $\id\,(\omega)<\infty$).
\item[(c)] Any $C\in{}^\perp\omega$ admits a monic $\omega$-preenvelope and $\pd_\omega(M)<\infty$ for any $M\in\A$ (respectively, $\id\,(\omega)<\infty$).
\end{itemize}
If one of the above equivalent conditions holds, then $\glWGPD_\omega(\A)=\id\,(\omega).$
\end{cor}
\begin{dem} It follows from Proposition \ref{PGdebilCotil},  Remark \ref{RKGdebilCotil} and Theorem \ref{debildimfin} (b).
\end{dem}

\begin{cor}\label{1APGdebil} Let $R$ be a ring such that $\GP(R)^\wedge=\Modu(R)=\GI(R)^\vee.$ Then $\glGPD(R)=\id(\Proj(R))$ and 
$\glGID(R)=\pd(\Inj(R)).$
\end{cor}
\begin{dem} We can apply Corollary \ref{APGdebil} to the class $\omega:=\Proj\,(R)$. Note that $W\GP_\omega$ coincide with the class 
$\GP(R)$ of the Gorenstein-projective $R$-modules and
thus $\glWGPD_\omega(\A)$ is just the global Gorenstein projective dimension $\glGPD(R)$ of the ring $R.$ In order to get the equality 
$\glGID(R)=\pd(\Inj(R)),$ we apply the dual of Corollary \ref{APGdebil} to the class $\nu:=\Inj(R).$
\end{dem}

\begin{cor}\label{Gdebildimfin} For a GP-admissible pair $(\X,\Y)$ in an abelian category $\A$ and $\omega:=\X\cap \Y,$ the 
following statements hold true.
\begin{itemize}
\item[(a)] Let $(\omega, \Y)$ be $\W$-cotilting. Then
 \begin{itemize}
  \item[\normalfont{(a1)}]  $\GP_{(\X, \Y)} ={}^{\perp}\Y.$
  \item[\normalfont{(a2)}] If $\A$ has enough projectives,  then $\Gpd_{(\X, \Y)}(M)=\pd_\Y(M)$ for any $M\in\A.$ Moreover 
  $\glGPD _{(\omega, \Y)} (\A) =\id\, (\Y) .$
\end{itemize}
\item[(b)] $\GP_{(\X,\Y)}={}^\perp\Y$ if and only if the pair $(\omega,\Y)$ is $\W$-cotilting.
\item[(c)] Let $\Y$ be closed under direct summands in $\A,$ with enough projectives. Then
 \begin{itemize}
\item[(c1)] $\GP_{(\X,\Y)}^\wedge=\A$ if and only if $(\omega,\Y)$ is $\W$-cotilting and $\pd_\Y(M)<\infty$ for any $M\in\A.$
\item[(c2)] $\glGPD_{(\X,\Y)}(\A)<\infty$ if and only if $(\omega,\Y)$ is $\W$-cotilting and $\id(\Y)$ is finite.
\end{itemize}
\end{itemize}
\end{cor}
\begin{dem} By Remark \ref{debiladmisiblegrueso}, we have that $(\omega,\Y)$ is WGP-admissible. Moreover, by Theorem \ref{iguales} we get 
$\GP _{(\X, \Y)}=W\GP _{(\omega, \Y)}.$ Thus, the result follows from Theorem \ref{debildimfin}, Corollary \ref{GdebilCotil} and Proposition \ref{PGdebilCotil}.
\end{dem}

\begin{rk}\label{1Gdebildimfin} (1) Let $R$ be a ring,  $\X:=\Proj (R)$ and $\Y:=\Flat (R).$ We apply this situation to Corollary \ref{Gdebildimfin} (c1). Note that in this case $\GP_{(\X,\Y)}$ is the class of the Ding-projetive $R$-modules $\DP(R).$ Thus, we have 
that $\DP(R)^\wedge=\Modu\,(R)$ if and only if $(\Proj(R),\Flat(R))$ is $\W$-cotilting and $\pd_{\Flat(R)}(M)<\infty$ for any $R$-module $M.$
\

(2) Let $R$ be a Ding-Cheng ring. Then, by Corollary \ref{GdebilCotil} and Corollary \ref{ThickGP13} (a), we get that the pair 
$(\Proj(R),\Flat(R))$ is $\W$-cotilting. Then, by (1), it follows that $\DP(R)^\wedge=\Modu\,(R)$ if and only if  
$\pd_{\Flat(R)}(M)<\infty$ for any $R$-module $M.$
\end{rk}

\begin{teo}\label{tiltcotor} 
Let $\A$ be an abelian category with enough projectives, and let $(\omega, \Y)$ be a $\W$-cotilting pair in $\A$, with $\omega$ closed under direct summands in $\A,$  and  
$\id\, (\Y) < \infty.$ Then,  the following statements hold true.
\begin{itemize}
  \item[(a)]  $(W\GP _{(\omega, \Y)}, \omega ^{\wedge})$ is a hereditary complete cotorsion pair in $\A.$
  \item[(b)] $\omega ={}^{\perp} \Y \cap \omega ^{\wedge}, \;\;\; W\GP_{(\omega, \Y)} ={}^{\perp} \omega = {}^{\perp} (\omega ^{\wedge}) = {}^{\perp} \Y \mbox{\;\;\;and \;\; } W\GP _{(\omega, \Y)} ^{\perp}  = \omega ^{\wedge}.$
   \item[(c)] $\glWGPD_{(\omega,\Y)}(\A)=\FPD_\omega(\A)=\FPD_{\omega^\wedge}(\A)=\resdim_\omega(\omega^\wedge)=\pd_\omega(\omega^\wedge)=$ $=\pd_{\omega^\wedge}(\omega^\wedge)=\id\,(\Y)< \infty.$
   \item[(d)] $\Q_{\omega}^{<\infty}=\A=W\GP _{(\omega , \Y)} ^{\wedge} =\Q_{\omega^\wedge}^{<\infty}.$
   \item[(e)] Let $\Y$ be closed under direct summands in $\A.$ Then, $\Q_{\Y}^{<\infty}=\A=\Q_{\Y^\wedge}^{<\infty}$ and 
   $\glWGPD_{(\omega,\Y)}(\A)=\FPD_\Y(\A)=\FPD_{\Y^\wedge}(\A).$
\end{itemize}
\end{teo}
\begin{dem}  (a) By Theorem \ref{debildimfin} (b) and $\id (\Y )< \infty,$  we get that $\glWGPD _{(\omega , \Y)} (\A)$ is finite.  Then, from  Proposition  \ref{cotorpair} (c), we conclude (a).
\

(b) By Theorem \ref{debildimfin} (a), we have $W\GP _{(\omega, \Y)} ={}^{\perp} \Y.$  Then, Proposition \ref{WThickGP9} (b3) implies that  
$$\omega = W\GP _{(\omega, \Y)} \cap \omega ^{\wedge} = {}^{\perp} \Y \cap \omega ^{\wedge}.$$ 
Moreover, since $W\GP _{(\omega, \Y)} ^{\wedge} = \A,$ we get from Proposition \ref{WThickGP9} (b3)  the equalities 
$$W\GP _{(\omega, \Y)} ={}^{\perp} \omega ={}^{\perp} (\omega ^{\wedge}) ={}^{\perp} \Y .$$
On the other hand, from Proposition \ref{cotorpair} (b) it follows  $W\GP _{(\omega, \Y)} ^{\perp} = \omega ^{\wedge}.$
\

(c) and (d):  It follows from Theorem \ref{WThickGP6}, Theorem \ref{debildimfin} (b) and Theorem \ref{WThickGP10} (b), since $W\GP _{(\omega , \Y)} ^{\wedge} = \A$ and $\glWGPD_{(\omega,\Y)}(\A)=\id(\Y)<\infty.$
\

(e) It follows from Theorem \ref{WThickGP10} (a), since $W\GP _{(\omega , \Y)} ^{\wedge} = \A.$
\end{dem} 

\begin{cor}\label{CotGorP} Let $(\X,\Y)$ be a GP-admissible pair in an abelian category $\A,$ with enough projectives, and 
such that $(\omega, \Y)$ is a $\W$-cotilting pair in $\A,$ where $\omega := \X \cap \Y$ is closed under direct 
summands and  $\id\, (\Y) < \infty.$  Then,   the following statements hold true.
\begin{itemize}
  \item[(a)]  $(\GP _{(\X, \Y)}, \omega ^{\wedge})$ is a hereditary complete cotorsion pair in $\A.$
  \item[(b)] $\omega ={}^{\perp} \Y \cap \omega ^{\wedge}, \;\;\; \GP_{(\X, \Y)} ={}^{\perp} \omega = {}^{\perp} (\omega ^{\wedge}) = {}^{\perp} \Y \mbox{\;\;\;and \;\; } \GP _{(\X, \Y)} ^{\perp}  = \omega ^{\wedge}.$
   \item[(c)] $\glGPD_{(\X,\Y)}(\A)=\FPD_\omega(\A)=\FPD_{\omega^\wedge}(\A)=\resdim_\omega(\omega^\wedge)=\pd_\omega(\omega^\wedge)=$ $=\pd_{\omega^\wedge}(\omega^\wedge)=\id\,(\Y)< \infty.$
   \item[(d)] $\Q_{\omega}^{<\infty}=\A=\GP _{(\X , \Y)} ^{\wedge} =\Q_{\omega^\wedge}^{<\infty}.$
   \item[(e)] Let $\Y$ be closed under direct summands in $\A.$ Then, $\Q_{\Y}^{<\infty}=\A=\Q_{\Y^\wedge}^{<\infty}$ and 
   $\glGPD_{(\X,\Y)}(\A)=\FPD_\Y(\A)=\FPD_{\Y^\wedge}(\A).$
\end{itemize}
\end{cor}
\begin{dem} By Remark \ref{debiladmisiblegrueso}, we have that $(\omega,\Y)$ is WGP-admissible. Moreover, Theorem \ref{iguales} says us that 
$\GP _{(\X, \Y)}=W\GP _{(\omega, \Y)}.$ Thus, the result follows from Theorem \ref{tiltcotor}.
\end{dem}

\begin{teo} \label{Wtiltcotor} Let $\A$ be an abelian category with enough projectives and injectives, and let $(\omega, \Y)$ be a WGP-admissible pair 
in $\A$, with both $\omega$  and $\Y$ closed under direct summands in $\A.$ Then, the following statements are equivalent.
\begin{itemize}
\item[(a)] The pair $(\omega, \Y)$ is $\W$-cotilting and $\id(\Y)<\infty.$ 
\item[(b)] $(W\GP_{(\omega,\Y)},\omega^\wedge)$ is an hereditary complete cotorsion pair in $\A$ such that $\id(\omega)<\infty.$
\item[(c)] $W\GP_{(\omega,\Y)}={}^\perp\omega$ and $\id(\omega)<\infty.$
\end{itemize}
If one of the above equivalent conditions holds, then $\pd_\Y(M)=\WGpd_{(\omega,\Y)}(M)=\pd_\omega(M),$ for any $M\in\A.$ Moreover $\glWGPD_{(\omega,\Y)}(\A)=\id(\Y)=\id(\omega)<\infty.$
\end{teo}
\begin{dem} (a) $\Rightarrow$ (b) By Theorem \ref{tiltcotor} (a), we get that $(W\GP_{(\omega,\Y)},\omega^\wedge)$ is an hereditary complete cotorsion pair in $\A$ and $W\GP_{(\omega,\Y)}={}^\perp\omega.$  Then, by Lemma \ref{Ldebildimfin},  Theorem \ref{debildimfin} (b) and Proposition \ref{WThickGP9} (b1) $\pd_\Y(M)=\WGpd_{(\omega,\Y)}(M)=\pd_\omega(M),$ for any 
$M\in\A=W\GP_{(\omega,\Y)}^\wedge.$ In particular 
$\id(\omega)=\pd_\omega(\A)=\pd_\Y(\A)=\id(\Y)$ and thus $\id(\omega)<\infty.$
\

(b) $\Rightarrow$ (c) Since $(W\GP_{(\omega,\Y)},\omega^\wedge)$ is an hereditary cotorsion pair, we have that $\omega^\wedge$ is coresolving 
and $W\GP_{(\omega,\Y)}={}^{\perp_1}(\omega^\wedge).$ Using now that $\A$ has enough injectives and $\omega^\wedge$ is coresolving , 
it follows that ${}^{\perp_1}(\omega^\wedge)={}^{\perp}(\omega^\wedge)$ and thus $W\GP_{(\omega,\Y)}={}^{\perp}(\omega^\wedge).$ Finally, by the dual of Lemma \ref{AB1}, we have ${}^{\perp}(\omega^\wedge)={}^{\perp}\omega.$
\

(c) $\Rightarrow$ (a) Since $W\GP_{(\omega,\Y)}={}^\perp\omega,$ we get from Lemma \ref{Ldebildimfin} that $\WGpd_{(\omega,\Y)}(M)=\pd_\omega(M),$ for any $M\in\A.$ Therefore $\glWGPD_{(\omega,\Y)}(\A)=\pd_\omega(\A)=\id(\omega)<\infty.$ Then, by 
Proposition \ref{PGdebilCotil} (b), we conclude (a).
\end{dem}

\begin{cor} \label{CWtiltcotor} Let $(\X,\Y)$ be a GP-admissible pair in an abelian category $\A,$ with enough projectives and injectives, and let $(\omega, \Y)$ be a WGP-admissible pair 
in $\A$, with both $\omega:=\X\cap\Y$  and $\Y$ closed under direct summands in $\A.$ Then, the following statements are equivalent.
\begin{itemize}
\item[(a)] The pair $(\omega, \Y)$ is $\W$-cotilting and $\id(\Y)<\infty.$ 
\item[(b)] $(\GP_{(\X,\Y)},\omega^\wedge)$ is an hereditary complete cotorsion pair in $\A$ such that $\id(\omega)$ is finite.
\item[(c)] $\GP_{(\X,\Y)}={}^\perp\omega$ and $\id(\omega)<\infty.$
\end{itemize}
If one of the above equivalent conditions holds, then $\pd_\Y(M)=\Gpd_{(\X,\Y)}(M)=\pd_\omega(M),$ for any $M\in\A.$ Moreover 
$\glGPD_{(\X,\Y)}(\A)=\id(\Y)=\id(\omega)<\infty.$
\end{cor}
\begin{dem} By Remark \ref{debiladmisiblegrueso}, we have that $(\omega,\Y)$ is WGP-admissible. Moreover, Theorem \ref{iguales} says us that 
$\GP _{(\X, \Y)}=W\GP _{(\omega, \Y)}.$ Thus, the result follows from Theorem \ref{Wtiltcotor}.
\end{dem}
\vspace{0.2cm}

We can give the following characterization of the finiteness of  the global Ding-projective dimension of a ring $R.$ 

\begin{cor}\label{1CWtiltcotor} For any ring $R,$ the following statements are equivalent.
\begin{itemize}
\item[(a)] The pair $(\Proj(R),\Flat(R))$ is $\W$-cotilting and $\id(\Flat(R))<\infty.$
\item[(b)]  $(\DP(R),\Proj(R)^\wedge)$ is a hereditary complete cotorsion pair in $\Modu\,(R)$ and $\id(\Proj(R))<\infty.$
\item[(c)]  $\DP(R)={}^\perp\Proj(R)$ and $\id(\Proj(R))<\infty.$
\item[(d)] $\glDPD(R)<\infty.$
\end{itemize}
Moreover, if one of the equivalent conditions holds, then 
$$\glDPD(R)=\id(\Proj\,(R))=\id(\Flat(R)).$$
\end{cor}
\begin{dem} It follows from Corollary \ref{CWtiltcotor} and Corollary \ref{Gdebildimfin} (c2), by taking $\X=\Proj(R)$ and $\Y=\Flat(R).$
\end{dem}
\vspace{0.2cm}

We can give the following characterization of the finiteness of  the global Gorenstein-projective dimension of a ring $R.$ 

\begin{cor}\label{2CWtiltcotor} For any ring $R,$ the following statements are equivalent.
\begin{itemize}
\item[(a)] The pair $(\Proj(R),\Proj(R))$ is $\W$-cotilting and $\id(\Proj(R))<\infty.$
\item[(b)]  $(\GP(R),\Proj(R)^\wedge)$ is a hereditary complete cotorsion pair in $\Modu\,(R)$ and $\id(\Proj(R))<\infty.$
\item[(c)]  $\GP(R)={}^\perp\Proj(R)$ and $\id(\Proj(R))<\infty.$
\item[(d)] $\glGPD(R)<\infty.$
\end{itemize}
Moreover, if one of the equivalent conditions holds, then $\glGPD(R)=\id(\Proj\,(R)).$
\end{cor}
\begin{dem} It follows from Corollary \ref{CWtiltcotor} and Corollary \ref{Gdebildimfin} (c2), by taking $\X=\Proj(R)=\Y.$
\end{dem}
\vspace{0.2cm}

Finally, we give the following characterization of the finiteness of  the global Gorenstein-injective dimension of a ring $R.$ 

\begin{cor}\label{3CWtiltcotor} For any ring $R,$ the following statements are equivalent.
\begin{itemize}
\item[(a)] The pair $(\Inj(R),\Inj(R))$ is $\W$-tilting and $\pd(\Inj(R))<\infty.$
\item[(b)]  $(\Inj(R)^\vee,\GI(R),)$ is a hereditary complete cotorsion pair in $\Modu\,(R)$ and $\pd(\Inj(R))<\infty.$
\item[(c)]  $\GI(R)=\Inj(R)^\perp$ and $\pd(\Inj(R))<\infty.$
\item[(d)] $\glGID(R)<\infty.$
\end{itemize}
Moreover, if one of the equivalent conditions holds, then $\glGID(R)=\pd(\Inj\,(R)).$
\end{cor}
\begin{dem} It follows from the duals of Corollary \ref{CWtiltcotor} and Corollary \ref{Gdebildimfin} (c2), by taking $\X=\Inj(R)=\Y.$
\end{dem}
\section{Cotilting objects and $\W$-cotilting pairs}

Tilting and cotilting objects were introduced in the eighties, by S. Brenner and M. Butler \cite{BB}  and by D. Happel and C. M. 
Ringel \cite{HR}, in the context of the abelian category $\modu(\Lambda)$ of the finitely generated left $\Lambda$-modules for 
some Artin algebra $\Lambda.$ A generalization of tilting and cotilting, in  $\modu(\Lambda),$ was given by Y. Miyashita in \cite{M}. In 
 the case of the abelian category $\Modu(R),$ for an arbitrary ring $R,$ a generalization of tilting and cotilting were given by L. Angeleri 
  H\"ugel and F. U. Coelho \cite{AC}. This generalization of cotilting (respectively, tilting) is suitable to be extended to abelian categories 
  and will be used throughout this section to be compared with the notion of $\W$-cotilting (respectively, $\W$-tilting) pair.

\begin{defi}\cite{AC} \label{CotilD} Let $\A$ be an abelian category with injective cogenerators. An object $M\in\A$ is {\bf cotilting} if the following conditions hold 
true.
\begin{itemize}
\item[(C1)] $\id(M)<\infty.$
\item[(C2)] $M$ is $\Pi$-orthogonal, that is,   $\Ext_\A^i(M^I,M)=0$ $\forall\,i\geq 1$ and any set $I.$
\item[(C3)] There is an injective cogenerator $Q$ in $\A$ such that $\resdim_{\Prod(M)}(Q)<\infty.$
\end{itemize}
\end{defi}

By dualizing the above definition, we get the notion of tilting object. For completeness, we write down this notion. 

\begin{defi}\cite{AC} \label{TilD} Let $\A$ be an abelian category with projective generators. An object $M\in\A$ is {\bf tilting} if the following conditions hold 
true.
\begin{itemize}
\item[(T1)] $\pd(M)<\infty.$
\item[(T2)] $M$ is $\Sigma$-orthogonal, that is,  $\Ext_\A^i(M,M^{(I)})=0$ $\forall\,i\geq 1$ and any set $I.$
\item[(T3)] There is a projective generator $P$ in $\A$ such that $\coresdim_{\Add(M)}(P)<\infty.$
\end{itemize}
\end{defi}

\begin{pro}\label{CotilDWCotil} Let $\A$ be an AB4*-abelian category with injective cogenerators. Then, for any $M\in\A$ satisfying 
conditions (C2) and (C3)  in Definition \ref{CotilD}, the pair $(\Prod(M),\Prod(M))$ is $\W$-cotilting.
\end{pro}
\begin{dem} Let $M\in\A$ and $\omega:=\Prod(M).$ By following the proof in \cite[Proposition 1.1]{AC}, we get that $\omega$ is a preenveloping class. 
\

Suppose that $M$ satisfies conditions (C2) and (C3)  in Definition \ref{CotilD}.  By (C2) and Remark \ref{B1}, we have that $\id_\omega(\omega)=0$ and $\omega\subseteq \X^\perp,$ for $\X:={}^\perp\omega.$ On the 
other hand, by (C3), there is an exact sequence 
$(*):\;0\to M_n\to M_{n-1}\to \cdots \to M_0\xrightarrow{f} Q\to 0,$ where $M_i\in\omega$ $\forall\,i.$ Since $\X^\perp$ is 
coresolving and $\omega\subseteq \X^\perp,$ from the exact sequence $(*)$ we get that $\Ker(f)\in\X^\perp.$ Therefore, 
$f:M_0\to Q$  is an $\X$-precover of $Q.$
\

We assert that any object in $\X={}^\perp\omega$ admits a monic $\omega$-preenvelope. Indeed, since $\omega$ is a preenveloping class, it is enough to show that any object in $\X$ can be embedded, through a monomorphism, into some object of $\omega.$
\

Let $Z\in\X.$ Since $Q$ is an injective cogenerator in $\A,$ there is a monomorphism $g:Z\to Q^I,$ for some set $I.$ On the other hand, we have the $\X$-precover 
$f:M_0\to Q,$ and thus we get the $\X$-precover $f^I:M_0^I\to Q^I,$ since  $M_0^I\in\X.$ In particular, there is a  morphism 
$g':Z\to M_0^I$ such that $g=f^I\,g',$ and hence $g'$ is a monomorphism; proving our assertion. Finally, from Remark \ref{RKGdebilCotil}, 
we conclude that the pair $(\Prod(M),\Prod(M))$ is $\W$-cotilting.
\end{dem}

\begin{rk}  Let $\A$ be an AB4*-abelian category with injective cogenerators and enough projectives and $M\in\A.$ By using the ideas of the  proof of Proposition \ref{CotilDWCotil}, we can show that the following statements hold true.
\begin{itemize}
\item[(a)] If $M$ satisfies condition (C2) in Definition \ref{CotilD} and $W\GI_{({}^\perp M,\Prod(M))}$ has an injective cogenerator in 
$\A,$ then the pair $(\Prod(M),\Prod(M))$ is $\W$-cotilting.
\item[(b)] If $M$ satisfies conditions (C2) and (C3) in Definition \ref{CotilD}, then the class $W\GI_{({}^\perp M,\Prod(M))}$ has an 
injective cogenerator in $\A.$ 
\end{itemize}
\end{rk}

\begin{cor}\label{Cotil1} Let $\A$ be an AB4*-abelian category with injective cogenerators and enough projectives. Then, for any 
cotilting object $M\in\A$ and $\omega:=\Prod(M),$ the following statements hold true.
\begin{itemize}
  \item[(a)]  $(W\GP _{\omega}, \omega ^{\wedge})$ is a hereditary complete cotorsion pair in $\A.$
  \item[(b)] $\omega ={}^{\perp} M \cap \omega ^{\wedge}, \;\;\; W\GP_{\omega} ={}^{\perp} M= {}^{\perp} (\omega ^{\wedge}) \mbox{\;\;\;and \;\; } W\GP _{\omega} ^{\perp}  = \omega ^{\wedge}.$
   \item[(c)] $\glWGPD_{(\omega,\omega)}(\A)=\FPD_\omega(\A)=\FPD_{\omega^\wedge}(\A)=\resdim_\omega(\omega^\wedge)=\pd_\omega(\omega^\wedge)=$ $=\pd_{\omega^\wedge}(\omega^\wedge)=\id\,(M)< \infty.$
   \item[(d)] $\Q_{\omega}^{<\infty}=\A=W\GP _{\omega} ^{\wedge} =\Q_{\omega^\wedge}^{<\infty}.$
\end{itemize}
\end{cor}
\begin{dem} Let $M\in\A$ be a cotilting object and $\omega:=\Prod(M).$  By Remark \ref{B1}, we know that 
${}^\perp M={}^\perp\omega$ and $\id(\omega)=\id(M).$ Then, the corollary follows from Proposition \ref{CotilDWCotil} and Theorem 
\ref{tiltcotor}.
\end{dem}
\vspace{0.2cm}

The following generalization of the shifting Lemma is very useful. The proof is straightforward and it is left to the reader.

\begin{lem}\label{ShifL} Let $(\X,\Y)$ be a pair of clases of objects in an abelian category $\A$ such that $\id_\X(\Y)=0.$ Then, for 
any exact sequence $0\to M\to Y_0\to Y_1\to \cdots\to Y_{n-1}\to Z_n\to 0,$ where $Y_i\in \Y$ for any $i\in[0,n-1],$ 
$$\Ext_\A^k(-,Z_n)|_{\X}\simeq \Ext_\A^{k+n}(-,M)|_{\X}\quad \text{for any}\; k>0.$$
\end{lem}

\begin{teo}\label{Cotil2} Let $\A$ be an AB4*-abelian category with injective cogenerators and enough projectives, and let $(\X,\Y)$ be 
a hereditary complete cotorsion pair in $\A$ such that $\id(\Y)<\infty$ and $\omega:=\X\cap\Y$ be closed under products. Then, there is 
some cotilting object $M\in\A$ such that $\omega=\Prod(M),$ $\id(M)=\id(\Y),$ $\Y=\omega^\wedge$ and $W\GP_\omega=\X={}^\perp M.$
\end{teo}
\begin{dem} By \cite[Theorem 3.12]{BMPS}, we have $\Inj(\A)\subseteq \Y=\omega^\wedge.$ In particular, by the dual of Lemma \ref{AB1}, $\id(\omega)=\id(\Y).$ Since $\omega\subseteq\Y,$ $\pd_\Y(\omega)=0$ and $\omega$ is closed under direct summands in $\A,$ by the dual of Proposition \ref{AB7} (b), we get that $\pd_\Y(\omega^\wedge)=\resdim_\omega(\omega^\wedge).$ In particular, we have that 
$$n:=\resdim_\omega(\Inj(\A))\leq \pd_\Y(\omega^\wedge)=\id_{\omega^\wedge}(\Y)\leq \id(\Y)<\infty.$$
Choose some injective cogenerator $Q$ in $\A.$ Then $\resdim_\omega(Q)\leq \resdim_\omega(\Inj(A))=n$  and thus there is an exact sequence
$0\to W_n\to \cdots\to W_1\to W_0\to Q\to 0,$ with $W_i\in\omega$ for any $i\in[0,n].$ 
\

Consider $M:=\bigoplus_{i=0}^n\,W_i.$ Since $\omega$ is closed under products, it follows that $\Prod(M)\subseteq \omega.$ In particular $\id_{\Prod(M)}(M)\leq\id_\omega(\omega)=0$ and thus the conditions (C2) and (C3) in Definition \ref{CotilD} hold for $M.$
Hence, by Proposition \ref{CotilDWCotil} and Theorem \ref{debildimfin} (a),  it 
follows that $W\GP_{\Prod(M)}={}^\perp\Prod(M).$ Furthermore, the inclusion  
$\Prod(M)\subseteq \omega$ implies that $\omega\subseteq {}^\perp\omega\subseteq  {}^\perp\Prod(M)=W\GP_{\Prod(M)}$ and 
so $\omega\subseteq W\GP_{\Prod(M)}.$ 
\

Now, by using the inclusion $\omega\subseteq W\GP_{\Prod(M)},$ we show that $\omega\subseteq \Prod(M).$ Indeed, let $m:=\id(\omega)=\id(\Y)<\infty.$ Consider $W\in\omega\subseteq W\GP_{\Prod(M)}.$ Then, there is an exact sequence 
$$(*):\quad 0\to W\xrightarrow{f_0}M_0\xrightarrow{f_1}M_1\to\cdots\to M_{m-1}\xrightarrow{f_m}M_m\to K_{m+1}\to 0,$$
where $M_i\in\Prod(M)$ and $K_{i+1}:=\Coker(f_i)\in {}^\perp\Prod(M),$ for all $i\in[0,m].$ Since $\id_{{}^\perp\Prod(M)}(\Prod(M))=0,$ we get from Lemma \ref{ShifL} $\Ext^1_\A(K_{m+1},K_m)\simeq \Ext^{m+1}_\A(K_{m+1},W)=0$ and thus 
the exact sequence $0\to K_m\to M_m\to K_{m+1}\to 0$ splits; proving that $K_m\in\Prod(M)\subseteq\omega\subseteq{}^\perp\omega.$ Then, by using the fact that ${}^\perp\omega$ is resolving, we get from $(*)$ that $K_1\in{}^\perp\omega.$ But now, $\id_{\Prod(M)}(\omega)\leq\id_\omega(\omega)=0$ and hence the exact sequence $0\to W\to M_0\to K_1\to 0$ splits; proving that $W\in \Prod(M).$ Then $\omega=\Prod(M)$ and $\id(M)=\id(\Prod(M))=\id(\omega)=\id(\Y).$ 
\

Finally, by the dual of Lemma \ref{AB1}, we get that ${}^\perp\omega={}^\perp(\omega^\wedge).$ Therefore $\X={}^\perp\Y={}^\perp\omega={}^\perp M=W\GP_\omega.$
\end{dem}

\begin{cor} \label{Cotil3} Let $\A$ be an AB4*-abelian category with injective cogenerators and enough projectives and 
$\omega\subseteq\A.$ Then, the following conditions are equivalent.
\begin{itemize}
\item[(a)] The pair $(\omega,\omega)$ is $\W$-cotilting, $\id(\omega)<\infty$ and $\omega=\Prod(\omega).$ 
\item[(b)] There is a cotilting object $M\in\A$ such that $\omega=\Prod(M).$
\end{itemize}
\end{cor}
\begin{dem} (a) $\Rightarrow$ (b) It follows from Theorem  \ref{tiltcotor} (a) and Theorem \ref{Cotil2}.
\

(b) $\Rightarrow$ (a)  It follows from Proposition \ref{CotilDWCotil}.
\end{dem}

\begin{lem}\label{AuxCotil4} Let $R$ be a left perfect, left noetherian and right coherent ring, and let $Q$ be an injective cogenerator in 
$\Modu(R).$ Then 
\begin{itemize}
\item[(a)] $\Proj(R)=\Prod({}_RR),$  ${}^\perp\Proj(R)={}^\perp R$   and $\id(\Proj(R))=\id({}_RR);$
\item[(b)] $\Inj(R)=\Add(Q),$ $\Inj(R)^\perp=Q^\perp$ and $\pd(\Inj(R))=\pd(Q).$
\item[(c)] ${}_RR$ is $\Pi$-orthogonal and $Q$ is $\Sigma$-orthogonal. 
\item[(d)] $\resdim_{\Prod({}_RR)}(M)=\pd(M)$ and $\coresdim_{\Add(Q)}(M)=\id(M)$ $\forall\,M\in\Modu(R).$
\end{itemize}
\end{lem}
\begin{dem} (a) By \cite[Theorem 3.3]{C}, it follows that $\Proj(R)$ is closed under products in $\A$ and thus $\Proj(R)=\Prod({}_RR).$ 
Then, by Remark \ref{B1} we get (a).
\

(b) By \cite[Proposition 18.13]{AF}, it follows that $\Inj(R)$ is closed under coproducts in $\A$ and thus $\Inj(R)=\Add(Q).$ 
Then, by the dual of Remark \ref{B1} we get (b).
\

(c) Let $I$ be a set. Then, by (a), we have that $R^I$ is projective and thus $\Ext^i_R(R^I,R)=0$ for any $i\geq 1.$ On the other hand, by 
(b), we have that $Q^{(I)}$ is injective and thus $\Ext^i_R(Q,Q^{(I)})=0$  for any $i\geq 1.$
\

(d) From (a) and (b), we know that $\Proj(R)=\Prod({}_RR)$ and $\Inj(R)=\Add(Q).$ Therefore, (d) holds true.
\end{dem}

\begin{teo}\label{Cotil4} Let $R$ be a left perfect, left noetherian and right coherent ring, and let $Q$ be an injective cogenerator in 
$\Modu(R).$ Then, the following statements are equivalent.
\begin{itemize}
\item[(a)] ${}_RR$ is cotilting in $\Modu(R).$
\item[(b)] $\id({}_RR)<\infty$ and $\pd(Q)<\infty.$
\item[(c)] The pair $(\Proj(R),\Proj(R))$ is $\W$-cotilting and $\id({}_RR)<\infty.$
\item[(d)] $(\GP(R),\Proj(R)^\wedge)$ is a hereditary complete cotorsion pair in $\Modu\,(R)$ and $\id({}_RR)<\infty.$
\item[(e)]  $\GP(R)={}^\perp ({}_RR)$ and $\id({}_RR)<\infty.$
\item[(f)] $\glGPD(R)<\infty.$
\item[(g)] $Q$ is tilting in $\Modu(R).$
\item[(h)] The pair $(\Inj(R),\Inj(R))$ is $\W$-tilting and $\pd(Q)<\infty.$
\item[(i)] $(\Inj(R)^\vee,\GI(R))$ is a hereditary complete cotorsion pair in $\Modu\,(R)$ and $\pd(Q)<\infty.$
\item[(j)]  $\GI(R)=Q^\perp$ and $\pd(Q)<\infty.$
\item[(k)] $\glGID(R)<\infty.$
\end{itemize}
If one of the above equivalent conditions holds, then $\Proj(R)^\wedge=\Inj(R)^\vee$ and 
$$\FPD(R)=\glGPD(R)=\id({}_RR)=\pd(Q)=\glGID(R)=\FID(R)<\infty.$$
\end{teo}
\begin{dem} The equivalence between  (a) and (b) follows from Lemma \ref{AuxCotil4}. By using Lemma \ref{AuxCotil4} and Corollary 
\ref{Cotil3}, it can be shown that  (a) and (c) are equivalent.  The equivalences between (c), (d), (e) and (f) follow from Lemma \ref{AuxCotil4} and Corollary \ref{2CWtiltcotor}.
\

The equivalence between (g) and (b) follows from Lemma \ref{AuxCotil4}. By using Lemma \ref{AuxCotil4} and the dual of Corollary 
\ref{Cotil3}, it can be shown that  (g) and (h) are equivalent.  The equivalences between (h), (i), (j) and (k) follow from Lemma \ref{AuxCotil4} and Corollary \ref{3CWtiltcotor}.
\

Let $\id({}_RR)<\infty$ and $\pd(Q)<\infty.$ We prove that  $\Proj(R)^\wedge=\Inj(R)^\vee.$ Indeed, let $X\in \Proj(R)^\wedge.$ 
 Then, by Lemma \ref{AuxCotil4} (a) and the dual of Lemma \ref{AB1}, we have $\id(X)\leq \id (\Proj(R)^\wedge)=\id({}_RR)<\infty$ 
 and thus $X\in \Inj(R)^\vee.$ Consider $Y\in \Inj(R)^\vee.$ Then, by Lemma \ref{AuxCotil4} (b) and  Lemma \ref{AB1}, we have 
 $\pd(Y)\leq \pd(\Inj(R)^\vee)=\pd(Q)<\infty$ and hence $Y\in\Proj(R)^\wedge;$ proving that $\Proj(R)^\wedge=\Inj(R)^\vee.$ 
 \
 
 Let $\omega:=\Proj(R)$ and $\nu:=\Inj(R).$ In particular, by Proposition \ref{GP2} and its dual,  we have that 
 $W\GP_(\omega,\omega)=\GP(R)$ and $W\GI_(\nu,\nu)=\GI(R).$ Assume that one of the above items hold true. Then by Theorem \ref{tiltcotor} and its dual, it follows that $\FPD(R)=\glGPD(R)=\pd_{\omega^\wedge}(\omega^\wedge)=\id_{\nu^\vee}(\nu^\vee)=\glGID(R)=\FID(R).$ Finally, from Corollary \ref{2CWtiltcotor} and Corollary \ref{3CWtiltcotor}, we get $\glGPD(R)=\id({}_RR)$ and 
 $\glGID(R)=\pd(Q).$
\end{dem}

\begin{cor}\label{Cotil5}   Let $\Lambda$ be an Artin $R$-algebra and $D:=\Hom_R(-,k):\modu\,(\Lambda)\to \modu\,(\Lambda^{op})$ be the 
usual duality, where $k$ is the injective envelope of $R/\mathrm{rad}(R).$ If $\id({}_{\Lambda}\Lambda)<\infty$ and 
$\id(\Lambda_{\Lambda})<\infty,$ then the following statements hold true.
\begin{itemize}
\item[(a)] ${}_{\Lambda}\Lambda$ is cotilting in $\Modu(\Lambda)$ and $D(\Lambda_{\Lambda})$ is tilting in $\Modu(\Lambda).$
\item[(b)] The pair $(\Proj(\Lambda),\Proj(\Lambda))$ is $\W$-cotilting.
\item[(c)] $(\GP(\Lambda),\Proj(\Lambda)^\wedge)$ is a hereditary complete cotorsion pair in $\Modu\,(\Lambda).$ 
\item[(d)]  $\GP(\Lambda)={}^\perp ({}_{\Lambda}\Lambda)$ and $\GI(\Lambda)=(D(\Lambda_{\Lambda}))^\perp.$
\item[(e)] The pair $(\Inj(\Lambda),\Inj(\Lambda))$ is $\W$-tilting.
\item[(f)] $(\Inj(\Lambda)^\vee,\GI(\Lambda))$ is a hereditary complete cotorsion pair in $\Modu\,(\Lambda).$ 
\item[(g)] $\Proj(\Lambda)^\wedge=\Inj(\Lambda)^\vee.$
\item[(h)] $\FPD(\Lambda)=\glGPD(\Lambda)=\id({}_\Lambda\Lambda)=\id(\Lambda_{\Lambda})=\glGID(\Lambda)=\FID(\Lambda)<\infty.$
\end{itemize}
\end{cor}
\begin{dem} Since $\Lambda$ an Artin $R$-algebra, it follows in particular that $\Lambda$ is an artinian ring and thus it is also left 
perfect, left noetherian and right coherent. Moreover, by \cite[Lemma 3.2.2]{A}, we know that $D(\Lambda_{\Lambda})$ is an injective 
cogenerator in $\Modu(\Lambda).$  Therefore, in order to prove the result,  by  Theorem \ref{Cotil4}, it is enough to show that 
$\pd(D(\Lambda_{\Lambda}))=\id(\Lambda_{\Lambda}).$
\

Note, firstly, that $\modu(\Lambda)$ is an abelian category with enough projectives and injectives,  $\proj(\Lambda):=\Proj(\modu(\Lambda))=\Proj(\Lambda)\cap\modu(\Lambda),$  $\inj(\Lambda):=\Inj(\modu(\Lambda))=\Inj(\Lambda)\cap\modu(\Lambda)$ and $\Lambda^{op}$ is an Artin $R$-algebra. Since $\Modu(\Lambda)$ has projective covers and the projective cover of a finitely generated left $\Lambda$-module is 
finitely generated, we have that 
\begin{center} $\pd(D(\Lambda_{\Lambda}))=\resdim_{\Proj(\Lambda)}(D(\Lambda_{\Lambda}))=\resdim_{\proj(\Lambda)}(D(\Lambda_{\Lambda})).$\end{center}
 Moreover, by using the duality $D:\modu\,(\Lambda)\to \modu\,(\Lambda^{op}),$ 
it follows that 
\begin{center}$\resdim_{\proj(\Lambda)}(D(\Lambda_{\Lambda}))=\coresdim_{\inj(\Lambda^{op})}(\Lambda_{\Lambda}).$\end{center} 
Since 
$\modu(\Lambda^{op})$ has enough injectives and $\Modu(\Lambda^{op})$ has injective envelopes, it follows that the injective 
 envelope of a finitely generated right $\Lambda$-module is finitely generated. Therefore 
 $\coresdim_{\inj(\Lambda^{op})}(\Lambda_{\Lambda})=\coresdim_{\Inj(\Lambda^{op})}(\Lambda_{\Lambda})=\id(\Lambda_{\Lambda});$ proving that $\pd(D(\Lambda_{\Lambda}))=\id(\Lambda_{\Lambda}).$
\end{dem} 

\begin{pro}\label{Cotil6} Let $\A$ be an AB4*-abelian category with injective cogenerators and enough projectives, and let $(\X,\Y)$ be 
a hereditary complete cotorsion pair in $\A.$ Let $Q$ be an injective cogenerator in $\A$ and $0\to Y_0\to M\to Q\to 0$ be an exact 
sequence with $Y_0\in\Y$ and $M\in\omega:=\X\cap\Y.$ If ${}^\perp M\subseteq\X$ and $\omega$ is closed under products, then  
 $\omega=\Prod(M),$ $\id(M)=\id(\Y),$ $\Y=\omega^\wedge$ and $W\GP_\omega=\X={}^\perp M.$
\end{pro} 
\begin{dem} Let ${}^\perp M\subseteq\X$ and $\omega:=\X\cap\Y$ be closed under products. By \cite[Theorem 3.12]{BMPS}, we 
have $\Y=\omega^\wedge.$ In particular, by the dual of Lemma \ref{AB1}, $\id(\omega)=\id(\Y).$ Furthermore, by the dual of Lemma \ref{AB1}, we get that ${}^\perp\omega={}^\perp(\omega^\wedge).$ Therefore $\X={}^\perp\Y={}^\perp\omega.$ 
\

Note that $\Prod(M)$ is always a preenveloping class. Moreover, we assert that the pair $(\Prod(M),\Prod(M))$ is $\W$-cotilting. Indeed, 
 by Remark \ref{RKGdebilCotil}, it is enough to show that any object in ${}^\perp M={}^\perp\Prod(M)$ can be embedded, through a 
 monomorphism, into some  object of $\Prod(M).$ Let $Z\in {}^\perp M.$ Then, there is a monomorphism $\alpha:Z\to Q^I,$ for some set $I.$ By Remark \ref{B1}, we get the exact sequence $0\to Y_0^I\to M^I\to Q^I\to 0.$ Note that $Y_0^I\in\Y,$ since $\Y$ is closed 
 under products. Consider the following pull-back diagram 
\[\xymatrix{0\ar[r] & Y_0^I\ar[r]\ar@{=}[d] & E\ar[r] \ar[d] & Z\ar[r] \ar[d]^{\alpha}& 0\\ 
0\ar[r] & Y_0^I\ar[r] & M^I\ar[r] & Q^I\ar[r] & 0.}\] 
Since   ${}^\perp M\subseteq\X,$ we have $\Ext^i_\A({}^\perp M,\Y)=0$ for all $i\geq 1,$ and thus the first row in the above diagram splits. Hence $\alpha$ factorizes through $M^I\to Q^I$ and so we get a monomorphism $Z\to M^I;$ proving our assertion. 
\

Since $(\Prod(M),\Prod(M))$ is a $\W$-cotilting pair, it follows from Theorem \ref{debildimfin} (a) that $W\GP_{\Prod(M)}={}^\perp M.$ 
 Using now that $\omega$ is closed under products, we have $\Prod(M)\subseteq \omega\subseteq\Y$ and hence $\omega\subseteq {}^\perp\omega\subseteq {}^\perp M=W\GP_{\Prod(M)}$ and $\X={}^\perp\Y\subseteq {}^\perp M\subseteq \X.$ Therefore 
 $\omega\subseteq W\GP_{\Prod(M)}$ and $\X={}^\perp M.$ 
 \
 
 Now, we prove that $\omega\subseteq \Prod(M).$ Let $W\in\omega\subseteq W\GP_{\Prod(M)}.$ Then, there is an exact sequence 
 $0\to W\to M_0\to K_0\to 0$ with $M_0\in\Prod(M)$ and $K_0\in {}^\perp M.$ Since 
 $\Ext^1_\A({}^\perp M,\omega)=\Ext^1_\A(\X,\omega)=0,$ the above exact sequence splits and hence $W\in\Prod(M);$ proving that 
 $\omega=\Prod(M).$ Therefore $\id(M)=\id(\Prod(M))=\id(\omega)=\id(\Y).$
\end{dem}

\bigskip

{\bf Acknowledgments} We thank professor Marco A. P\'erez by some suggestions and nice discussions about this paper.

\footnotesize

\vskip3mm \noindent Victor Becerril\\ Instituto de Matem\'aticas\\ Universidad Nacional Aut\'onoma de M\'exico.\\ 
Circuito Exterior, Ciudad Universitaria\\
C.P. 04510, M\'exico, D.F. MEXICO.\\ {\tt mathvick06@gmail.com}

\vskip3mm \noindent Octavio Mendoza Hern\'andez\\ Instituto de Matem\'aticas\\ Universidad Nacional Aut\'onoma de M\'exico.\\ 
Circuito Exterior, Ciudad Universitaria\\
C.P. 04510, M\'exico, D.F. MEXICO.\\ {\tt omendoza@matem.unam.mx}

\vskip3mm \noindent Valente Santiago\\
Departamento de Matem\'aticas, Facultad de Ciencias,\\
Universidad Nacional Aut\'onoma de M\'exico,\\
Circuito Exterior, Ciudad Universitaria,\\
M\'exico D.F. 04510, M\'EXICO.\\{\tt valente.santiago.v@gmail.com}

\end{document}